%% file: main.tex
\documentclass{TPMod}
\include{preamble.tex}

\begin{document}

\title{Lagrangian cobordisms and $K$-theory of symplectic bielliptic surfaces}
\author{Álvaro Muñiz-Brea}

\begin{abstract}
    \textsc{Abstract:} 
    We consider a family of closed symplectic manifolds $4$-manifolds which we call symplectic bielliptic surfaces and study its Lagrangian cobordism group of weakly-exact Lagrangian $G$-branes (that is, Lagrangians equipped with a grading, a Pin structure and a $G$-local system); relations come from  Lagrangian cobordisms satisfying a tautologically unobstructedness-type condition, also equipped with $G$-brane structures.
    Our first theorem computes its subgroup generated by tropical Lagrangians. 
    When $G$ is the unitary group of the Novikov field, we use homological mirror symmetry to compute the Grothendieck group of the Fukaya category and show it agrees with our computation for the cobordism group.
    This leads us to conjecture that tropical Lagrangians generate the whole cobordism group.
\end{abstract}

\maketitle


\section{Introduction}

    \subsection{Lagrangian cobordisms and the Grothendieck group}\label{sec:introcobandk0}
One of the main goals of symplectic topology is to understand and classify the Lagrangian submanifolds of a given symplectic manifold.
In its most classical form, this problem asks for a classification of Lagrangians up to Hamiltonian isotopy; although easy to state, this problem is remarkably hard---barely any answers are known outside real dimension two.
In an attempt to define a coarser relation between Lagrangians that would be amenable to classification, Arnold introduced in a series of two papers \cite{arnol1980lagrange,arnol1980lagrange2} the notion of a \emph{Lagrangian cobordism}.
A Lagrangian cobordism between Lagrangians $L_-,L_+ \subset X$ is a properly embedded Lagrangian submanifold $V \subset \C \x X$ such that, outside some compact set $K \subset \C$, one has 
$$
V \setminus \pi\inv_\C(K) = L_- \x \R_{<-a_-} \bigsqcup L_+ \x \R_{>a_+}
$$
for some values $a_\pm \in \R_{>0}$. 
(It is helpful to think of $V$ via its projection to $\C$, which is depicted in \Cref{fig:cobprojection}.)
When such a cobordism exists, one says that the Lagrangians $L_-$ and $L_+$ are Lagrangian cobordant.
This definition extends straightforwardly to Lagrangian cobordisms between Lagrangian tuples, where the cobordism must now fiber over several horizontal rays outside some compact set.
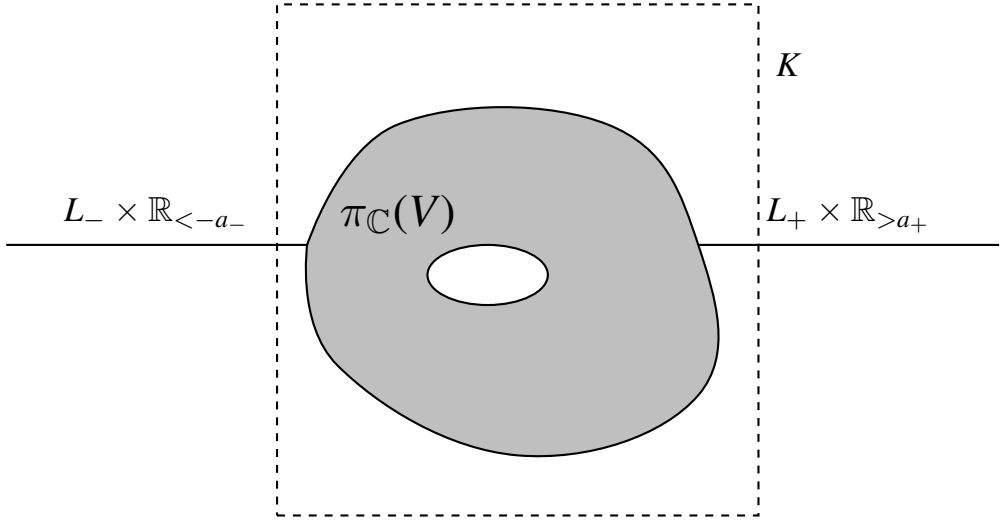
\begin{figure}
    \centering
    
    \begin{tikzpicture}[scale = 0.8, transform shape, thick]

\filldraw[color = black, fill = lightgray]  plot[smooth, tension=.7] coordinates {(-3,0) (-1.5,2) (2,2) (3.5,0) (3.5,-2.5) (0.5,-3.5) (-2.5,-2) (-3,0)};
        \filldraw[color = black, fill = white]  (0,-0.5) ellipse (1 and 0.5);
        
\draw (-8,0) -- (-3,0);
        
\draw (3.5,0) -- (8.5,0);
        
\draw[dashed]  (-3.5,4) rectangle (4.5,-4.5);
        
\node[scale = 1.5] at (-5.5,0.5) {$L_- \times \mathbb{R}_{<-a_-}$};
        \node[scale = 1.5]  at (6,0.5) {$L_+  \times \mathbb{R}_{>a_+}$};
        \node[scale = 2]  at (-1.5,0.5) {$\pi_{\mathbb{C}}(V)$};
        \node[scale = 1.5] at (5,3) {$K$};

    \end{tikzpicture}
    
    \caption{The shaded region (including the two horizontal lines) depicts a typical projection of a two-ended Lagrangian cobordism $V \subset X \x \C$ to $\C$. We have included (in dashed lines) an example of a compact region $K \subset\C$ outside which $V$ is product type: in $\C \setminus K$ the projection looks like two straight lines, and living over them in $X \x \C$ we have the cylidrical Lagrangians $L_- \x \R_{<-a_-}$ and $L_+ \x \R_{>a_+}$.}
    \label{fig:cobprojection}
\end{figure}

\begin{rmk}
    There is a different but related notion of Lagrangian cobordism between Legendrian submanifolds, where one instead tries to classify Legendrian submanifolds in a contact manifold up to  Lagrangian cobordism in the symplectisation (see e.g. \cite{chantraine2010lagrangian}). 
    While both definitions attempt to classify objects within a given manifold by looking at Lagrangian cobordisms in a bigger space, the invariants obtained are very different.
\end{rmk}

There are several algebraic structures that can be obtained from the notion of Lagrangian cobordism, and many of them have been connected to the Fukaya category.
For instance, Biran-Cornea work in the compact, monotone setting and use Lagrangian cobordisms to construct a `Lagrangian cobordism category' \cite[Section 7]{biran2013lagrangian}. 
In \cite{biran2014lagrangian} they show that Lagrangian cobordisms provide a \emph{functorial} way to decompose one end of a cobordism as an iterated mapping cone of the other ends, and using this they show there is a natural functor from their Lagrangian cobordism category to a category which encodes triangle decompositions in the derived Fukaya category.
With a different set-up, Nadler-Tanaka \cite{nadler2020stable} work in the non-compact, exact case and construct a stable $\infty$-category whose higher morphisms are `cobordisms between cobordisms', and they conjecture it to be equivalent to the (partially) wrapped Fukaya category. 
Further work of Tanaka \cite{tanaka2016fukaya} shows that there is indeed a functor from this stable $\infty$-category to (modules over) the wrapped Fukaya category, and that this functor is exact \cite{tanaka2016fukaya2} (in particular, exactness of the functor recovers the above-mentioned result of Biran-Cornea).

In this paper, we focus on a linear algebra invariant going back to Arnold, the {\it Lagrangian cobordism group}.
In its most naive form, it is an abelian group freely generated by Lagrangians and whose relations come from Lagrangian cobordisms:
\begin{defn}\label{def:lagcobgroup}
    Let $\mcal L$ be a collection of Lagrangians in $X$ and $\mcal L_{cob}$ a collection of cobordisms in $X \x \C$ whose ends are in $\mcal L$. 
    We define the {\it Lagrangian cobordism group} as the quotient
    $$
    \Cob(X) \equiv \Cob (X;\mcal L,\mcal L_{cob}) := \Z^{\mcal L} / \sim,
    $$
    where $\sim$ is the equivalence relation on $\Z^{\mcal L}$ generated by expressions of the form 
    $$
    L_1^+ +\dots+L^+_k \sim L_1^-+\dots+L_s^-
    $$
    whenever there exists a cobordism in $\mcal L_{cob}$ between tuples  $(L_1^+,\dots,L_k^+)$ and $(L_1^-,\dots,L_s^-)$. 
\end{defn}
\begin{rmk}\label{rmk:restrictionofdata}
    Often the Lagrangians  $\mcal L$ under consideration are not plain Lagrangians but come equipped with some extra structure (for instance, one can take $\mcal L$ to be the collection of orientable Lagrangians \emph{plus} a choice of orientation, or impose conditions such as exactness or monotonicity).
    In this case, it makes sense to ask the cobordisms in $\mcal L_{cob}$ to carry the same extra structure, and this data should be compatible with restriction (whenever restriction of data makes sense).
    A Lagrangian cobordism $V$ equipped with some extra data $\mcal D$ will then induce a relation 
    $$
    (L_1^+,\mcal D\restr{L_1^+}) + \dots + (L_k^+,\mcal D\restr{L_k^+}) \sim (L_1^-,\mcal D\restr{L_1^-}) + \dots + (L_s^-,\mcal D\restr{L_s^-})
    $$
    in the cobordism group.
\end{rmk}
Arnold himself performed the first computations of these groups, showing that in $X = \R^2$ oriented immersed Lagrangians are classified up to oriented cobordism by the Maslov index and the area they enclose.
He also computed the cobordism group of $T^*S^1$, where Lagrangians are classified essentially by the same data (their homology class and the area they enclose).
Soon after Eliashberg \cite{eliashberg1984cobordisme} argued that immersed, exact Lagrangian cobordism groups are isomorphic to homotopy groups of certain Thom spaces, and Audin \cite{audin1985quelques} used this to show that these Lagrangian cobordism groups are topological.

Nonetheless, in recent years it has been discovered that this flexibility for Lagrangian cobordism groups of immersed exact Lagrangians disappears when one equips the Lagrangians and/or cobordism with various decorations.
In a series of two papers \cite{biran2013lagrangian,biran2014lagrangian} Biran-Cornea brought Lagrangian cobordisms back to the attention of the symplectic topology community by exhibiting a relation between Lagrangian cobordisms and cone decompositions in the Fukaya category. Whereas Eliashberg and Audin had shown that the most naive notion of a Lagrangian cobordism reduces to algebraic topology, Biran-Cornea showed that if the Lagrangians in $\mcal L$ and $\mcal L_{cob}$ are required to be \emph{monotone} (with a fixed monotonicity constant) and satisfy a Gromov-Witten type condition, then Lagrangian cobordism is no longer a flexible notion.
One of their main results is the following: if $V$ is a monotone Lagrangian cobordism between a tuple $(L_1,\dots,L_k)$ and $L_0$, then $L_0$ is generated by $L_1,\dots,L_k$ in the derived Fukaya category $D^b\mcal Fuk(X)$ (cf. \cite[Theorem A]{biran2014lagrangian}).
It follows   that the Lagrangian cobordism group admits a map 
\begin{equation}\label{eq:BCmap}
    \Cob(X) \to K_0(D^b\mcal Fuk(X))
\end{equation}
to the Grothendieck group of the triangulated envelope of the Fukaya category (cf. \cite[Corollary 1.2.1]{biran2014lagrangian}).\footnote{Recall the Grothendieck group of a triangulated category is the free abelian group on its objects modulo relations $Z_1 - Z_2 + Z_3 = 0$ whenever there is a distinguished triangle $Z_1 \to Z_2 \to Z_3 \to Z_1[1]$.}
In further work, Biran-Cornea-Shelukhin \cite{biran2021lagrangian} take $\mcal L$ to be \emph{weakly-exact} Lagrangians (meaning $\omega(\pi_2(X,L)) = 0$) and $\mcal L_{cob}$ to be \emph{quasi-exact} cobordisms (meaning there exists a compatible almost complex structure such that they bound no holomorphic disks).
Under these assumptions, they extend the results in \cite{biran2014lagrangian} and construct a map $\Cob(X) \to K_0(D^b\mcal Fuk(X))$.
In this paper, we analyze a particular instance of the following question:
\begin{center}
    \emph{``is the map $\Cob(X) \to K_0(D^b\mcal Fuk(X))$ an isomorphism?''}
\end{center}
\begin{rmk}
By construction, the map (\ref{eq:BCmap}) is always surjective,  hence the question is whether it is injective.
The main result in \cite{biran2014lagrangian} shows that cobordisms induce cone decompositions in the Fukaya category.
Injectivity would mean that every triangle relation in the Fukaya category can be obtained from Lagrangian cobordisms.
\end{rmk}

To date, answers to the above question (in the compact case) have only appeared for $\dim X = 2$: \cite{haug2015lagrangian} computes the Lagrangian cobordism group of $T^2$, and shows the map (\ref{eq:BCmap}) is an isomorphism in this case; and \cite{perrier2019lagrangian,rathel2023unobstructed} extend the result to higher genus surfaces.
In this paper, we study the question for a closed symplectic $4$-manifold.
We work with a similar set-up to that in \cite{biran2021lagrangian}: the collection of Lagrangians $\mcal L$ will consist of weakly-exact Lagrangians equipped a grading, a Pin structure and a local system, and the collection $\mcal L_{cob}$ will consist of Lagrangian cobordisms carrying the same decorations and satisfying a quasi-exactness type condition (see \Cref{def:cobgroupKk} for the precise definition of the cobordism group, and \Cref{sec:branes} for an exposition of the extra data that we put on our Lagrangians).
A minor adaptation of the results in \cite{biran2021lagrangian} (together with \cite{haug2015lagrangian}, who works out the inclusion of the above decorations) ensures there is a well-defined map $\Cob(X) \to K_0(D^b\mcal Fuk(X))$. We will take $X = \mcal K$ to be a symplectic bielliptic surface as in \Cref{sec:biellipticsurfaces} and we will consider the subgroup $\Cob^{trop}(\mcal K) \subset \Cob(X)$ generated by Lagrangian sections and fibers.
We present an explicit computation of $\Cob^{trop}(\mcal K)$ in \Cref{th:cobK}, and show that the composition $\Cob^{trop}(\mcal K) \into \Cob(\mcal K) \to K_0(D^b\mcal Fuk(X))$ is an isomorphism in \Cref{th:isoK}.
\begin{rmk}
    With a slightly different set-up, Lagrangian cobordism between non-compact Lagrangians in exact symplectic manifolds---where the definition of a cobordism is suitably modified---have also been considered in the literature \cite{nadler2020stable,bosshard2021lagrangian}.
    As far as cobordism groups in this setting, \cite{bosshard2021lagrangian} defines a corresponding version of the Lagrangian cobordism group and constructs an analog of the map (\ref{eq:BCmap}) (where the Fukaya category is replaced by the \emph{wrapped} Fukaya category). 
    He also shows that such a map is an isomorphism for punctured Riemann surfaces.
    In later work, he extends his results to arbitrary Weinstein sectors, showing the Lagrangian cobordism group is isomorphic to the middle relative homology of the Weinstein sector  \cite{bosshard2023note}.
    A consequence of his computation is that the map (\ref{eq:BCmap}) is \emph{not} an isomorphism for flexible Weinstein manifolds.
    We remark however that his definition of a Lagrangian cobordism is different from ours (due to the non-compact nature of Lagrangians in wrapped Fukaya categories).
\end{rmk}

\subsection{Main results}\label{sec:results}
Consider the symplectic manifold $T^*\R^2$ with coordinates $(x_1,y_1,x_2,y_2)$ and the standard symplectic structure $dx_1\wdg dy_1 + dx_2 \wdg dy_2$.
Let $\mcal K$ be the quotient of $T^*\R^2$ by the action of $\Z^4$ as well as the relation 
$$
(x_1,y_1,x_2,y_2) \sim (x_1 + 1/2, y_1, -x_2,-y_2).
$$
Note that all actions commute with the symplectic form, hence $\mcal K$ inherits a symplectic structure from $T^*\R^2$.
We call $\mcal K$ a \emph{symplectic bielliptic surface} (cf. \Cref{def:biellipticsymplectic}).

The projection to the $(x_1,x_2)$ coordinates gives a Lagrangian torus fibration $\pi: \mcal K \to K$ over a tropical Klein bottle $K$.
Let $\det_\Z K = \wedge^2 T_\Z K$ be the integral orientation bundle of $K$, and denote by $\xi := \pi^*\det_\Z K$ its pullback to $\mcal K$.
Given an abelian group $G$, we consider a cobordism group $\Cob(\mcal K)$ whose generators are weakly-exact Lagrangian $G$-branes and whose relations come from tautologically unobstructed Lagrangian cobordisms equipped a $G$-brane structure (see \Cref{sec:branes} and \Cref{def:cobgroupKk} for precise definitions).
Our first result is the following:
\begin{maintheorem}[$=$\Cref{th:cobK}]\label{th:theoremA}
    Let $\mcal K$ be a symplectic bielliptic surface and $\pi:\mcal K \to K$ its Lagrangian torus fibration over a tropical Klein bottle.
    Let $\Cob^{trop}(\mcal K) \subset \Cob(\mcal K)$ be the subgroup generated by Lagrangian fibers and Lagrangian sections.
    Then there is an isomorphism
    $$
    \Cob^{trop}(\mcal K) \cong H_2(\mcal K;\xi) \oplus G_{(2)} \oplus (S^1 \oplus G) \oplus (S^1 \oplus G)
    $$
    for $G_{(2)} \subset G$ the subgroup of $2$-torsion elements.
\end{maintheorem}
\Cref{th:theoremA} shows that $\Cob^{trop}(\mcal K)$ is {\it finite-dimensional} in the sense of \Cref{def:infinitedimensionality}---in fact, one can see directly from the definition that it is a $2$-dimensional cobordism group (roughly speaking, this means that $2$ is the minimum $d$ such that  $\Cob^{trop}(\mcal K)$ can be parametrized by a  $d$-dimensional family of Lagrangians).
We compare this with the work of Sheridan-Smith, who show that a symplectic 4-manifold admitting a Lagrangian of genus at least 1 has an infinite-dimensional cobordism group \cite[Lemma 7.9]{sheridan2020rational}.
In fact, we are able to extend the result of Sheridan-Smith to arbitrary dimension and prove the following (see \Cref{ap:Roitmantheorem}):
\begin{prop}\label{prop:symplecticRoitman}
    Let $X$ be a closed symplectic manifold of dimension at least 4 with trivial canonical bundle.
    If $X$ contains a Lagrangian torus of Maslov index zero, then the cobordism group of graded Lagrangians $\Cob(X)$ is infinite-dimensional.
\end{prop}
The crucial difference in \Cref{th:theoremA} is that both Sheridan-Smith and \Cref{prop:symplecticRoitman} work in the Calabi-Yau case (i.e. the canonical bundle is trivial), whereas symplectic bielliptic surfaces have torsion but non-trivial canonical bundle.
Given a nowhere vanishing quadratic volume form $\eta^2 \in H^0((\Omega^n)^{\otimes 2})$, a graded Lagragian with respect to $\eta^2$ is naturally oriented with respect to the local system where $\pm \eta$ takes real values (see \cite[Remark 11.18]{seidel2008fukaya}).
When the canonical bundle is trivial the quadratic form $\eta^2$ admits a square root, hence the local system is  trivial and graded Lagrangians are naturally oriented in the usual sense---this is the setting of \cite{sheridan2020rational} and \Cref{prop:symplecticRoitman}.
However, in this paper the canonical bundle is non-trivial, hence Lagrangians are only oriented with respect to some non-trivial local system (this is the local system $\xi$ of \Cref{th:theoremA}).
Obviously, the result of Sheridan-Smith and \Cref{prop:symplecticRoitman} still holds for the oriented version of the cobordism group of a bielliptic surface. 
The remarkable fact is that introducing these new graded Lagrangian cobordisms---despite eliminating some too, as not every oriented Lagrangian is graded---creates sufficiently many new relations to get a finite-dimensional cobordism group.

As remarked above, the result in \Cref{th:theoremA} is, to the best of the author's knowledge, the first computation of a Lagrangian cobordism group  in dimension greater than two.
Although \Cref{th:theoremA} only gives a computation of a subgroup  of $\Cob(\mcal K)$ (the subgroup $\Cob^{trop}(\mcal K)$ generated by tropical Lagrangians), our second main result shows that this subgroup is indeed an interesting subgroup:

\begin{maintheorem}[$=$\Cref{th:isoK}]\label{th:theoremB}
    Fix $G = U_\Lambda$ to be the unitary group of the Novikov field.
    Then the composition
    $$
    \Cob^{trop}(\mcal K)  \into \Cob(\mcal K) \to K_0(D^b\mcal Fuk(X))
    $$
    is an isomorphism.
\end{maintheorem}
\begin{rmk}\label{rmk:dbcongdpi}
    To prove \Cref{th:theoremB} we use Abouzaid's homological mirror symmetry result \cite{abouzaid2021homological} to show that symplectic bielliptic surfaces are mirror to algebraic bielliptic surfaces as defined in \Cref{sec:biellipticsurfaces}.
We then show that the composition 
    $$
        \Cob^{trop}(\mcal K) \to K_0(D^b\mcal Fuk(\mcal K)) \into K_0(D^\pi\mcal Fuk(\mcal K)) \xrightarrow{\sim} K_0(D^bCoh(Y))
    $$
    is an isomorphism, which yields the result.
    Here, we denoted by $D^\pi\mcal Fuk(\mcal K)$ the split-closure of the Fukaya category. 
    In general, this is larger than $D^b\mcal Fuk(\mcal K)$, but it follows from our argument that the inclusion $K_0(D^b\mcal Fuk(\mcal K)) \into K_0(D^\pi\mcal Fuk(\mcal K))$ is an isomorphism.
    It is a theorem of Thomason \cite{thomason1997classification} that this isomorphism of Grothendieck groups is enough to conclude $D^b\mcal Fuk(\mcal K) \simeq D^\pi \mcal Fuk(\mcal K)$.
    In particular, for bielliptic surfaces there is no need to take the split-closure of the Fukaya category to obtain a homological mirror symmetry equivalence.

\end{rmk}

\Cref{th:theoremB} on its own---together with mirror symmetry considerations and the analogous statement for Chow groups---lead us to the following conjecture:
\begin{conj}
    The inclusion $\Cob^{trop}(\mcal K)  \into \Cob(\mcal K)$ is an isomorphism.
\end{conj}
The main difficulty to address this question is our lack of understanding of what are all the weakly-exact Lagrangians in $\mcal K$.

\subsection{Organization of the paper}
The paper is organized as follows.
In \Cref{sec:tropicalgeometry} we review the necessary background on tropical geometry, including the main definitions (\Cref{sec:tropicalmanifolds}), the construction of the Albanese variety and the Albanese morphism (\Cref{sec:albanese}) and the relationship between tropical and symplectic geometry (\Cref{sec:lagtorusfib}).
The tropical manifold we study is a tropical Klein bottle, and this is introduced in \Cref{sec:kleinbottles}.
\Cref{sec:lagrangiancobordisms} deals with the main object of this paper, Lagrangian cobordisms. 
We  present the main constructions of cobordisms in \Cref{sec:cobconstructions} and use them to compute the cobordism group of $T^2$ \Cref{sec:lagcobT2}---this will be useful for the main computations in \Cref{sec:computationcob}.
In the last part of this section, \Cref{sec:branes}, we  explain the extra data we consider on Lagrangians and cobordisms.
The main results of the paper are in \Cref{sec:computationcob}.
In \Cref{sec:biellipticsurfaces,sec:hmsbielliptic} we introduce symplectic and algebraic bielliptic surfaces and explain how they fit into Abouzaid's homological mirror symmetry result \cite{abouzaid2021homological} to conclude they are homologically mirror.
We then build up the necessary results to prove the main theorems in \Cref{sec:admissible,sec:cycleclassmap,sec:symplecticalbanese}.
The main theorems are in the next two sections: we compute the subgroup of the Lagrangian cobordism group generated by tropical Lagrangians in \Cref{sec:finitedimensionality} (see \Cref{th:cobK}) and show it is isomorphic to the Grothendieck group of the Fukaya category in \Cref{sec:isoK} (see \Cref{th:isoK}).
In the last part, \Cref{ap:compatibility}, we prove some compatibility results involving the brane structure we put on our Lagrangians and cobordisms---it is essential for our computations but can safely be skipped by the reader on a first read.
We also include four appendices at the end.
The first three prove (technical) results claimed in several parts of the paper, whose proofs we decided to separate from the main text to improve the readability of the paper and whose content might be of independent interest to the reader.
The last appendix contains the proof of \Cref{prop:symplecticRoitman}: the content of this Proposition ties directly to the main results of the paper, but its proof is a simple linear algebra result and does not provide any additional insights.

\paragraph{Acknowledgments} 
I would  like to thank my advisor Nick Sheridan for his guidance on this project. 
He has been extremely generous with his time, very patient and has provided invaluable insights.
The idea to use multivalued perturbations as in \Cref{ap:multivaluedFuk} is entirely due to him.
I would also like to thank Jeff Hicks for many helpful discussions. 
I owe to him much of my understanding of Lagrangian surgery.
I also thank Arend Bayer for a useful conversation that motivated the results in \Cref{ap:integralChern}.
This work was supported by an ERC Starting Grant (award number 850713-HMS).

\section{Tropical geometry}\label{sec:tropicalgeometry}

    In this section we give a review of the tropical geometry background that will be needed to understand some of the constructions in this paper. 
    The symplectic manifold we consider in this paper comes with a Lagrangian torus fibration over a tropical affine manifold.
    Moreover, all the Lagrangians we consider are \emph{tropical Lagrangians}: Lagrangians living over tropical subvarieties of the base.
    Understanding the tropical geometry of the base is hence essential for our computations.

    \subsection{Tropical affine manifolds}\label{sec:tropicalmanifolds}

Let $\Aff(\R^n):=\R^n \rtimes \GL(n,\Z)$ denote the group of affine transformations of $\R^n$ whose linear part is integral. 
A {\it tropical affine manifold of dimension $n$} is a smooth manifold together with an atlas of charts whose transition functions belong to $\Aff(\R^n)$.

\begin{example}\label{ex:vsaretropical}
A choice of basis on a real vector space gives it the structure of  a tropical affine manifold. 
In particular, $\R^n$ has a natural tropical affine structure.
\end{example}

Let $B_1,B_2$ be tropical affine manifolds of dimension $n_1,n_2$.
A smooth map $f:B_1\to B_2$ is said to be {\it integral affine} if, locally in some charts, it is an affine map $\R^{n_1} \to \R^{n_2}$ with integer slope.
We say it is a {\it tropical map} if, locally in some charts, it is a piecewise affine map with integer slope.

\begin{example}\label{ex:tropicalquotient}
Let $B$ be a tropical affine manifold and $G$ be a group acting freely and properly on $B$ by integral affine transformations. 
Then $B/G$ has a natural tropical affine structure making the projection a tropical submersion.
\end{example}

Let $B$ be a tropical affine manifold.
We say a smooth function $f: B \to \R$ is {\it tropical affine} if it is an integral affine map $B \to \R$, where we view $\R$ as a tropical affine manifold by equipping it with its vector space tropical affine structure as in \Cref{ex:vsaretropical}.
Similarly, a continuous function $f : B \to \R$ is said to be {\it tropical} if it is piecewise integral affine.
We denote by $\Aff_B \subset C^\infty_B$ the subsheaf of affine functions.
It naturally gives a lattice bundle of integer covectors 
$$
T^*_\Z B :=  d(\Aff_B) \subset T^*B,
$$
where $d: \Omega^0(B)  \to \Omega^1(B)$ is the usual differential. 
Dually, we get a lattice of vectors $T_\Z B \subset TB$ defined as
$$
T_\Z B := \set{
v\in TB \st \alpha(v) \in \Z \text{ for all }\alpha \in T_\Z^* B
}.
$$
Both $T_\Z B$ and $T^*_\Z B$ form local systems of lattices over $B$.
Global sections $\alpha \in H^0(T^*_\Z B)$ are called \emph{tropical $1$-forms}.

\begin{rmk}
    A tropical affine manifold determines a $\Z^n$-lattice $T_\Z B \subset TB$.
    However, not every $\Z^n$-lattice $\Lambda \subset TB$ equips $B$ with a tropical affine structure---the lattice must be integrable.
    Sometimes we will refer to a tropical affine manifold as the data $(B,\Lambda)$ of a smooth manifold together with a $\Z^n$-lattice $\Lambda \subset TB$: it will always be implicit that $\Lambda$ is integrable. 
\end{rmk}

Given a tropical affine manifold $B$, a {\it tropical affine submanifold} is a submanifold $V \subset B$ such that $TV = L\otimes \R$ for some sublattice $L \subset T_{\Z} B\restr{V}$. 
We remark that $L$ is not part of the data.
\begin{example}
    Tropical affine submanifolds of $(\R^n,\Z^n)$ are precisely linear subspaces whose defining equations have integral coefficients.
\end{example}     \subsection{The Albanese variety}\label{sec:albanese}
Associated to any tropical affine manifold $B$ there is a {\it tropical Albanese variety} $\Alb(B)$ constructed as follows.
Given an integral $1$-form $\alpha \in H^0(T^*_\Z B)$ and a singular $1$-chain $\gamma \in C_1(B)$, we can regard $\alpha$ as a usual differential form and integrate $\alpha$ over $\gamma$. 
Since integral forms are closed (they are locally exact), the integration pairing $\int: H^0(T^*_\Z B) \otimes C_1(B) \to \R$ descends to a map
\begin{equation*}\label{eq:pairing}
    \int: H^0(T_\Z^* B) \otimes H_1(B;\Z) \to \R
\end{equation*}
over classes in the first homology of $B$. 
In particular, there is a map $H_1(B;\Z) \to \Hom(H^0(T_\Z^* B),\R)$, and we define the {\it tropical Albanese variety} of $B$ to be the quotient
\begin{equation}\label{eq:albanesevariety}
    \Alb(B) := \frac{\Hom(H^0(T_\Z^* B),\R)}{H_1(B;\Z)}.
\end{equation}

\begin{rmk}
    Associated to any smooth algebraic variety there is an (algebraic) Albanese variety, which is an abelian variety satisfying a universal property.
    Whereas this algebraic object is always smooth, the tropical  Albanese variety associated to a tropical affine manifold need not even be a manifold.
    However, in all the examples we will consider in this paper it will not just be a manifold but moreover it will come equipped with a natural structure of a tropical affine torus.
\end{rmk}

Denote by $Z_0(B)$ the homologically trivial (i.e. degree zero) $0$-cycles on $B$. 
There is a natural map
\begin{align}\label{eq:albz0}
    \begin{split}
        \alb:Z_0(B) & \to \Alb(B) \\
        \alb(b^+ - b^-) & = \int_\gamma-: H^0(T_\Z^*B) \to \R,
    \end{split}
\end{align}
where $\gamma \in C_1(B)$ is any $1$-chain with $\partial\gamma = b^+ - b^-$ (the quotient in \Cref{eq:albanesevariety} makes the map independent of the choice of $\gamma$).
In particular, after choosing a base-point $b_0 \in B$, composition of the above map with the inclusion $b \in B \mapsto b - b_0 \in Z_0(B)$ gives an {\it albanese map} (still denoted by $\alb$)
\begin{equation*}\label{eq:alb}
    \alb:B \to \Alb(B) 
\end{equation*}
It is an isomorphism whenever $B$ is a tropical affine torus (see \Cref{def:tropicalaffinetorus}).

     \subsection{From tropical  to symplectic geometry}\label{sec:lagtorusfib}
Recall that there is a map from smooth manifolds to symplectic manifolds that assigns to each smooth manifold $B$ its cotangent bundle $T^*B$.
Moreover, the canonical projection $T^*B \to B$ is a Lagrangian fibration (i.e. the fibers are Lagrangian).
Note that $T^*B$ is never compact.

When $B$ has the structure of a compact tropical affine manifold, we can refine this construction to give a compact symplectic manifold: the quotient 
$$
X(B) := T^*B / T^*_\Z B
$$
inherits a symplectic structure from  $T^*B$.
The projection $X(B) \to B$ is now a Lagrangian torus fibration.
There is a rich duality between the tropical geometry of $B$ and the symplectic geometry of $X(B)$, which we now briefly recall.

First of all, $X(B)$ possesses a broader source of Lagrangian sections.
In the cotangent bundle $T^*B$ one can obtain Lagrangians as graphs of closed $1$-forms---that is, sections of the sheaf $\Omega^1_c$.
In $X(B)$, the sheaf $\Omega^1_c \cong C^\infty / \R$ is replaced by $C^\infty / \Aff$, as affine functions have vanishing differential in the quotient $X(B) = T^*B/ T_\Z^*B$.
\begin{example}
    The function $f: \R \to \R, f(x) = \frac{1}{2}x^2$ does not descend to a section of $C^\infty_{\R/\Z}/\R$, but it does give a section of $C^\infty_{\R/\Z}/\Aff$.
    The corresponding Lagrangian in $X(\R/\Z) = T^2$ is a geodesic of homology class $(1,1)$.
\end{example}

Second, in the cotangent bundle one can produce Lagrangians from submanifolds $ V \subset B$ by taking the conormal bundle $N^*V \subset T^*B$.
On the other hand, some form of rigidity appears in $X(B)$ due to the compactness of the fibers.
Without further assumptions, the projection to $X(B)$ of the conormal lift of an arbitrary submanifold need not be a (compact) embedded Lagrangian.
Instead, under the assumption that $V \subset B$ is a \emph{tropical} submanifold (see \Cref{sec:tropicalmanifolds}), the quotient
$$
L_V := N^*V / N^*_\Z V
$$
is a compact Lagrangian submanifold of $X(B)$.
We call $L_V$ a {\it tropical Lagrangian} and say it is a \emph{tropical lift of $V$}.
\begin{rmk}
   The setting in which one can obtain a Lagrangian from a subset (not even a submanifold) of the base is an active topic of research.
   Lagrangian lifts of subsets with much greater generality than what is presented here---so-called \emph{tropical subvarieties}---have been constructed; see e.g. \cite{mikhalkin2019examples,hicks2020tropical,mak2020tropically,matessi2021lagrangian, sheridan2021lagrangian}.
\end{rmk}

     \subsection{Tropical Klein bottles}\label{sec:kleinbottles}

Let us now introduce the tropical affine manifold that we will study in this paper.
Consider $(\R^2,\Z^2)$ with the standard tropical affine structure, and let $\Lambda\subset\R^2$ be a lattice.
Recall that  the quotient of a tropical affine manifold by a group of tropical automorphisms inherits a tropical affine structure (see \Cref{ex:tropicalquotient}).
\begin{defn}\label{def:tropicalaffinetorus} 
    We call the quotient 
    \begin{equation}\label{eq:tropicalaffinetorus}
       T^2_\Lambda:=\R^2/\Lambda 
    \end{equation}
    a {\it tropical affine torus}.
\end{defn}

\begin{defn}\label{def:kleinbottle}
    Let $T^2_\Lambda$ be a tropical affine torus and $\psi \in \Aff(T^2_\Lambda)$ an orientation reversing, affine involution without fixed points.
    We call the quotient
    \begin{equation}\label{eq:tropicalKleinbottle}
        K_{\Lambda,\psi} := T^2_\Lambda / \vspan{\psi}
    \end{equation}
     a {\it tropical Klein bottle}.
\end{defn}

\begin{rmk}\label{rem:exoticKlein}\label{rem:exotictori}
    There exist tropical affine manifolds with the topology of a torus which do not arise as a quotient of the form in \Cref{def:tropicalaffinetorus} (see \cite[Section 4, Theorem A]{mishachev1996classification}). 
    Our definition only covers those tropical affine tori whose associated symplectic manifold $X(T^2_\Lambda) = T^* T^2_\Lambda / T_\Z^* T^2_\Lambda$, as defined in \Cref{sec:lagtorusfib}, is a symplectic 4-torus. 
    It follows from this that there exist tropical affine manifolds with the topology of a Klein bottle that are not of the form (\ref{eq:tropicalKleinbottle}).
    Again, our definition covers those tropical affine Klein bottles whose associated symplectic manifold is a quotient of a 4-torus
\end{rmk}

Tropical Klein bottles have been classified by Sepe \cite{sepe2010classification}. 
In fact, they classify tropical affine structures on all types of Klein bottles, not just the restricted types we consider in this paper (cf. \Cref{rem:exoticKlein}).
Let us recall the part of their result that relates to those Klein bottles in \Cref{def:kleinbottle}, which we phrase in an equivalent way that is more convenient for our purposes:

\begin{lem}[\cite{sepe2010classification}]\label{lem:classificationklein}
    Let $L\subset \R^2$ be a lattice and denote by $(\R^2,L)$ the smooth manifold $\R^2$ together with the tropical affine structure defined by $L$. 
    Every tropical Klein bottle is isomorphic to a unique `standard' Klein bottle, i.e. one of the form $K_L:=((\R^2,L)/\Z^2)/\Z_2$, where $\Z_2$ acts as above with
    \begin{equation}\label{eq:standardKlein}
        A=\begin{pmatrix}
        1 & 0\\
        0 & -1
        \end{pmatrix},
        \quad
        c=(1/2,0).
    \end{equation}
\end{lem}

\begin{rmk}\label{rmk:Kleinfamilies}
    Only those integral affine structures $L$ that are preserved by this standard $A$  give well-defined tropical Klein bottles (cf. \Cref{ex:tropicalquotient}). 
    In other words, the map 
    $$
    \left\lbrace\begin{array}{cc}\text{isomorphism classes of}\\ \text{tropical Klein bottles} \end{array}\right\rbrace \to \{L\subset\R^2 \text{ lattice}\}
    $$
    is an injection but not a surjection, and its image is precisely those lattices that are symmetric with respect to the involution $(x,y)\mapsto(x,-y)$. 
    These lattices are easily classified: choosing $l_1,l_2\in L$ primitive lattice points lying in both axis, one can easily check using the $A$-invariance of $L$ that it has to be of the form $L_1 = \Z\vspan{l_1,l_2}$ or $L_2 = \Z \vspan{l_1,1/2(l_1+l_2)}$.
    The corresponding Klein bottles will be denoted by $K_1$ and $K_2$, and correspond to the families in equations (2) and (4) in \cite[Theorem 3.1]{sepe2010classification}.
\end{rmk}

\begin{rmk}\label{rmk:kleintypeIandII}
        The two types of Klein bottles obtained in the above remark give rise to complex surfaces $TK_i / T_\Z K_i$ corresponding to bielliptic surfaces of type I and II in the classification of complex bielliptic surfaces (see for instance \cite[List VI.20]{beauville1996complex} for such classification, originally due to Bagnera-de Franchis).
        Although such classification contains a total of seven types, only the first two admit maximal degenerations and thus are relevant from the point of view of mirror symmetry.
\end{rmk}

\begin{rmk}
It follows from \Cref{lem:classificationklein} and \Cref{rmk:Kleinfamilies} that tori $T^2_\Lambda$ admitting a free $\Z_2$ orientation reversing integral affine action  are `rational', in the sense that there is a basis of $\Lambda$ with rational slope (although not necessarily rational coordinates).
\end{rmk}

\section{Lagrangian cobordisms}\label{sec:lagrangiancobordisms}
    Recall from \Cref{sec:introcobandk0} that a {\it (planar) Lagrangian cobordism} from $(L_1^+,\dots,L_k^+)$ to $(L_1^-,\dots,L_s^-)$ is a properly embedded Lagrangian submanifold $V\subset X\x \C$ such that
    $$
    V\setminus \pi\inv_\C(K)=\left(\bigsqcup_{i=1}^k L^+_i \x \gamma_i^+\right) \bigsqcup \left(\bigsqcup_{j=1}^s L^-_j \x \gamma_j^-\right),
    $$
    for some compact region $K\subset \C$ and horizontal curves $\gamma^\pm_j(t) =(\pm t, j), \pm t >  a_j^\pm$.
    Given collections $\mcal L$ and $\mcal L_{cob}$ of Lagrangians in $X$ and Lagrangian cobordisms in $X\x \C$ (all of whose ends are in $\mcal L$), one defines the {\it Lagrangian cobordism group} as the free abelian group on $\mcal L$ modulo the relations generated by $\mcal L_{cob}$ (see \Cref{def:lagcobgroup}).
    More generally, the collection of Lagrangians and cobordisms can be equipped with various extra decorations (e.g. orientations) or required to satisfy some property (e.g. monotonicity), and in this case the cobordism group is defined as the free abelian group on the decorated Lagrangians with the desired property modulo the relations generated by the decorated cobordisms, where the decorations must restrict appropriately (cf. \Cref{rmk:restrictionofdata}).
    
In this section we  review the main constructions of Lagrangian cobordisms and detail the decorations we will put on our Lagrangians and cobordisms.
    In \Cref{sec:cobconstructions} we present the Lagrangian cobordism associated to a Hamiltonian isotopy (so-called suspension cobordism) and several types of surgeries.
    We use these to compute the Lagrangian cobordism group of $T^2$ in \Cref{sec:lagcobT2}, whose cobordism relations will be used in \Cref{sec:computationcob}.
    \Cref{sec:branes} recalls the definition of a Pin structure, a grading and a local system---these form the decorations that we will put on our Lagrangians in \Cref{sec:computationcob}.

\subsection{Cobordism constructions}\label{sec:cobconstructions}
There are essentially two main sources of Lagrangian cobordisms: Hamiltonian isotopies and Lagrangian surgery.\footnote{A third construction due to Sheridan-Smith constructs cobordisms from tropical curves \cite{sheridan2021lagrangian}. However, this construction produces \emph{cylindrical} cobordisms--- Lagrangian submanifolds in $X \x \C^*$ that fiber over rays near $0$ and $\infty$. Every cobordism in $ X \x \C$ gives a cylindrical cobordism by quotienting $\C$ by a large enough imaginary translation, but not every cylindrical cobordism gives a cobordism in $X \x \C$.}
We briefly recall these.
\begin{lem}[\cite{audin1994symplectic}]
If $L_0$ and $L_1$ are Hamiltonian isotopic, then there is a Lagrangian cobordism between  $L_0$ and $L_1$.
\end{lem}

The second main cobordism source is that of Lagrangian surgery.
Lagrangian surgery produces, from two (embedded) Lagrangians $L_0,L_1$ intersecting transversely at some point $p$, a third (embedded) Lagrangian $L_0 \# L_1$ which agrees with $L_0$ and $L_1$ outside an arbitrarily small neighborhood of $p$.
The Lagrangians $L_0,L_1$ and $L_0\#L_1$ are furthermore related by a Lagrangian cobordism.
We now recall the construction of $L_0 \# L_1$, which is classical and goes back to Lalonde-Sikorav \cite{lalonde1991sous} in the $2$-dimensional case and to Polterovich \cite{polterovich1991surgery} in arbitrary dimension.

Choose a Darboux chart $\varphi:U\to \C^n$ around $p\in U$ such that $\varphi(L_0)=\R^n$ and $\varphi(L_1)=(i\R)^n$. Consider a curve $\gamma:\R\to \C$ such that 
\begin{itemize}
    \item $\gamma(t)=t$ for $t\in(-\infty,-1]$;
    \item $\gamma(t)=it$ for $t\in[1,\infty)$;
    \item $\gamma'(t)\in\R_+\x i\R_+\subset \C$ for $t\in (-1,1)$.
\end{itemize}
We define the {\it surgery neck} $N$ to be
$$
N:=\gamma\cdot S^{n-1}=\{(\gamma(t)x_1,\dots,\gamma(t)x_{n})\in\C^n\st \sum_{i=1}^{n}x_i^2=1\}
$$
and the surgery of $L_0$ and $L_1$ to be $L_1\# L_2:=[(L_1\cup L_2)\setminus U]\cup \varphi\inv (N)$.
Biran-Cornea show that the Lagrangians $L_i$ and their surgery are related by a Lagrangian cobordism:
\begin{prop}\label{pr:surgery}
Let $L_0,L_1\subset X$ be Lagrangians intersecting transversely at a single point $p$. 
There is a cobordism with ends $L_0,L_1$ and $L_0\# L_1$. 
\end{prop} 
\begin{proof}
See \cite[Lemma 6.1.1]{biran2013lagrangian}.
\end{proof}

Since the surgery construction is local, one immediately gets:
\begin{cor}\label{cor:multiplesurgery}
Let $L_0$ and $L_1$ be Lagrangians intersecting transversely at a finite collection of points. 
Let $L_0\# L_1$ be the Lagrangian obtained by performing the above surgery at each intersection point.\footnote{When $L_0$ and $L_1$ intersect at a unique point, the Lagrangian $L_0 \# L_1$ is topologically a connect sum. This is no longer true when the intersection consists of more than one point.} 
Then there is a Lagrangian cobordism with ends $L_0, L_1$ and $L_0\# L_1$.
\end{cor} 

\begin{rmk} 
One can show that when performing surgery at a single point, $L_0\# L_1$ is independent, up to Hamiltonian isotopy, of the specific curve $\gamma$ used to construct the surgery neck. 
This is no longer the case when there are multiple intersection points. 
In particular, byconcatenating two cobordisms with different surgery necks we obtain cobordant (embedded) Lagrangians that need not be Hamiltonian isotopic. 
This shows that Lagrangian cobordism is a strictly coarser equivalence relation than Hamiltonian isotopy.
\end{rmk} 

\begin{rmk}\label{rmk:hprinciple}
It follows from \Cref{cor:multiplesurgery} that (embedded) Lagrangian cobordisms, without extra decorations, are topological.
Namely, there is an $h$-principle for Lagrangian immersions, showing they are governed by algebraic topology. 
One can then perturb the Lagrangian immersion by a Hamiltonian isotopy to make all the intersections transverse, and applying the surgery cobordism yields an embedded Lagrangian cobordism.
However, once we equip Lagrangians and cobordisms with extra data (see \Cref{rmk:restrictionofdata}) some rigidity appears.
This was first noticed by Chekanov \cite{chekanov1997lagrangian}, who showed that monotone Lagrangian cobordisms are no longer flexible.
Monotone Lagrangians are precisely the Lagrangians considered in the original work of Biran-Cornea \cite{biran2013lagrangian,biran2014lagrangian}, where they show that monotonicity imposes enough rigidity on Lagrangian cobordisms to preserve Floer theory (meaning that monotone Lagrangian cobordant Lagrangians are isomorphic objects in the Fukaya category).
\end{rmk} 

A generalization of the above surgery cobordism to non-transverse intersections has been constructed by Hicks \cite{hicks2019tropical,hicks2020tropical}.
He shows that if $L_0$ and $L_1$ are Lagrangians whose intersection can be modeled on the intersection of the zero section and the graph of a convex function, then there exists a Lagrangian $L_0 \# L_1$ agreeing with $L_0$ and $L_1$ outside a neighborhood of the intersection region; furthermore, there is a cobordism between $L_0, L_1$ and $L_0\# L_1$ \cite[Corollary 3.5]{hicks2020tropical}.
Now assume $X(B) \to B$ comes equipped with a Lagrangian torus fibration.
Let $L_0 = \Gamma_0$ be the zero section and $L_1 = \Gamma(d\phi) : = \Gamma(d\tilde\phi)$  be the graph of (a smoothing of) a \emph{tropical polynomial}---a  piecewise linear function  $\phi: B \to \R$ with integer slope which is furthermore convex ($\tilde\phi$ denotes a suitable smoothing, see \cite[Section 3.2]{hicks2020tropical}).
Then Hicks uses his surgery construction to show that the projection of the surgery Lagrangian $\Gamma_0\# \Gamma(d\phi)$ to $B$ lives arbitrarily close to the \emph{tropical hypersurface} $V(\phi)$ (the locus of points in $B$ where $\phi$ is not smooth) \cite[Theorem 3.17]{hicks2020tropical}.

We will use the following instance of Hicks' construction in our computations in the next sections.
The tropical hypersurface $V(\phi)$ will be a collection of disjoint, codimension-$1$ tropical submanifolds $V_1,\dots ,V_k \subset B$.
Each $V_i$ comes with a `weight' $w_i \in \Z$ that records the `amount of bending' of $\phi$ at $V_i$.
Let $L_i = N^*V_i / N_\Z^* V_i$ be the Lagrangian lift of $V_i$ and write  $L_\phi : = \sum_i w_i L_i$.
A surgery cobordism shows  that $\Gamma_0 \# \Gamma(d\phi) \sim L_\phi$ in the cobordism group.\footnote{An interesting point is that a cobordism from $\Gamma_0 \# \Gamma(d\phi)$ to $L_\phi$ does \emph{not} exist; rather, there is a cobordism between $(\Gamma_0 \# \Gamma(d\phi), \Gamma_0)$ and $(L_\phi,\Gamma_0)$.}
Let us state the precise result we will need for further reference:
\begin{prop}\label{prop:jeffsurgery}
    Consider the Lagrangian torus fibration $X(B) \to B$, and let $\Gamma_0$ be the zero section.
    Let $\phi$ be a convex piecewise linear function such that $V(\phi)=V_1\sqcup \dots \sqcup V_k$ is a collection of disjoint tropical submanifolds and let $L_\phi$ be defined as above.
    Then the relation
    $$
    \Gamma_0 + L_\phi \sim  \Gamma(d\phi)
    $$
    holds in $\Cob(X(B))$. 
\end{prop}
\begin{rmk}
    In fact, in our computations the sections $\Gamma(d\phi)$ will be product-type, meaning the function $\phi$ will only depend on one of the coordinates $(x,y) \in K$.
    One could then pull back such a section to $T^2 \x T^2 \to \mcal K$ and surger it with the $0$-section by considering the product $V \x \Gamma_0 \subset (T^2 \x \C) \x T^2$ (or $\Gamma_0 \x V \subset T^2 \x (T^2\x\C)$) of a standard surgery cobordism $V \subset T^2 \x \C$ in one of the $T^2$-factors with the $0$-section $\Gamma_0$ in the other factor.
    This cobordism is invariant under the covering group of the cover $T^2 \x T^2 \x \C \to \mcal K \x C$, thus descends to $\mcal K \x \C$ and recovers the cobordism relation of \Cref{prop:jeffsurgery}.
\end{rmk}

\subsection{The Lagrangian cobordism group of $T^2$}\label{sec:lagcobT2}
In this section we recall the computation of the Lagrangian cobordism group of $T^2$ whose generators are oriented embedded non-contractible Lagrangians and whose relations come from oriented embedded cobordisms.
This will not just exhibit the use of Hamiltonian isotopies and Lagrangian surgeries to compute cobordism groups, but in fact will be essential in our computation of the cobordism group of a bielliptic surface in \Cref{sec:computationcob}.
The computation of $\Cob(T^2)$ was first done in \cite{haug2015lagrangian} using homological mirror symmetry. 
We will present a computation that only requires standard symplectic topology arguments.

Denote by $\Cob(T^2)_{hom}$  the subgroup of $\Cob(T^2)$ generated by homologically trivial formal sums of Lagrangians (this is the kernel of the {\it cycle-class map} $\Cob(T^2) \to H_1(T^2;\Z)$, which exists because all Lagrangians and cobordisms are oriented---cf. \Cref{sec:cycleclassmap}). 
The following Lemma is an application of Stokes' theorem.
\begin{lem}\label{lem:fluxintorus}
    There is a well-defined map $\Phi:\Cob(T^2)_{hom} \to \R/(\Z\, area(T^2))\cong S^1$ sending a homologically trivial collection of Lagrangians $L_1+\dots+ L_k$ to the area $\omega(u)$ of a topological 2-chain with $\partial u=L_1+\dots+L_k$. 
\end{lem}

We now state the main cobordism relations in $T^2$, which will be used to compute $\Cob(T^2)$ and will also be useful later in \Cref{sec:computationcob}.
\begin{lem}\label{lem:relationsT2}
The following relations hold in $\Cob(T^2)$:
\begin{itemize}
    \item For any $a,b,\theta\in \R/\Z$, we have
    \begin{equation}\label{eq:relation1T2}
        S^1\x\{a\}-S^1\x\{a+\theta\}\sim \{b\}\x S^1 - \{b+\theta\}\x S^1.
    \end{equation}
    \item For any $a,\theta\in S^1$, we have
    \begin{equation}\label{eq:relation2T2}
    S^1\x\{0\}-S^1\x\{\theta\}\sim S^1\x\{a\}-S^1\x\{a+\theta\}.
    \end{equation}
\end{itemize}
\end{lem}
\begin{proof}  
    The first relation is equivalent to $S^1\x\{a\} + \{b+\theta\}\x S^1 \sim S^1\x\{a+\theta\} + \{b\}\x S^1$, or using surgery to $(S^1\x\{a\}) \# (\{b+\theta\}\x S^1) \sim (S^1\x\{a+\theta\}) \# (\{b\}\x S^1)$.
    These two Lagrangians can be easily seen to be Hamiltonian isotopic by flux considerations.
    The second relation is an immediate corollary of the first: both sides are cobordant to $\{b\}\x S^1 - \{b+\theta\}\x S^1$ for any $b\in S^1$.
\end{proof}

With this one can then show the following:
\begin{prop}\label{prop:fluxisaniso}
    The map $\Phi$ of \Cref{lem:fluxintorus} is an isomorphism.
\end{prop}
\begin{proof}
    Surjectivity is obvious: for any $\theta\in S^1$, we have $S^1\x\{0\}-S^1\x\{\theta\}=\partial(S^1\x[0,\theta])$ and $\text{area }(S^1\x[0,\theta])=\theta$. 
    To prove injectivity we proceed as follows.
    Let $L_1+\dots+L_k\in \ker \Phi \subset \Cob(T^2)_{hom}$; we must prove that $L_1+\dots+L_k=0$ in $\Cob(T^2)$.
    We first write each $L_i$ as an iterated surgery of $m_i$ horizontal lines and $n_i$ vertical lines, where $(m_i,n_i) = [L_i] \in H_1(T^2;\Z)$ (such surgery yields a Lagrangian in the same homology class as $L_i$, and by sliding one of the horizontal or vertical Lagrangians we can arrange the surgery to be Hamiltonian isotopic to $L_i$).
    Hence we may assume that all the $L_i$ are either horizontal or vertical straight lines.
    If $L^{\pm}_{i_1},\dots,L^{\pm}_{i_l}$ are vertical Lagrangians with homology class $[L^\pm_{i_j}]=(0,\pm1)$, using \Cref{eq:relation1T2} we may turn each pair $L^\pm_{i_j}$ into a cobordant pair of horizontal Lagrangians. 
    We have now reduced to the case where we have a collection of horizontal Lagrangians $L_1^\pm,\dots,L_n^\pm$ with homology class $[L_i^\pm]=(\pm1,0)$. 
    By \Cref{eq:relation2T2}, for each $i=2,\dots,n$ we may slide the pair $L_i^\pm$ so that $L_i^+$ lies over $L_{i-1}^-$, thus cancelling. 
    Doing this for all $i$ leads the two Lagrangians $L_1^+$ and $L_n^-$, and since they bound area $k\, area(T^2)$ for some $k \in \Z$ (they are in the kernel of $\Phi$) they must lie on top of eachother. 
    Thus all the Lagrangians cancel and we are left with the zero-sum on $\Cob(T^2)$, as wanted.
\end{proof}

\begin{cor}\label{pr:CobT2}
The Lagrangian cobordism group of closed, embedded, oriented and non-contractible Lagrangians of $T^2$ is
$$
\Cob(T^2)\cong \Z^2 \oplus S^1.
$$
\end{cor}

\begin{proof}
Consider the short exact sequence
$$
0\to \Cob(T^2)_{hom} \to \Cob (T^2) \xrightarrow{cyc} H_1(T^2;\Z) \to 0,
$$
which splits after a choice of base point $p\in T^2$. 
Using \Cref{prop:fluxisaniso} the result follows.
\end{proof}

For Lagrangians equipped with $G$-local systems, there is an extension of $\Phi$ to a map $\Cob^{loc}(T^2)_{hom} \to S^1 \oplus G$ that records the product of the monodromies of each of the local systems.
A similar argument to that of \Cref{prop:fluxisaniso} shows this map is an isomorphism, thus one obtains:
\begin{cor}\label{cor:cobt2locsys}
    For any abelian group $G$, the Lagrangian cobordism group of closed, embedded, oriented and non-contractible Lagrangians of $T^2$ equipped with $G$-local systems is
    $$
    \Cob(T^2)\cong \Z^2 \oplus (S^1 \oplus G).
    $$
\end{cor}

     \subsection{Extra data}\label{sec:branes}

Recall from \Cref{rmk:restrictionofdata} that one often equips Lagrangians and cobordisms with various decorations (this is essential to pass from a flexible to a rigid notion).
In this section we discuss the extra structure that we will consider on all our Lagrangians: gradings, Pin structures and local systems.
The first two are necessary to define a $\Z$-graded Fukaya category  with signs (i.e. away of  fields of characteristic $2$) \cite{seidel2008fukaya,fukaya2009lagrangian}.
The third decoration is a $G$-local system on each Lagrangian, where $G$ is some abelian group.
When $G = U_\Lambda$ is the unitary group of the Novikov field, it is expected that this data is needed to have a homological mirror symmetry equivalence between the Fukaya category (without taking the split closure) and the derived category of coherent sheaves of its mirror.
We will refer to Lagrangians with this extra structure---that is, a grading, a Pin structure and a $G$-local system---as {\it Lagrangian $G$-branes} (or simply Lagrangian branes).
\begin{rmk}
    We work with an arbitrary $G$ because our computation of the cobordism group works in that generality (see \Cref{th:cobK}). 
However, to compute the Grothendieck group and show that it agrees with the cobordism group (\Cref{th:isoK}), we need to apply homological mirror symmetry and thus we specialize to the case $G = U_\Lambda$.
\end{rmk}

We start with gradings.
This is some extra data on the Lagrangians that allows us to define a $\Z$-graded Fukaya category (meaning the morphism spaces are $\Z$-graded vector spaces). 
For a complete exposition on gradings see \cite[Section (11j)]{seidel2008fukaya} or Seidel's original paper \cite{seidel2000graded}.

Let $X^{2n}$ be a symplectic manifold such that $2c_1(X) = 0$. 
This means the complex line bundle $(\Omega^n)^{\otimes 2}$ of holomorphic quadratic volume forms is trivial.
Let $\eta^2: X \to (\Omega^n)^{\otimes 2}$ be a nowhere zero section (note this is the same data as a section of $\Omega^n$ defined up to a sign $\pm 1$, which we denote by $\pm \eta$). Given a Lagrangian submanifold $L \subset X$, there is an associated {\it (squared) phase map} $\alpha_L : L \to S^1$ given by
$$
\alpha_L(p) = \frac{\eta_p^2(vol_{L,p}^{\otimes 2})}{|\eta_p^2(vol_{L,p}^{\otimes 2})|}
$$
where $vol_{L,p}$ is a generator of  $\wdg^n T_p L$. 
The cohomology class $\mu_L \in H^1(L;\Z)$ determined by $\alpha_L$ is called the {\it Maslov class of $L$}.
By definition, $\mu_L$ is the obstruction to the existence of a lift $\tilde \alpha_L : L \to \R$ to the universal cover $\R \to S^1, t \mapsto e^{i2\pi t}$.
\begin{defn}\label{defn:gradedLagrangian}
    Let $L \subset X$ be a Lagrangian with $\mu_L = 0$.
    A pair $(L,\tilde\alpha_L)$ is called a {\it graded Lagrangian}.
\end{defn}

\begin{rmk}\label{rmk:choiceoftrivialisation}
    The phase map $\alpha_L$ depends on the homotopy class of the trivialization $\eta^2$. 
    Whether a Lagrangian admits a grading or not  will therefore also depend on such choice.
\end{rmk}

\begin{rmk}\label{rmk:relativeorientations}
    As explained in \cite[Remark 11.18]{seidel2008fukaya}, a graded Lagrangian is naturally oriented with respect to the $\Z$-local system $\xi \subset \wedge^n_\C TX$ where $\pm \eta$ takes integer values.
    Hence, every graded Lagrangian defines a class in the homology group $H_n(X;\xi)$.
\end{rmk}

Next we turn to (twisted) Pin structures, again following \cite{seidel2008fukaya}.
These are needed to orient the moduli-spaces that define the $A_\infty$ operations, so that one can define a Fukaya category over a field of characteristic different from two.
First recall that there is a group $Pin_n$, constructed explicitly as a subgroup of the Clifford algebra of $\R^n$, that admits a $2$-to-$1$ map $Pin_n \to O_n$.
A \emph{Pin structure} on an $n$-dimensional manifold $M$ is a principal $Pin_n$-bundle $P \to M$ together with an isomorphism $P \x_{Pin_n} \R^n \cong TM$, where $Pin_n$ acts on $\R^n$ via the composition $Pin_n \to O_n \into \GL(\R^n)$.
The obstruction to the existence of a Pin structure is the second Stiefel-Whitney class $w_2(M) \in H^2(M;\Z_2)$.
More generally, given some class $w \in H^2(M;\Z_2)$ and an oriented vector bundle $E \to M$ with $w_2(E) = w$, one can define a \emph{twisted Pin structure} as a Pin structure on $TM \oplus E$.
The obstruction to the existence of a Pin structure is now $w_2(E) - w_2(M)$.

If $M = X$ is a symplectic manifold and $E \to X$ is an oriented vector bundle, one can define a  Fukaya category whose objects are Lagrangians equipped with a Pin structure on the bundle $TL \oplus E\restr{L}$.
This structure is enough to orient the moduli-spaces defining the $A_\infty$-operations, and hence to define a Fukaya category away of characteristic $2$ \cite{seidel2008fukaya,fukaya2009lagrangian}.

Lastly, let us recall that, for an (abelian) group $G$, a $G$-local system on a manifold  $M$ is the data of a group homomorphism 
$$
\xi: \pi_1(M) \to G.
$$
Geometrically, one can think of a local system as a principal $G$-bundle with a \emph{flat} connection.
Namely, the fibers of a $G$-principal bundle give an assignment of principal homogenous $G$-space  $\xi_p$ to each $p \in M$.
The data of a flat connection gives, via parallel transport, isomorphisms $\xi_{\gamma(0)} \cong \xi_{\gamma(1)}$ for every path $\gamma:[0,1] \to M$.
Furthermore, these isomorphisms depend only on the homotopy class of $\gamma$ relative to its endpoints.
Considering isomorphisms associated to loops returns the map $\xi: \pi_1(M) \to G$.

\begin{ex}\label{rmk:branechoiceforfibration}
Suppose that $\pi:X \to B$ comes with a Lagrangian torus fibration.
In this case, there are natural choices of $\eta^2: X \to (\Omega^n)^{\otimes 2}$ and $w\in H^2(X;\Z_2)$ to define gradings and twisted Pin structures. 
Namely, we choose $\eta^2$ to be the complexification of a section of $(\wedge^n T^*B)^{\otimes 2}$, and $w = \pi^*(w_2(B))$ to be the pullback of the second Stiefel-Whitney class of $B$.
In full generality one can consider $E := \pi^*(TB \oplus (\det B)^{\oplus 3})$ (this choice ensures that $w_2(E) = \pi^*w_2(B)$ while making $E$ orientable), but when $w = 0$ we will simply take $E = 0$ to be the trivial vector bundle.
The obstruction to a twisted Pin structure (a Pin structure on $TL \oplus E\restr{L}$) is then $w_2(L) + (\pi^*w_2(B))\restr{L}$.
Note these choices ensure that fibers and sections admit a brane structure.
Furthermore, with this choice of $\eta^2$ the $\Z$-local system $\xi$ of \Cref{rmk:relativeorientations} is the pullback $\pi^*(\wedge^n T^*_\Z B)$ of the integral orientation bundle  of the base.
Hence every graded Lagrangian is oriented with respect to $\pi^*(\wedge^n T^*_\Z B)$.
\end{ex}

Summarizing, the Lagrangians we consider will be \emph{Lagrangian $G$-branes}---that is, they will carry the data of a grading, a twisted Pin structure and a $G$-local system.
A cobordism between such Lagrangians is then a Lagrangian cobordism $V \subset X \x \C$ carrying also a $G$-brane structure;\footnote{The almost complex structure on $X \x \C$ used to define gradings is chosen to be of the form $J_{X \x \C} = J_X \x j_\C$ at infinity, where $j_\C$ is the standard complex structure on $\C$.} it induces a relation between the ends equipped with the restricted brane structure.
The fact that we can restrict Pin structures and gradings follows from the product-type structure of the cobordism and complex structure at infinity, together with the fact that the ends are all horizontal (thus, the composition of the phase map of the cobordism with the inclusion of one of its ends coincides with the phase map of the end).

\section{The Lagrangian cobordism group of a bielliptic surface}\label{sec:computationcob}

    In this section we prove the main results of this paper.
    The first, see \Cref{th:cobK}, is a computation of  the subgroup $\Cob^{trop}(\mcal K) \subset \Cob(\mcal K)$ generated by Lagrangian sections and fibers of the projection $\pi:\mcal K \to K$, where  $\mcal K$ is a symplectic bielliptic surface as defined in \Cref{sec:biellipticsurfaces}. 
    Using this computation together with Abouzaid's homological mirror symmetry result  \cite{abouzaid2021homological}, we  show in \Cref{th:isoK} that $\Cob^{trop}(\mcal K)$ is naturally isomorphic to the Grothendieck group of the Fukaya category.

    \subsection{Algebraic and symplectic bielliptic surfaces}\label{sec:biellipticsurfaces}
We start by defining the algebraic varieties and symplectic manifolds that we call bielliptic surfaces. 
On the algebraic side, we always work over an algebraically closed field of characteristic zero.

In the Kodaira-Enriques classification of surfaces, a {\it bielliptic surface} (also known as a {\it hyperelliptic surface}) is a minimal surface $Y$ characterised by 
$$
\kappa (Y) = 0, \qquad
q (Y) = 1, \qquad 
p_g(Y) = 0
$$
where $\kappa$ denotes the Kodaira dimension, $q (Y) = h^1(Y,\mcal O_Y)$ is the irregularity and $p_g(Y) = h^2(Y,\mcal O_Y)$ is the geometric genus (this condition $p_g(Y) = 0$ can be omitted in characteristic zero).
Since $q(Y) = 1$ equals the dimension of the Albanese variety of $Y$, $\Alb(Y)$ is an elliptic curve; moreover, the Albanese morphism $\alb: Y \to \Alb(Y)$ is an elliptic fibration and all fibers are smooth \cite{bombieri1977enriques}.
One can show that there exists another elliptic fibration  $Y \to \P^1$ that is transverse to $\alb$.
This justifies the name \emph{bielliptic}, as these surfaces admit two (transverse) elliptic fibrations.
These two fibrations can be seen explicitly, as follows from the following structure theorem for bielliptic surfaces (see \cite[Theorem 4, §3]{bombieri1977enriques}):

\begin{thmdef}\label{def:biellipticalgebraic}
Let $Y$ be a bielliptic surface. 
There exist elliptic curves $E_1,E_2$ and a finite group $G \subset E_1$  of translations of $E_1$ acting also on $E_2$ such that $Y=(E_1\x E_2)/G$.
The quotient $E_1/G$ is again an elliptic curve and $E_2/G \cong \mathbb P^1$.
\end{thmdef}

It follows from Serre duality that $h^0(Y,\mcal K_Y) = 0$, so in particular the canonical bundle $\mcal K_Y$ is not trivial; nonetheless, it can be proved using the canonical bundle formula for elliptic fibrations and the structure theorem above that it is either two-, three-, four- or six-torsion. 

We remark that  the quotient $E_1/G$ is again an elliptic curve, and the projection $Y\to E_1/G$ is a fiber bundle (all  fibers being smooth elliptic curves isomorphic to $E_2$).
In fact, one can identify $E_1/G$ (together with the projection $\pi_1:Y\to E_1/G$) with the Albanese variety of $Y$.\footnote{This property actually provides a third (equivalent) definition of bielliptic surfaces: a bielliptic surface is an algebraic surface such that the Albanese morphism $Y\to \Alb(Y)$ is an elliptic fibration.} 
On the other hand, the projection $\pi_2: Y \to E_2 /G$ is  an elliptic fibration over $E_2/G \cong \mathbb P^1$.
This fibration is however not smooth: the generic fiber is isomorphic to $E_1$, but over the branching points of the projection $E_2 \to E_2 /G$ we have non-reduced fibers, their multiplicity being equal to that of the associated branching point. 
For more details on the algebraic geometry of bielliptic surfaces we refer the reader to \cite{bombieri1977enriques}. 

We now introduce symplectic bielliptic surfaces. 
In \Cref{sec:hmsbielliptic} we explain how symplectic and algebraic bielliptic surfaces are homologically mirror as an instance of Abouzaid's homological mirror symmetry result \cite{abouzaid2014family,abouzaid2017family,abouzaid2021homological}.
Recall Abouzaid's result is proved using family Floer theory, whose input is a symplectic manifold together with a Lagrangian torus fibration over some tropical affine manifold. 
In our case, we take the tropical affine manifold to be a Klein bottle (see \Cref{def:kleinbottle}) and construct a symplectic manifold $X(K)$ as explained in \Cref{sec:lagtorusfib}:

\begin{defn}\label{def:biellipticsymplectic}
Let $K$ be a tropical Klein bottle. 
The total space $\mcal K := X(K)$ of the Lagrangian torus fibration $X(K) \to K$ is called a {\it symplectic bielliptic surface}.
\end{defn}

The following Lemma shows this definition agrees with the symplectic of  \Cref{def:biellipticalgebraic},
hence the name symplectic bielliptic surface.
Namely, we have the following:

\begin{lem}\label{lem:biellipticisquotientoftori}
    Every symplectic bielliptic surface is the quotient of the product of two symplectic $2$-tori by the action of a finite group of symplectomorphisms.
\end{lem}
\begin{proof}
    Recall from \Cref{def:kleinbottle} that every tropical Klein bottle is the quotient of a tropical affine torus $T^2_\Lambda$ by a finite group.
    The discussion of \Cref{rmk:Kleinfamilies} shows that there is a tropical morphism $T \to T^2_\Lambda$ from a product tropical torus $T \simeq S^1 \x S^1$   onto $T^2_\Lambda$ which is a diffeomorphism of the underlying smooth manifolds (but might not be a tropical isomorphism).
    The composition $T \to T^2_\Lambda \to  K$ is a local diffeomorphism and also a tropical map, thus it induces a map 
    $$
    X(T) \simeq X(S^1) \x X(S^1) \to \mcal K
    $$
    which can be seen to be a projection $X(T) \to X(T) /G$.
    Hence every symplectic bielliptic surface can be expressed as a quotient
    $$
    \mcal K \simeq ((T^2,\omega_1) \x (T^2,\omega_2))/G
    $$
    of two symplectic $2$-tori by a finite group of symplectomorphisms.
\end{proof}

To be more precise, if $(q_i,p_i)$ are Arnold-Liouville coordinates on the Lagrangian torus fibration $(T^2,\omega_i) \to S^1_{q_i}$, then the group $G$ can either be:
\begin{enumerate}
    \item $G = \Z_2$, acting by 
    $$
    (q_1,p_1,q_2,p_2) \mapsto (q_1+\theta,p_1,-q_2,-p_2)
    $$
    and $2\theta =0\neq \theta$; or
    \item $G = \Z_2\oplus\Z_2$, where the first generator acts as before and the second as 
    $$
    (q_1,p_1,q_2,p_2) \mapsto (q_1,p_1+\phi_1,q_2,p_2+\phi_2),
    $$
    and again $2\phi_i =0 \neq \phi_1\phi_2$.
\end{enumerate}

\begin{rmk}
    Note that while the first action modifies the topology of the base of the Lagrangian torus fibration $T^2 \x T^2 \to S^1 \x S^1$, the second only modifies its tropical affine structure.
\end{rmk}

\begin{rmk}
    As mentioned in \Cref{rmk:kleintypeIandII}, these two symplectic bielliptic surfaces correspond to algebraic bielliptic surfaces of type I and II in the classification of algebraic bielliptic surfaces. 
    The classification of complex bielliptic  surfaces (see e.g. \cite[List VI.20]{beauville1996complex}) includes seven families, with more possible groups than $G = \Z_2$ or $G = \Z_2 \oplus \Z_2$.
    However, it is only in two of these families  that the complex structure can be deformed in a maximally degenerating way.
    Hence, these are the only two that make sense for mirror symmetry, where to a symplectic manifold one associates a {\it maximal degeneration} of an algebraic-geometric object. 
\end{rmk}

    \subsection{Homological mirror symmetry for bielliptic surfaces}\label{sec:hmsbielliptic}
Let $\mcal K$ be a symplectic bielliptic surface. 
Recall that, by definition, $\mcal K$ comes equipped with a Lagrangian torus fibration $\pi:\mcal K \to K$.
We observe that the fibration is smooth (i.e. there are no singular fibers) and that $\pi_2(K) = 0$.
This puts us in the setting to apply the results in \cite{abouzaid2021homological}, where a (rigid analytic) mirror is constructed and homological mirror symmetry is proved for such a pair.

Denote by $\check {\mcal K}$ the rigid analytic mirror of $\mcal K$.
Set-wise, $\check {\mcal K}$  is the disjoint union 
\begin{equation}\label{eq:mirrorconstruction}
    \check {\mcal K} := \bigsqcup_{p\in K} H^1(F_p; U_\Lambda),
\end{equation}
where $F_p := \pi\inv(p) \subset \mcal K$ is the Lagrangian torus fiber over $p\in \mcal K$ and $U_\Lambda \subset \Lambda^*$ is the unitary subgroup of the Novikov field $\Lambda$.
The space $\check {\mcal K}$ can be endowed with the structure of  a rigid analytic space \cite{abouzaid2014family}.
Examining \Cref{eq:mirrorconstruction} and using that a tropical Klein bottle is the quotient of a tropical $2$-torus (\Cref{def:kleinbottle}), one sees that the rigid analytic space $\check{\mcal K}$ is the quotient of a product $\mcal E_1 \x \mcal E_2$ of two Tate curves by the action of a finite group $G$, and each of these Tate curves is the analytification of an algebraic elliptic curve $E_i$ over the Novikov field.
Hence we have that
$$
\check{\mcal K} = (\mcal E_1 \x \mcal E_2) / G = (E_1^{an} \x E_2^{an}) / G = (E_1 \x E_2/G)^{an}
$$
is the analytification of an algebraic variety $Y := (E_1 \x E_2) / G$. 
Note that $Y$ is an algebraic bielliptic surface (see \Cref{def:biellipticalgebraic}).
One can then apply the rigid analytic GAGA principle to obtain an equivalence of categories
$$
D^b Coh_{an} (\check {\mcal K}) \cong D^b Coh (Y),
$$
so that combining this with the homological mirror symmetry statement of \cite{abouzaid2021homological} we obtain:

\begin{prop}\label{prop:HMSbielliptic}
    Let $\mcal K$ be a symplectic bielliptic surface and $D^\pi\mcal Fuk(\mcal K)$ the split-closed derived Fukaya category of $\mcal K$.
    Then there is an algebraic bielliptic surface $Y$ over the Novikov field $\Lambda$ and an equivalence of triangulated categories
    $$
        D^\pi\mcal Fuk(\mcal K) \cong D^b Coh(Y).
    $$
\end{prop}

\begin{rmk}
    Strictly speaking, \cite{abouzaid2021homological} proves there is a fully faithful embedding $D^\pi\mcal Fuk(\mcal K) \into D^bCoh_{an}(\check{\mcal K})$ of the Fukaya category into the category of (twisted) coherent sheaves, with no discussion of essential surjectivity.
    However, for a smooth quasi-projective variety $Y$ we have that $D^bCoh(Y) \simeq Perf(Y)$ is idempotent-complete.
In this case, essential surjectivity is equivalent to showing that the image of the family Floer functor split-generates $D^bCoh(Y)$.
It is a result of Orlov \cite{orlov2008remarks} that sufficiently many powers of any ample line bundle are enough to split-generate $D^bCoh(Y)$, so it suffices to show such an ample line bundle exists in the image of the family Floer embedding.
    One can check that the construction in \cite{abouzaid2014family} of the complex of coherent sheaves mirror to a Lagrangian section is indeed a line bundle,
and that when the function defining the section is strictly convex such line bundle  is ample. 
    As $K$ admits a strictly convex Lagrangian section by \Cref{lem:classificationoftropicalsections}, we indeed have an equivalence $D^\pi\mcal Fuk(\mcal K) \simeq D^b Coh_{an}(\check{\mcal K})$.
\end{rmk}
\begin{rmk}
    In fact, for symplectic bielliptic surfaces we have that $D\mcal Fuk(\mcal K) \cong D^\pi\mcal Fuk(\mcal K)$---see \Cref{rmk:dbcongdpi}.
    Hence we obtain a homological mirror symmetry equivalence $D\mcal Fuk(\mcal K) \cong D^b Coh(Y)$ with no need of taking split-closure on the A-side.
\end{rmk}

\subsection{Lagrangian sections in bielliptic surfaces}\label{sec:admissible}
Let $\mcal K$ be a symplectic bielliptic surface (recall this is the total space of a Lagrangian torus fibration over a Klein bottle).
From now on  we fix $K = K_1 $ to be a tropical Klein bottle of type I (see \Cref{rmk:Kleinfamilies}). 
Hence $\mcal K = T^*K_1 / T^*_\Z K_1$ will always be a symplectic bielliptic surface mirror to an algebraic bielliptic surface of type I (see \Cref{rmk:kleintypeIandII}).
In this section we classify Lagrangian sections and compute their homology class.

Recall from \Cref{sec:lagtorusfib} that Lagrangian sections in $\mcal K = T^*K / T^*_\Z K$ are in one-to-one correspondence with elements in $H^0(C^\infty_K/\Aff_K)$.
Among these, global sections of $H^0(C^\infty_K)$ give Lagrangian sections that are Hamiltonian isotopic to the zero section. 
Thus, the group
\begin{equation}\label{eq:H1Affasquotient}
    H^1(\Aff) \cong H^0(C^\infty / \Aff) / H^0(C^\infty)    
\end{equation}
parametrizes Lagrangian sections up to Hamiltonian isotopy.
We now determine this group:

\begin{lem}\label{lem:classificationoftropicalsections}
    Let $K$ be a tropical Klein bottle. 
    Then $H^1(\Aff)\cong \Z^2\oplus\Z_2 \oplus S^1$, with an element $(m,n,[l],e^{2\pi i\theta})$ represented by the section  $f_{m\theta}^{nl} \in H^0(C^\infty_K / \Aff_K)$ defined by
    \begin{equation}\label{eq:fmnltheta}
        f_{m\theta}^{nl}(x,y) := 
        \frac{m}{2}x^2 + \frac{n}{2}y^2 + \frac{l}{2} y + \theta x.
    \end{equation}
\end{lem}
\begin{proof}
    Take an element in $H^0(C^\infty_K/\Aff_K)$.
    Pulling it back to $H^0(C^\infty_{\R^2}/\Aff_{\R^2})$ and lifting it to a smooth  function, we are looking at functions $f\in H^0(C^\infty_{\R^2})$ satisfying the periodicity conditions 
        \begin{align}
         f(x,y) -  f(x+1,-y) &= g_1(x,y) \label{eq:periodicity1}\\
         f(x,y) -  f(x,y+1) &= g_2(x,y)\label{eq:periodicity2}
    \end{align}
    for some affine functions $g_1,g_2\in \Aff(\R^2)$.
    Now note that 
    $$
    g_1(x,y) + g_2 (x +1, -y) = f(x,y) - f(x+1,-y+1) = g_1(x+1, y-1) - g_2(x, y-1),
    $$
    and this compatibility guarantees the existence of a quadratic function $f_q$ with the same periodicity as $f$.
    In particular, the difference $f - f_q$ descends to a well-defined smooth function on $K$, hence $f$ and $f_q$ define the same class in $H^1(\Aff)$.
    Thus, we may assume that the element in $H^0(C^\infty_K/\Aff_K)$ is given by a quadratic polynomial $f$, and similarly that it vanishes at a preferred point $0\in K$.
    Then we want to find  quadratic polynomials
    $$
     f(x,y) = \sum_{\substack{i,j=0\\ 0<i+j\leq 2}}^2 a_{ij} x^i y^j
    $$
    that satisfy the periodicity conditions (\ref{eq:periodicity1}) and (\ref{eq:periodicity2}).
    These imply $2a_{20},2a_{02}, 2a_{01}  \in \Z$ and $a_{11} = 0$.
    Since adding an affine function does not change the sections in $H^0(C^\infty_K/\Aff_K)$, we mod out by $(a_{01},a_{20})\in \Z^2$.
    This proves that $H^1(\Aff) \cong \Z^2 \oplus \Z_2 \oplus S^1$, with an element $(m,n,[l],e^{2\pi i\theta})$ represented by the function $f_{m\theta}^{nl}$ of \Cref{eq:fmnltheta}.
\end{proof}

We will now show that $f_{m\theta}^{n[l]}$ can be approximated by (a smoothing of) a piecewise-linear function with integer slope, in the sense that they define the same class in $H^1(\Aff)$.
We first construct the approximation and then prove in \Cref{lem:linearapproximation} that it defines the same class as $f_{m\theta}^{n[l]}$.

Let $\mathfrak f$ and $\mathfrak f_\phi$  be the unique (up to a constant) piecewise linear functions with integer slope $\R \to \R$ such that:
\begin{itemize}
    \item if $x$ is a smooth point of $\mathfrak f$, then $\mathfrak f'(x + 1) = \mathfrak f'(x) + 1$, and $\mathfrak f'(x) = 0$ for $x \in (-\frac{1}{2}, \frac{1}{2})$;
    \item $\mathfrak f_\phi(x+1) = \mathfrak f_\phi(x) + \phi$, $\mathfrak f_\phi'(x) = 0$ for $x \in (0,\phi)$ and $\mathfrak f_\phi'(x) = 1$ for $x \in (1-\phi,1)$.
\end{itemize}
From these we define a new piecewise linear function $\mathfrak f^{nl}_{m\theta}: \R^2 \to \R$ by
\begin{equation}\label{eq:piecewiselinearapproximation}
\mathfrak f^{nl}_{m\theta}(x,y) := (m\mathfrak f + \mathfrak f_\theta)(x) + (n\mathfrak f + \mathfrak f_l)(y). 
\end{equation}
It follows from the definition of $\mathfrak f$ and $\mathfrak f_\phi$ that  $\mathfrak f^{nl}_{m\theta}$ satisfies periodicity equations as in (\ref{eq:periodicity1})-(\ref{eq:periodicity2}).
Hence---after a suitable smoothing as in \cite[Section 3.2]{hicks2020tropical}---it gives a well-defined element in  $H^0(C^\infty_K / \Aff_K)$.

\begin{lem}\label{lem:linearapproximation}
    The functions $f^{nl}_{m\theta}$ and $\mathfrak f^{nl}_{m\theta}$ have the same class in $H^1(\Aff)$.
\end{lem}
\begin{proof}
    Recall that  $H^1(\Aff) \cong H^0(C^\infty / \Aff) / H^0(C^\infty)$, so we must show that the difference $f^{nl}_{m\theta} - \mathfrak f^{nl}_{m\theta}$ gives a well-defined smooth function on $K$.
    This is equivalent to showing that the periodicity functions $g_i$ of $f^{nl}_{m\theta}$ are the same as those of $\mathfrak f^{nl}_{m\theta}$.
    Note that $f^{nl}_{m\theta}(x,y) = h_{m\theta}(x) + g^{nl}(y)$ decomposes as a sum of one-variable functions, and that the periodicity conditions (\ref{eq:periodicity1})-(\ref{eq:periodicity2}) split for each summand. 
    Therefore, it is sufficient to prove that for $\mathfrak f_{m\phi} := m \mathfrak f + \mathfrak f_\phi$ and $h_{m\phi}(x) := \frac{m}{2} x^2 + \phi x$ one has
    \begin{align}
        \mathfrak f_{m\phi}(x) - \mathfrak f_{m\phi}(x+1) &=  h_{m\phi}(x) -  h_{m\phi}(x+1) \label{eq:copyperiodicity1}\\
        \mathfrak f_{m\phi}(x) - \mathfrak f_{m\phi}(-x) &= h_{m\phi}(x) - h_{m\phi}(-x) \quad \text{(if }\phi = 1/2).\label{eq:copyperiodicity2}
    \end{align}
Since the function $\mathfrak f$ (resp. $\mathfrak f_\phi$) satisfies \Cref{eq:copyperiodicity1,eq:copyperiodicity2} whenever $m=1,l=0$ (resp. $m=0,l=1$), the result follows by linearity.
\end{proof}

\begin{cor}\label{cor:Hamiltonianisotopicsections}
    The Lagrangian sections $\Gamma(df_{m\theta}^{nl})$ and $\Gamma(d \mathfrak f^{nl}_{m\theta})$ are Hamiltonian isotopic.
\end{cor}

Next we compute the homology class of an arbitrary section (recall from \Cref{rmk:relativeorientations} and \Cref{rmk:branechoiceforfibration} that we take homology with coefficients in the $\Z$-local system $\xi = \pi^*\det_\Z K$, where $\det_\Z K = \wedge^2 T_\Z K$ is the integral orientation bundle).
For this we first prove an important cobordism relation, which will also be useful later on.
In the remainder of this section we work with undecorated cobordisms, as these are enough to give homology relations.
Nonetheless, we show in \Cref{ap:compatibility} that the cobordisms we consider admit brane structures and work out the restriction of the brane structures to the ends.
This will be essential for the results in \Cref{cor:generatorscobhom}.

Write $f_{m\theta} := \mathfrak f_{m\theta}^{00}$ and $f^{nl}:= \mathfrak f_{00}^{nl}$ for (smoothings of) the functions defined in \Cref{eq:piecewiselinearapproximation}.
Writing $\Gamma_0$ for the zero-section and  $L_{f_{m\theta}}$ for the Lagrangian lift of the non-smooth locus of $f_{m\theta}$, we can use \Cref{prop:jeffsurgery} to obtain a cobordism relation\footnote{In \Cref{prop:jeffsurgery} the functions defining the Lagrangian sections are assumed to be convex. 
The functions $f_{m\theta}$ we consider are only convex when $m \geq 0$. However, it is easy to see that the claimed cobordism relation still holds if we revert the orientation of $L_{f_{m\theta}}$ when $m <0$.}
\begin{equation}\label{eq:jeffsurgerysections}
    \Gamma(df_{m\theta}) \sim L_{f_{m\theta}} + \Gamma_0
\end{equation}
and similarly for $\Gamma^{n[l]}$.
Now consider $\Gamma_{m\theta}^{n[l]} = \Gamma(df_{m\theta}^{nl})$.
We write $\otimes$ for the fiber-wise addition operation defined in \cite{subotic} (see \Cref{ap:compatibility} for more details on this operation).
Note that the operation $\otimes$ has the property
$$
\Gamma(d\varphi_1)\otimes \Gamma(d\varphi_2) = \Gamma(d(\varphi_1 + \varphi_2))
$$
for $\varphi_i \in H^0(C^\infty/\Aff)$. 
In particular, writing $\Gamma_{m\theta}:=\Gamma_{m\theta}^{0[0]}$ and  $\Gamma^{n[l]}:=\Gamma_{00}^{n[l]}$, we can decompose $\Gamma_{m\theta}^{n[l]}$ as a product 
\begin{equation}\label{eq:tensorofsections}
    \Gamma_{m\theta}^{n[l]}=\Gamma_{m\theta}\otimes \Gamma^{n[l]}.
\end{equation}
\begin{lem}\label{lem:cobsectionstogenerators}
    There is a cobordism relation $\Gamma_{m\theta}^{n[l]} \sim L_{ f_{m\theta}}\otimes L_{f^{nl}} + L_{f_{m\theta}} + L_{f^{nl}} + \Gamma _0$.
\end{lem}
\begin{proof}
Recalling that $\Gamma_{m\theta}$ (resp. $\Gamma^{n[l]}$) is Hamiltonian isotopic to $\Gamma(df_{m\theta})$ (resp. $\Gamma(df^{nl})$), we can combine \Cref{eq:tensorofsections} with \Cref{eq:jeffsurgerysections} to get cobordism relations
\begin{align*}
    \begin{split}
        \Gamma_{m\theta}^{n[l]} &= \Gamma_{m\theta}\otimes \Gamma^{n[l]}\\
        &\sim ( L_{ f_{m\theta}} + \Gamma_0 )\otimes ( L_{ f^{nl}} + \Gamma_0 )\\
        &=L_{ f_{m\theta}}\otimes L_{f^{nl}} + L_{f_{m\theta}} + L_{f^{nl}} + \Gamma _0.
    \end{split}
\end{align*}
\end{proof}

\begin{rmk}
    We are using implicitly that Lagrangian correspondences induce homomorphisms between Lagrangian cobordism groups \cite[Appendix C]{hanlon2022aspects}.
    In full generality, this is only true if either all the Lagrangians are undecorated (so that one can apply surgeries to resolve possible self-intersections arising from geometric composition, see \Cref{rmk:hprinciple}) or if one is content with working with immersed cobordisms. 
    However, homomorphisms of cobordism groups obtained by fiberwise summing with a Lagrangian section are actually induced by global  symplectomorphisms, thus giving well-defined group homomorphisms even at the embedded level. 
A detailed study of the transformation of the brane structures under fiberwise addition can be found in \Cref{ap:compatibility}.
\end{rmk}

We can now compute the homology class of the sections $\Gamma^{n[l]}_{m\theta}$.
In \Cref{ap:homology} we show that $H_2(\mcal K;\xi) \cong \Z^4 \oplus \Z_2$ and give explicit generators. 
Let us recall what these are.
Consider the tropical submanifolds 
\begin{equation}\label{eq:tropicalsubmanifolds}
    \{0\}, K, C_1 = \{ x = 0 \}, C_2 = \{ y = 1/2\}, C_3 = \{ y = 0\}
\end{equation}
inside $K$.
As in \Cref{sec:lagtorusfib}, for a tropical submanifold $V \subset K$ we denote by $L_V = N^*V / N^*_\Z V$ the Lagrangian in $\mcal K$ living over $V$.
Then the $\Z$-factors in $H_2(\mcal K;\xi)$ are generated by  $L_{\{0\}}$ (a fiber), $L_{K}$ (the zero-section) and the Lagrangians $L_{C_1}$ and $L_{C_2}$, whereas the $\Z_2$-factor is generated by the difference $L_{C_2} - L_{C_3}$.
Using this, we prove the following:

\begin{cor}\label{lem:homologyofsections}
    The homology class of $\Gamma_{m\theta}^{n[l]}$ is
    $$
    [\Gamma_{m\theta}^{n[l]}] = (mn,1,m,n,[   l])\in \Z^4 \oplus \Z_2 \cong H_2(\mcal K;\xi).
    $$
\end{cor}
\begin{proof}
    Using \Cref{lem:cobsectionstogenerators} it is enough to compute the homology class of the terms appearing on the right-hand side.
    By construction of $f_{m\theta}$, the tropical submanifold $V(f_{m\theta})$---the locus where $f_{m\theta}$ is not smooth, each component equipped with a multiplicity as explained before \Cref{prop:jeffsurgery}---is homologous to $m$ copies of $C_1$, whereas $V(f^{nl})$ is a combination of $n$ copies of $C_2$ and $l$ copies of the difference $C_2 - C_3$.
In particular, the  term $L_{ f_{m\theta}}\otimes L_{f^{nl}}$ in \Cref{lem:cobsectionstogenerators}---which lives over the intersection $V(f_{m\theta}) \cap V(f^{nl})$---is a collection of $m \cdot n$ fibers.
It then follows that
\begin{align*}
[\Gamma_{m\theta}^{n[l]}] &= (mn,0,0,0,[0]) + (0,0,m,0,[0]) + (0,0,0,n,[l]) + (0,1,0,0,[0])\\
    &= (mn,1,m,n,[l]),
\end{align*}
which proves the Corollary.
\end{proof}

\subsection{The cycle class map}\label{sec:cycleclassmap}
We start by defining  the precise cobordism group that we will consider (namely, the classes $\mcal L$ and $\mcal L_{cob}$ in \Cref{def:lagcobgroup}).
Given a symplectic manifold $X$, we say a Lagrangian $L \subset X$ is {\it weakly exact} if $\omega(\pi_2(\mcal K,L)) = 0$, and we say it is {\it tautologically unobstructed} or {\it quasi-exact} if there exists some almost complex structure $J$ on $X$ such that $L$ bounds no non-constant $J$-holomorphic disks.
We denote by $(L,\alpha,P,\eta)$ the data of a Lagrangian $G$-brane, where $L\subset \mcal K$ is a Lagrangian, $\alpha: L \to \R$ its grading, $P \to L$ a principal $Pin_2$-bundle giving the Pin structure and $\eta:\pi_1(L) \to G$ a local system.
Here, gradings and twisted Pin structures  are defined with respect to the natural choices explained in \Cref{rmk:branechoiceforfibration}; that is, we grade Lagrangians with respect to a quadratic volume form given by complexifying a section of $(\wedge^2 T^* K)^{\otimes 2}$, and a twisted Pin structure is in this case a regular Pin structure (as $w_2(K) = 0$).
The same notation is used for cobordisms.
\begin{defn}\label{def:cobgroupKk}
    Let $\mcal L =\{(L,\alpha,P,\eta)\}$ be the collection of weakly-exact Lagrangian $G$-brane, and $\mcal L_{cob} =\{(V,\alpha,P,\eta)\}$ be the collection of  Lagrangian cobordisms $V\subset \mcal K \x \C$ carrying a $G$-brane structure,  all of whose ends are  weakly-exact, and such that their pullback to the cover $T^4 \x \C \to \mcal K \x \C$ is tautologically unobstructed.
    The {\it Lagrangian cobordism group} of $\mcal K$ is the group
    $$
    \Cob(\mcal K) := \Cob(\mcal K;\mcal L,\mcal L_{cob}) / \sim,
    $$
    where $\sim$ is the equivalence relation generated by expressions of the form 
    $$
    (L,\alpha,P,\eta) \sim (L,\alpha, P \otimes \beta, \eta \otimes \beta)
    $$
    for any $\Z_2$-local system $\beta$ on $L$.
\end{defn}
\begin{rmk}
    The weaker condition that the pullback to $T^4 \x \C$ (instead of the cobordism itself) is tautologically unobstructed is explained further in \Cref{sec:isoK}.
\end{rmk}
\begin{rmk}
    Let us briefly justify the additional relation imposed in \Cref{def:cobgroupKk}.
    Given any Lagrangian brane $(L,\alpha, P, \eta)$, one can tensor both the $Pin_2$-bundle $P$ and the local system with a $\Z_2$-local system $\beta$ to obtain a new Lagrangian brane $(L,\alpha,P\otimes\beta,\eta\otimes\beta)$. These Lagrangian branes are naturally isomorphic in the Fukaya category, whereas there is no reason to expect they should be related by a Lagrangian cobordism.
    Hence, to hope for an isomorphism between the cobordism group and the Grothendieck group of the Fukaya category, we impose this additional relation in the cobordism group.
\end{rmk}

We now introduce the first invariant of cobordism classes.
In  \Cref{rmk:relativeorientations} and \Cref{rmk:branechoiceforfibration} we argued that every graded Lagrangian in $\mcal K$ is oriented with respect to the local system $\pi^*(\det_\Z K)$.
Since cobordisms are also graded and compatible with restriction of gradings, there is a  natural \emph{cycle class map} 
$$
cyc:\Cob(\mcal K) \to H_2(\mcal K;\xi)
$$
that assigns to each Lagrangian its twisted fundamental class.
This map is compatible with the \emph{open-closed} map $\mcal{OC}: HH_0(\mcal Fuk(\mcal K)) \to H_2(\mcal K;\xi)$,\footnote{This version of the open-closed map (mapping to twisted homology) has not appeared in the literature before. It is not needed for our proofs and we just include it for context.} in the sense that there is a commutative diagram
$$
    \xymatrix{
                \Cob(\mcal K)\ar[r] \ar[drr]_{cyc} & K_0(\mcal Fuk(\mcal K))  \ar[r]^{\mcal T} & HH_0(\mcal Fuk(\mcal K)) \ar[d]^{\mcal {OC}}\\
                    &  & H_2(\mcal K;\xi)
            }.
$$
Here $\mcal T$ denotes the Dennis trace map, which exists for any triangulated category and sends an object to its identity endomorphism, and the map $\Cob(\mcal K) \to K_0(\mcal Fuk(\mcal K))$ is a Biran-Cornea type map as in \Cref{eq:BCmap} which is further discussed in \Cref{sec:isoK}.

\subsection{Tropical cobordism group and the refined cycle class map}\label{sec:tropcobandrefinedcyc}
The computation in \Cref{ap:homology} shows that $H_2(\mcal K;\xi) \cong \Z^4 \oplus \Z_2$ and reflects an interesting phenomenon. 
We showed in \Cref{sec:hmsbielliptic} that $\mcal K$ is mirror to an algebraic bielliptic surface over the Novikov field. 
The (co)homology of bielliptic surfaces has been studied (see for instance \cite{serrano1990divisors} for the complex case), and one finds it does {\it not} match our computation of $H_2(\mcal K;\xi)$ (it does up to torsion). 
Namely, for an algebraic bielliptic surface of type I one finds $H^{\text{even}}(Y) \cong \Z^4 \oplus \Z_2^2$---an extra $\Z_2$ factor appearing with respect to $H_2(\mcal K;\xi)$.
Moreover, in \cite{bergstrom2019brauer} the authors exhibit explicit generators for the torsion part of $H^{\text{even}}(Y)$, and these happen to be (linear combinations of) lifts of the tropical divisors $\set{y=0}$ and $\set{y=1/2}$.
By mirror symmetry, these torsion divisors correspond to Lagrangian lifts equipped with torsion local systems.
These Lagrangians have the same underlying Lagrangian submanifold, hence are not distinguished by the naive map $\Cob(\mcal K) \to H_2(\mcal K;\xi)$. 
We will therefore use a refined cycle class map that distinguishes these two Lagrangians, just as the mirror cycle class map distinguishes the corresponding divisors.

Although it is possible to define a refinement of $cyc$ on the whole cobordism group, for convenience we will consider a subgroup $\Cob^{trop}(\mcal K) \subset \Cob(\mcal K)$ and define the refinement only on this subgroup.
Note that, after choosing a Pin structure on $K$, Lagrangian fibers and sections of the projection $\mcal K \to K$ come equipped with a preferred Pin structure: for  a section $\Gamma$ there is a natural diffeomorphism $\Gamma \cong K$ given by the projection, whereas the tangent space to the fibers is naturally isomorphic to the (co)tangent space of $K$.
Furthermore, fibers and sections are weakly exact, since their pullback to the universal cover $\R^2 \to \mcal K$ are exact Lagrangians.\footnote{In fact, for $L$  a fiber or a section we get an even stronger property, sometimes referred to as \emph{topological unobstructedness}: the relative homotopy groups $\pi_2(\mcal K,L)$ vanish.}
To summarize, fibers and sections are part of $\mcal L$, and they come with a preferred Pin structure.
\begin{defn}
    The \emph{tropical cobordism group} is the subgroup $\Cob^{trop}(\mcal K) \subset \Cob(\mcal K)$ generated by Lagrangian fibers and Lagrangian sections of $\pi:\mcal K \to K$ equipped with the preferred Pin structure.
\end{defn}

We now introduce the following construction to refine the cycle class map.
Let $\iota: L \into \mcal K$ be a Lagrangian brane (in particular, $L$ is oriented with respect to $\xi = \pi^*(\wedge^2 T_\Z T^*K)$ by \Cref{rmk:relativeorientations}).
Recall that the twisted version of Poincaré duality gives an isomorphism $H_i(M;\zeta) \cong H^{n-i}(M;\zeta^{\vee}\otimes or_M)$ for any local system $\zeta$ on a manifold $M$.
We consider the composition
$$
H_3(\mcal K;\xi) \cong H^1(\mcal K;\xi^{\vee})\xra{i^*} H^1(L;\xi^{\vee}) \cong H^1(L;or_L),
$$
where in the first isomorphism we have used that $or_{\mcal K} \cong \Z$ (as for any symplectic manifold), the second isomorphism follows from $L$ being $\xi$-oriented and the third isomorphism is twisted Poincaré duality for $L$.
The above composition gives an assignment 
\begin{align*}
    H_3(\mcal K;\xi) &\to H_1(L;\Z)\\
    C & \mapsto C_L.
\end{align*}
Now suppose $L$ carries a $G$-local system, i.e. a homomorphism $\eta: \pi_1(L) \to G$.
\begin{lem}
    There is a well-defined pairing
    \begin{align}\label{eq:collectingmonodromy}
        \begin{split}
          \Cob^{trop}(\mcal K)\otimes H_3(\mcal K;\xi) &\to G\\
        (L,\eta)\otimes C &\mapsto \eta(C_L)  
        \end{split} 
    \end{align}
    where $\eta(C_L)$ is the product of the monodromies of the local system $\eta$ around the connected components of the $1$-cycle $C_L \subset L$.
\end{lem}
\begin{proof}
    We show that the map does not depend on the cobordism class.
    First note that the extra relation imposed in \Cref{def:cobgroupKk} is trivial on the subgroup $\Cob^{trop}(\mcal K)$, so that compatibility with this relation is immediate.
    Let $(L_i,\eta_i), i = 1,\dots,k$ be the ends of a cobordism $(V,\eta_V)$ and let $C \in H_3(\mcal K;\xi)$.
    Consider the transverse intersection $(C\x\C ) \pitchfork V \subset \mcal K \x \C$. 
    This is a $2$-cycle with boundary $C_{L_1} \cup \dots \cup C_{L_k}$, and we equip it with a local system given by restricting $\eta_V$.
    The product of the monodromies along the boundary of this cycle must vanish, and since $\eta_V\restr{L_i} = \eta_i$ we get
    $$
    1 = \prod_i \eta_V(C_{L_i}) = \prod_i \eta_i(C_{L_i}).
    $$
    This proves that the map (\ref{eq:collectingmonodromy}) does not depend on a representative of the cobordism class, hence it is well-defined.
\end{proof}

Denote by $H_3(\mcal K;\xi)_{(2)} \subset H_3(\mcal K;\xi)$ the  subgroup of $2$-torsion elements.
A similar computation to that of \Cref{ap:homology} shows that $H_3(\mcal K;\xi)_{(2)} \cong \Z_2$.
Restricting the map (\ref{eq:collectingmonodromy}) to $H_3(\mcal K,\xi)_{(2)}$, we obtain a map 
\begin{equation}\label{eq:recordingtwotorsionmonodromy}
    \Psi: \Cob^{trop}(\mcal K) \to \Hom(H_3(\mcal K;\xi)_{(2)}, G) \cong G_{(2)}.
\end{equation}

\begin{defn}\label{def:refinedcyc}
    We call the map 
    $$
    \widetilde{cyc} = cyc \oplus \Psi : \Cob^{trop}(\mcal K) \to H_2(\mcal K;\xi) \oplus G_{(2)}$$
    the {\it refined cycle class map}.
\end{defn}
\begin{rmk}
    Intersecting Lagrangians with other classes in $H_3(\mcal K;\xi)$ records monodromies of the local systems around different $1$-cycles.
    These will be detected by an Albanese-type map defined in \Cref{sec:symplecticalbanese}, hence we forget them here.
\end{rmk}
Denote by  $\Cob^{trop}(\mcal K)_{0} \subset \Cob^{trop}(\mcal K)$ the kernel of $\widetilde{cyc}$.
It consists of the following Lagrangians:

\begin{prop}\label{cor:generatorscobhom}
    The subgroup $\Cob^{trop}(\mcal K)_{0} \subset \Cob^{trop}(\mcal K)$ is generated by Lagrangian branes supported on the differences $F_{b^+} - F_{b^-}$ and $L_{f_{m^+\theta^+}} - L_{f_{m^-\theta^-}}$.
    Here, all Lagrangians are equipped with the preferred Pin structure.\footnote{We already argued that fibers and sections have a preferred Pin structure. Given a cobordism between $\Gamma_0,\Gamma_{m\theta}$ and $L_{f_{m\theta}}$ together with a Pin structure extending the preferred one on $\Gamma_0$ and $\Gamma_{m\theta}$, one can simply define the preferred Pin structure on $L_{f_{m\theta}}$ to be that induced by the cobordism. The fact that the surgery cobordism admits a unique Pin structure restricting to the preferred one on $\Gamma_0$ and $\Gamma_{m\theta}$ is proved in \Cref{lem:matchingPin}.}
\end{prop}
\begin{proof}
    We first analyze the kernel of $cyc : \Cob^{trop}(\mcal K) \to H_2(\mcal K;\xi)$.
    By \Cref{lem:homologyofsections}, a general element of $\ker(cyc)$ will be a finite sum of the form
    \begin{equation}\label{eq:kercyc}
        \sum_{i=1}^{N^+} F_{b^+_i} 
        - \sum_{i=1}^{N^-} F_{b^-_i}
        + \sum_{(m_i^\pm,n_i^\pm,l_i^\pm\theta_i^\pm)}(\Gamma_{m_i^+\theta_i^+}^{n_i^+[l_i^+]}
        - \Gamma_{m_i^-\theta_i^-}^{n_i^-[l_i^-]})
    \end{equation}
    for some $b_i^{\pm}\in K$ and $(m_i^\pm,n_i^\pm,l_i^\pm,\theta_i^\pm)\in \Z^2\oplus\Z_2\oplus S^1$ satisfying:
    \begin{align}
        \sum(m_i^+n_i^+ - m_i^-n_i^-)&=N^+ - N^-.\label{eq:sumfibers}\\
        \sum(m_i^+-m_i^-)&=0 \label{eq:summ}\\ 
        \sum(n_i^+ - n_i^-) & = 0 \label{eq:sumn}\\
        \sum(l_i^+ -l_i^-)&=0 \quad (\text{mod } 2)\label{eq:sumnl} 
    \end{align}

    We want to apply cobordism relations as those in \Cref{lem:cobsectionstogenerators}.
    However, recall that this Lemma was proved for undecorated cobordisms (as we were only interested in the homology classes of the underlying Lagrangians), whereas relations in $\Cob^{trop}(\mcal K)$ involve cobordisms carrying $G$-brane structures and whose pullback to $T^4 \x \C$ is tautologically unobstructed (cf. \Cref{def:cobgroupKk}). 
    Once we equip our Lagrangians with this extra structure---and fix it on the end $\Gamma^{n[l]}_{m\theta}$---there might be some obstruction to the existence of a Lagrangian cobordism (recall the Lagrangian cobordism itself must carry such extra structure, and that the restriction to the ends of the cobordism must agree with the extra structure specified on the Lagrangians). 
    The key to prove \Cref{cor:generatorscobhom} will then be to show that, given a grading, the preferred Pin structure and a local system on $\Gamma^{n[l]}_{m\theta}$, the cobordism relations of \Cref{lem:cobsectionstogenerators} hold in $\Cob(\mcal K)$.
    
    In \Cref{ap:compatibility} we show that \Cref{lem:cobsectionstogenerators} can be enhanced to take into account this extra structure.
    Namely, Lemmas \ref{lem:gradingfiberwisesum}-\ref{lem:pinfiberwisesum} show that, given any grading on $\Gamma^{n[l]}_{m\theta}$, there are gradings on $\Gamma_{m\theta}$ and $\Gamma^{n[l]}$ such that \Cref{eq:tensorofsections} holds when we put the preferred Pin structure on all three sections, whereas \Cref{prop:admissibilitysurgery} shows that the surgery cobordism of \Cref{eq:jeffsurgerysections} is tautologically unobstructed when pulled back to $T^4\x\C$ and can be equipped with a $G$-brane structure restricting appropriately to the ends.
The upshot is that, given a $G$-brane structure on $\Gamma^{n[l]}_{m\theta}$, \Cref{lem:cobsectionstogenerators} holds when we equip all the Lagrangians in the right-hand side with their natural Pin structures.
    Furthermore, one can choose the local system on $L_{f^{nl}}$ to be trivial (see \Cref{lem:matchinglocsyst}).
    
    With this in mind, one can conclude as follows.
    Apply the cobordism relation of \Cref{lem:cobsectionstogenerators} to all the sections in \Cref{eq:kercyc},  and note that, since the local system and Pin structure on the  Lagrangians $L_{f^{n[l]}}$ are all the same, \Cref{eq:sumnl} translates into an algebraic cancellation of these Lagrangians.
    Hence any null-homologous combination as in \Cref{eq:kercyc} consists of fibers,  tropical Lagrangians $L_{f_{m\theta}}$ plus a cobordism class  of the form
    $$
        \mathbb L= \sum_i ((\Gamma_0,\eta^+_{i\Z_2})- (\Gamma_0,\eta^-_{i\Z_2 })).
    $$
    Now note that $\ker (\widetilde{cyc}) = \ker (cyc) \cap \ker (\Psi)$, where $\Psi$ is the map defined in \Cref{eq:recordingtwotorsionmonodromy}.
    We have that $\Psi(\mathbb L) = 0$ if and only if $\prod_i \eta^+_{i\Z_2} = \prod_i \eta^-_{i\Z_2}$.
    Under this condition, one can iteratively surger copies of $\Gamma_0$ with $L_{f^{10}}$ to obtain a Lagrangian $L := \Gamma_0 \#\dots \# \Gamma_0 \# L_{f^{10}}$ and a relation
    \begin{align*}
        \sum_i (\Gamma_0,\eta^+_{i\Z_2}) + (L_{f^{10}},\mathbf 1) 
        &\sim
        (L, \prod_i \eta^+_{i\Z_2})\\
        &=
        (L, \prod_i \eta^-_{i\Z_2})\\
        &\sim
        \sum_i (\Gamma_0,\eta^-_{i\Z_2}) + (L_{f^{10}},\mathbf 1).
    \end{align*}
    This is equivalent to $\mathbb L \sim0$, which concludes the proof.
\end{proof}

\subsection{Symplectic Albanese map}\label{sec:symplecticalbanese}
We will need one more tool to prove \Cref{th:theoremA}, namely a modification of an Abel-Jacobi type map introduced by  Sheridan-Smith in \cite[Lemma 2.10]{sheridan2021lagrangian}.
They define, for any tropical affine torus $B$, a map 
\begin{align}\label{eq:sheridansmithmap}
    \begin{split}
        \alb_{fib}:\Cob_{fib}(X(B))_{\hom} &\to \Alb(B)\\
        F_{b^+} - F_{b^-} &\mapsto \alb(b^+ - b^-)
    \end{split}
\end{align}
by looking at the flux swept by a homologically trivial collection of fibers.
Here, $\alb$ is the map of \Cref{eq:albz0}; $\Cob_{fib}(X(B))$ is a cobordism group generated by fibers of the projection $X(B) \to B$ and whose relations come from cobordism all of whose ends are fibers; and $\Cob_{fib}(X)_{\hom}$ denotes the kernel of the degree map $\Cob_{fib}(X) \to \Z$. 
The following Proposition provides two generalizations: it shows their map is also well-defined for a symplectic bielliptic surface (which fibers over a \emph{quotient} of a tropical affine torus), and it extends their map to the whole kernel of the cycle class map (not just the part generated by fibers).
To state it, let us recall from \Cref{cor:generatorscobhom} that $\Cob(\mcal K)_0$ is generated by Lagrangians of the form $F_{b^+} - F_{b^-}$ and $L_{f_{m^+\theta^+}} - L_{f_{m^-\theta^-}}$.

\begin{prop}\label{lem:extensionalbanesemap}
    There is a well-defined map
    \begin{align}\label{eq:albanesebielliptic}
        \begin{split}
        \alb:\Cob^{trop}(\mcal K)_{0} &\to \Alb(K)\\
        \sum_i(F_{b_i^+} - F_{b_i^-}) &\mapsto \alb(\sum(b_i^+ - b_i^-))\\
        \sum_i(L_{f_{m_i^+\theta_i^+}} - L_{f_{m_i^-\theta_i^-}}) &\mapsto 0.
        \end{split}
    \end{align}
\end{prop}

\begin{proof}
    We first recall the basic idea of the proof of \cite[Lemma 2.10]{sheridan2021lagrangian} that proves the existence of the map (\ref{eq:sheridansmithmap}).
    For such a map to be well-defined, one must show that if $\set{F_{b_i}-F_{b_i'}}$ are the ends of a Lagrangian cobordism $V$, then the homomorphism 
    $$
    \alb\left(\sum(b_i^+ - b_i^-)\right): H^0(T^*_\Z K) \to \R
    $$
    coincides with integration over some class $\beta \in H_1(B;\Z)$. 
    Sheridan-Smith construct such class by `closing up' the Lagrangian cobordism to an $(n+1)$-cycle in $X(B) \x \C$ (attaching cylinders fibering over paths from $b_i$ to $b_i'$), and then projecting the homology class of this cycle to $H_1(B)$ via the K\"unneth map $H_{n+1}(X(B)) \to H_1(B)\otimes H_n(F) \cong H_1(B)$.
    An application of Stokes' theorem  and the fact that $V$ is Lagrangian then yield the result.

    For the map (\ref{eq:albanesebielliptic}) to be well-defined, we must show that if $V \subset \mcal K \x \C$ is a cobordism with ends $\mathbb F = \{F_{b_i^+} - F_{b_i^-}\}$ and $\mathbb L = \{L_{f_{m^+\theta^+}} - L_{f_{m^-\theta^-}}\}$, then $\alb(\sum(b_i^+ - b_i^-)) = 0$.
    In other words, there exists a class $\beta \in H_1(K;\Z)$ such that $\alb(\sum(b_i^+ - b_i^-)) = \int_\beta$. 
    We first close up the cobordism to obtain a ($\xi$-oriented) $3$-cycle.
    Choose paths in $K$ from $b_i^+$ to $b_i^-$; taking their  preimage we get a $3$-chain $W_1\subset \mcal K$  whose boundary is $\mathbb F$.
    On the other hand, the collection $\mathbb L$ is a linear combination of Lagrangian lifts of tropical circles $\set{x = \theta}$, and since $\mathbb L$ is homologically trivial there are as many positive as negative ones. 
    We group them in pairs $\theta_i^\pm$ and
    choose isotopies $\nu_i^s = \set{x = \theta_i^s}$ between the curves $\set{x = \theta_i^+}$ and $\set{x = \theta_i^-}$.
    Writing $W_2^i = \cup_s L_{\eta^i_s}$ for the union of the Lagrangian lifts of the curves $\eta^s_i$, we obtain a $3$-chain $W_2 = \cup_i W_2^i$ whose boundary is $\mathbb L$.
    The union $V + W_1 + W_2$ is a $\xi$-oriented $3$-cycle.
    We will use this $3$-cycle to show that the extension by zero of \Cref{eq:albanesebielliptic} gives a well-defined map $\Cob^{trop}(\mcal K)_{0} \to \Alb(K)$.
    
    Let $\alpha \in H^0(T^*_\Z K)$ be a globally defined tropical $1$-form, which corresponds to some global section of the local system given by the homology of the fibers. 
    Under this correspondence, the value of $\alb(\sum_i(b_i^+ - b_i^-))(\alpha)$ can be computed as the symplectic area of the $2$-chain $W_1 \pitchfork ([K] \x \alpha \x \C)$ (note this is now a chain with $\Z$-coefficients).
    Here, we use the notation $[K] \x \alpha \in H_3(\mcal K;\xi)$ to denote the submanifold of $\mcal K$ obtained as the product of the $0$-section and a geodesic representative of $\alpha$ passing through the origin in the fiber; this is a trivial $S^1$-bundle over the $0$-section, hence a $\xi$-cycle.
    We make the following observation:
    \begin{center}
         \emph{The only global tropical $1$-form in $K$ are $\alpha = dx$ and its multiples, and the corresponding homology class in the fibers intersects $L_{\eta_s^i}$ along the $0$-section. Hence $\omega_{ \mcal K \x \C}(W_2 \pitchfork ( [K] \x \alpha \x \C)) = 0$.}
    \end{center}
    With these in mind, one can compute as in \cite[Lemma 2.10]{sheridan2021lagrangian} to obtain
    $$
    \alb(\sum_i(b_i^+ - b_i^-))(\alpha) = \int_{(\pi_{\mcal K})_*(V + W_1 + W_2)\pitchfork ([K]\x\alpha)} \omega_{\mcal K}
    $$
    (the key being that the extra $2$-chain $W_2 \pitchfork ( [K] \x \alpha\x \C)$ that appears with respect to their proof is contained in the zero-section, which is Lagrangian). 
    
    The last step is to obtain a $1$-cycle $\beta\in H_1(K;\Z)$ from the $3$-cycle $(\pi_{\mcal K})_*(V + W_1 + W_2)\in H_3(\mcal K;\xi)$.
    This is not as straightforward as in \cite[Lemma 2.10]{sheridan2021lagrangian}, since (i) the cycle lives in $H_2(\mcal K;\xi)$ (instead of $H_2(\mcal K;\Z)$) and (ii) the fibration is not trivial, hence there is not a K\"unneth decomposition or a projection to $H_1(K;\Z)$.
    Nonetheless, by analyzing the Leray spectral sequence $E^2_{pq} = H_p(K;\mcal H_q(F;\Z)\otimes \xi) \Rightarrow H_{p+q}(\mcal K; \xi)$ we show in \Cref{lem:projectionH3H1} that there is a projection map $pr:H_3(\mcal K;\xi) \to H_1(K;\Z)$.
    We define the cycle 
    $$
    \beta := pr((\pi_{\mcal K})_* (V+ W_1 + W_2))     \in H_1(K; \Z). 
    $$

Lastly, consider the SES
    $$
    0 \to H_0(K; \mcal H_2(F)) \to H_2(\mcal K;\Z) \to H_1(K;\mcal H_1(F)) \to 0
    $$
    coming from the Leray spectral sequence.
    It shows that the class $(\pi_{\mcal K})_* (V+ W_1 + W_2) \pitchfork ([K]\x\alpha) \in H_2(\mcal K;\Z)$ is of the form
    $$
    (\pi_{\mcal K})_* (V+ W_1 + W_2) \pitchfork ([K]\x\alpha) = \beta \otimes \alpha + \gamma,
    $$
    where $\gamma \in H_0(K;\mcal H_2(F)) \subset H_2(\mcal K;\Z)$. 
    The class $\gamma$ is Lagrangian (it is supported on a fiber), hence one concludes
    $$
    \alb(\sum_i(b_i^+ - b_i^-))(\alpha) = \int_{(\pi_{\mcal K})_*(V + W_1 + W_2)\pitchfork ([K]\x\alpha)} \omega_{\mcal K}
    = \int_{\beta\x\alpha}\omega_{\mcal K}
    = \int_\beta \alpha,
    $$
    completing the proof.
\end{proof}

\begin{lem}\label{lem:projectionH3H1}
    There is a projection map $H_3(\mcal K;\xi) \to H_1(K;\Z)$.
\end{lem}
\begin{proof}
    Let us consider the Leray spectral sequence
    $$
    E^2_{pq} = H_p(K;\mcal H_q(F;\Z)\otimes \xi) \Rightarrow H_{p+q}(\mcal K; \xi).
    $$
    For degree reasons one has $E^\infty_{11} \cong E^2_{11} = H_1(K;\mcal H_1(F;\Z)\otimes \xi)$, hence there is a short exact sequence
    \begin{equation}\label{eq:sesH3}
        0 \to H_1(K;\mcal H_2(F;\Z)\otimes \xi) \to H_3(\mcal K;\xi) \to E^\infty_{21} \to 0.
    \end{equation}
    It follows from the computations in \Cref{ap:homology} that $\im(E^2_{21} \to E^2_{02}) = 0$, hence  $E^\infty_{21} \cong E^2_{21} = H_2(K;\mcal H_1(F)\otimes \xi)$.
    This group is isomorphic to $\Z$ and generated by $[K]\x \alpha$, thus \Cref{eq:sesH3} splits canonically giving a decomposition 
    $$
    H_3(\mcal K;\xi) \cong H_2(K;\mcal H_1(F;\Z)\otimes \xi) \oplus H_1(K;\mcal H_2(F;\Z)\otimes \xi).
    $$
    The geometry of Lagrangian torus fibrations shows that $\mcal H_2(F;\Z) \cong \wedge^2 T^*_\Z K = \xi^\vee$, and hence the local system $\mcal H_2(F;\Z) \otimes \xi$ is trivial.
    Putting everything together, we get the desired projection $pr:H_3(\mcal K;\xi) \to H_1(K;\Z)$.
\end{proof}
As mentioned in \cite[Remark 2.11]{sheridan2021lagrangian}, there is an extension of the fibered albanese map that incorporates the local systems on the Lagrangians.
In the present case, if $\lambda = \lambda_x \otimes \lambda_y$ is a local system on $F$ with monodromies $\lambda_x,\lambda_y \in G$ along the homology classes corresponding to $dx$ and $dy$, then this map takes the form
\begin{align}\label{eq:alblocsys}
    \begin{split}
        \alb^{loc}_{fib}:\Cob_{fib}(\mcal K)_{\hom} &\to \Alb(K) \oplus G\\
        \sum ((F_i^+,\lambda_i^+) - (F_i^-,\lambda_i^-)) & \mapsto (\alb(\sum_i(F_i^+ - F_i^-)), \prod_i\lambda_{i,x}^+ (\lambda_{i,x}^-)^{-1}).
    \end{split}
\end{align}
The fact that this map is well-defined follows from a similar argument to that of \Cref{lem:extensionalbanesemap}.
Namely, given $(V,\lambda)$ a cobordism with local system between a linear combination of fibers, its intersection with the $5$-cycle $K \x [\alpha] \x \C$ is a $2$-dimensional surface with boundary inside $V$.
It follows that the product of the monodromies around the boundary components of this surface of any local system must vanish.
In particular, for the local system $\lambda$---which restricts to $\lambda_i^\pm$ on the ends of $V$---we have that $\prod_i\lambda_{i,x}^+ (\lambda_{i,x}^-)^{-1} = 1$.

\begin{prop}\label{prop:fiberedalbaneseiso}
    The fibered albanese map $\alb_{fib}^{loc}:\Cob_{fib}(\mcal K)_{\hom} \to \Alb(K) \oplus G$ is an isomorphism.
\end{prop}
\begin{proof}
    The map is clearly surjective (recall $H^0(T^*_\Z K) \cong \Z \vspan{dx}$). 
To prove injectivity we proceed as follows.
    Denote by $pr: K \to S^1$ the projection to the line $\{y = 1/2\}$. 
    Consider a Lagrangian fiber with local system $(F_p, \lambda = \lambda_{x} \otimes \lambda_{y})$.
    We will show that
    \begin{equation}\label{eq:cobordismtozerosection}
        (F_p,\lambda) \sim (F_{pr(p)},\lambda_x).
    \end{equation}
    Given this, injectivity follows from injectivity of $\alb_{fib}^{loc}: \Cob_{fib}(T^2)_{\hom} \to S^1 \x G$ (a consequence of \Cref{cor:cobt2locsys}).

    Write $\pi_x:K \to S^1$ and $\pi_y:K \to I = [0,1]$ for the projections to the (equivalence class of the) $x$ and $y$ coordinates of $K$.
    Note both of these maps give {\it tropical} $S^1$-fibrations of $K$, meaning the fibers are embedded tropical circles.
    Let $\theta\in S^1$ and $t \in I$ be the unique points with $p \in \pi\inv_x(\theta) \cap \pi_y\inv (t)$.
    
    We first claim that if $\iota: K \to K$ is the map induced by $(x,y) \mapsto (x,-y)$, then the cobordism class 
    $$
    \mathbb{F}_1 = (F_p,\lambda_x\otimes \lambda_y) - (F_{\iota(p)},\lambda_x \otimes \lambda_y\inv)
    $$
    is $2$-torsion.
    Indeed, let $p: T^2 \x T^2 \to \mcal K$ be the double cover of $\mcal K$ (cf. \Cref{lem:biellipticisquotientoftori}) and write $T^2 \x T^2 = (\R /\Z \x \R / l_1\Z) \x (\R/ \Z \x \R / l_2 \Z)$.
    We denote by $F^i_t$ the fibers of the projection $\R / \Z \x \R/  l_i\Z \to \R/\Z$.
    Then the pull back of the class $\mathbb F_1$ under $p$ is a linear combination of Lagrangian torus fibers of the form (see \Cref{fig:relationsCobfib}):
    $$
        p\inv(\mathbb F_1) = (F^1_\theta - F^1_{\theta + 1/2}) \x (F^2_t + F^2_{1-t}).
    $$
    Now note that the linear combination $F^1_\theta - F^1_{\theta + 1/2}$ is a null-homologous combination of fibers in $\Cob(T^2)$, and that its image under the map $\Phi$ of \Cref{lem:fluxintorus} is $2$-torsion.
    It then follows from \Cref{prop:fluxisaniso} that this class is $2$-torsion in $\Cob(T^2)$, and hence---by crossing the cobordism with the constant Lagrangian $F^2_t + F^2_{1-t}$ in the other $T^2$-factor---that $p\inv(\mathbb F_1)$ is $2$-torsion in $\Cob(T^2 \x T^2)$.
    The cobordism is furthermore $p$-invariant, hence descends to give the relation $2 \mathbb{F}_1 \sim 0 \in \Cob(\mcal K)$.

    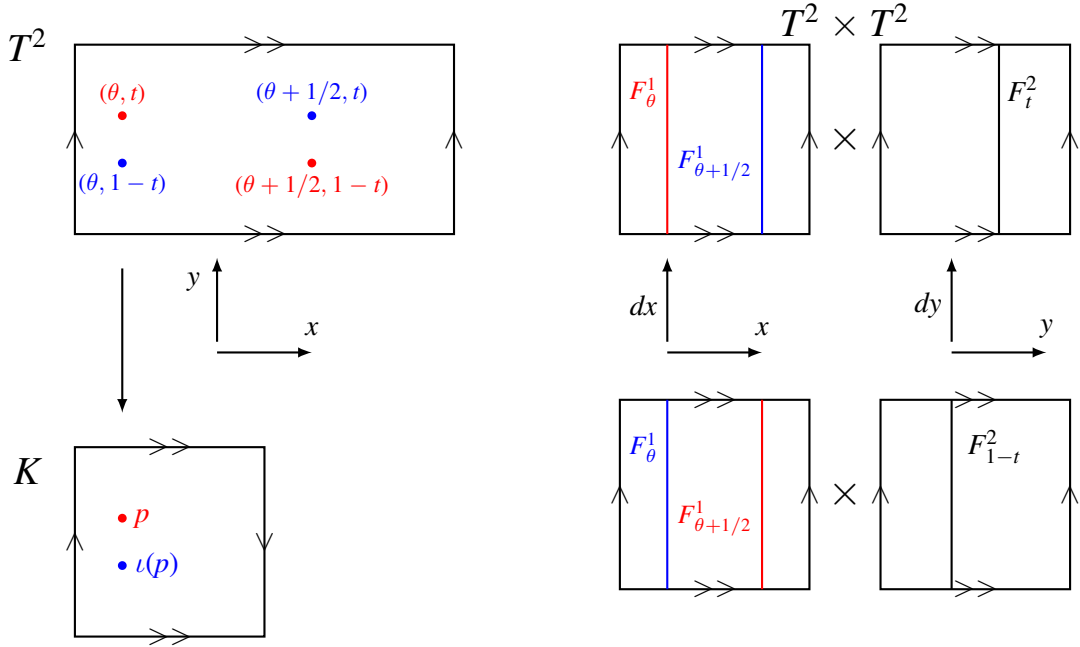
\begin{figure}
        \centering
            \begin{tikzpicture}[scale = 0.63, thick]
\begin{scope}[shift={(-0.5,-10)}]
                    \draw  (-3.5,9) rectangle (0.5,5);
                    \node at (-1.5,9) {$>>$};
                    \node at (-1.5,5) {$>>$};
                    \node[rotate = 90] at (-3.5,7) {$>$};
                    \node[rotate = 90] at (0.5,7) {$<$};
                \end{scope}
\begin{scope}[shift={(-0.5,-10)}]
                    \filldraw[red]  (-2.5,7.5)  circle (2pt) node [anchor=west] {$p$};
                    \filldraw[blue]  (-2.5,6.5)  circle (2pt) node [anchor=west] {$\iota(p)$};
                \end{scope}
                
\node (v1) at (-3,3) {};
                \node (v2) at (-3,-0.5) {};
                \draw [-latex] (v1) edge (v2);
                
\begin{scope}[shift={(3,-1.5)}, xscale=2]
                    \draw  (-3.5,9) rectangle (0.5,5);
                    \node at (-1.5,9) {$>>$};
                    \node at (-1.5,5) {$>>$};
                    \node[rotate = 90] at (-3.5,7) {$>$};
                    \node[rotate = 90] at (0.5,7) {$>$};
                \end{scope}
\begin{scope}[shift={(-0.5,-1.5)}]
                    \filldraw[red]  (-2.5,7.5)  circle (2pt) node [font = \footnotesize, anchor=south] {$(\theta,t)$};
                    \filldraw[blue]  (-2.5,6.5)  circle (2pt) node [font = \footnotesize, anchor=north] {$(\theta,1-t)$};
                \end{scope}
                
                \begin{scope}[shift={(3.5,-1.5)}]
                    \filldraw[blue]  (-2.5,7.5)  circle (2pt) node [font = \footnotesize, anchor=south] {$(\theta+1/2,t)$};
                    \filldraw[red]  (-2.5,6.5)  circle (2pt) node [font = \footnotesize, anchor=north] {$(\theta+1/2,1-t)$};
                \end{scope}

                \node[scale = 1.5, font = \small] at (-5,-1.5) {$K$};
                \node[scale = 1.5, font = \small] at (-5,7.5) {$T^2$};

                \begin{scope}[shift={(1.5,-0.5)}]
                    \draw [-latex](-2.5,1.5) node (v8) {} -- (-0.5,1.5);
                    \draw [-latex](v8) -- (-2.5,3.5);
                    \node at (-0.5,2) {$x$};
                    \node at (-3,3) {$y$};
                \end{scope}

\begin{scope}[shift={(-2,0)}]
                    \begin{scope}[]
                        \begin{scope}[shift={(13,-1.5)}]
                            \draw  (-3.5,9) rectangle (0.5,5);
                            \node at (-1.5,9) {$>>$};
                            \node at (-1.5,5) {$>>$};
                            \node[rotate = 90] at (-3.5,7) {$>$};
                            \node[rotate = 90] at (0.5,7) {$>$};
                        \end{scope}
                        
                        \node[scale = 1.5] at (14.25,5.5) {$\times$};
                        
                        \begin{scope}[shift={(18.5,-1.5)}]
                            \draw  (-3.5,9) rectangle (0.5,5);
                            \node at (-1.5,9) {$>>$};
                            \node at (-1.5,5) {$>>$};
                            \node[rotate = 90] at (-3.5,7) {$>$};
                            \node[rotate = 90] at (0.5,7) {$>$};
                        \end{scope}
                        
\begin{scope}[]
                            \draw [red](10.5,7.5) -- (10.5,3.5);
                            \node[red, font = \small] at (10,6.5) {$F^1_{\theta}$};
                            \draw [blue](12.5,7.5) -- (12.5,3.5);
                            \node[blue, font = \small] at (11.5,5) {$F^1_{ \theta +1/2}$};
                        \end{scope}
                        
                        \draw(17.5,7.5) -- (17.5,3.5);
                        \node at (18,6.5) {$F^2_{t}$};

\begin{scope}[shift={(0,-7.5)}]
                            \begin{scope}[shift={(13,-1.5)}]
                                \draw  (-3.5,9) rectangle (0.5,5);
                                \node at (-1.5,9) {$>>$};
                                \node at (-1.5,5) {$>>$};
                                \node[rotate = 90] at (-3.5,7) {$>$};
                                \node[rotate = 90] at (0.5,7) {$>$};
                            \end{scope}
                        
                        \node[scale = 1.5] at (14.25,5.5) {$\times$};
                        
                        \begin{scope}[shift={(18.5,-1.5)}]
                            \draw  (-3.5,9) rectangle (0.5,5);
                            \node at (-1.5,9) {$>>$};
                            \node at (-1.5,5) {$>>$};
                            \node[rotate = 90] at (-3.5,7) {$>$};
                            \node[rotate = 90] at (0.5,7) {$>$};
                        \end{scope}
                        
\begin{scope}[]
                            \draw [blue](10.5,7.5) -- (10.5,3.5);
                            \node[blue, font = \small] at (10,6.5) {$F^1_\theta$};
                            \draw [red](12.5,7.5) -- (12.5,3.5);
                            \node[red, font = \small] at (11.5,5) {$F^1_{\theta +1/2}$};
                        \end{scope}
                        
                        \draw(16.5,7.5) -- (16.5,3.5);
                        \node at (17.4,6.5) {$F^2_ {1- t}$};
                        \end{scope}

                        \begin{scope}[shift={(0,-0.5)}]
                            \draw [-latex](10.5,1.5) node (v3) {} -- (12.5,1.5);
                            \draw [-latex](v3) -- (10.5,3.5);
                            \node at (12.5,2) {$x$};
                            \node at (10,2.5) {$dx$};
                        \end{scope}

                        \begin{scope}[shift={(6,-0.5)}]
                            \draw [-latex](10.5,1.5) node (v3) {} -- (12.5,1.5);
                            \draw [-latex](v3) -- (10.5,3.5);
                            \node at (12.5,2) {$y$};
                            \node at (10,2.5) {$dy$};
                        \end{scope}
                    \end{scope}

                    \node[scale = 1.5, font = \small] at (14.25,8) {$T^2 \times T^2$};
                \end{scope}
        \end{tikzpicture}
        \caption{Schematic picture to show that $2\mathbb F_1 \sim 0$. The left picture shows the (projection to $K$ of the) class $\mathbb F_1 \in \Cob_{fib}(\mcal K)_{\hom}$ as well as its preimage $p\inv(\mathbb F_1) \in \Cob_{fib}(T^2 \x T^2)_{\hom}$. The right picture shows the decomposition of  $p\inv(\mathbb F_1)$ in the product $T^2 \x T^2$; namely, the top part shows the class $(F_\theta^1 - F_{\theta + 1/2}^1)\x F^2_t$, whereas the bottom part shows $(F_\theta^1 - F_{\theta + 1/2}^1)\x F^2_{1-t}$. Note that while the torus in the left picture represents a tropical affine base, the $2$-tori in the right picture represent symplectic tori living above the tropical base (that is, the symplectic manifold $T^2 \x T^2$ in the right picture is the total space of the Lagrangian torus fibration over the tropical $2$-torus $T^2$ in the left picture).}
        \label{fig:relationsCobfib}
    \end{figure}
    Next, we claim that the cobordism class 
    $$
    \mathbb{F}_2 = (F_p,\lambda_x\otimes \lambda_y) - (F_{pr(p)},\lambda_x)
    $$
    is $4$-torsion.
    By the previous claim we have that 
    $$
    4(F_p,\lambda_x\otimes \lambda_y) \sim 2(F_p,\lambda_x\otimes \lambda_y) + 2(F_{\iota(p)},\lambda_x\otimes \lambda_y\inv).
    $$
    As before, we can pull back the class $4\mathbb{F}_2$ to $T^2 \x T^2$ and obtain a linear combination of fibers of the form
    $$
        p\inv(\mathbb F_2) = (F^1_\theta + F^1_{\theta + 1/2}) \x (2F^2_t + 2F^2_{1-t} - 4F^2_{1/2}).
    $$
    The linear combination $2F^2_t + 2F^2_{1-t} - 4F^2_{1/2}$ is null-homologous and maps to $0$ under $\Phi$, hence vanishes in $\Cob(T^2)$ by \Cref{prop:fluxisaniso}.
    Crossing with the constant Lagrangian $F^1_\theta + F^1_{\theta + 1/2}$ in the other $T^2$ factor yields a cobordism in $T^2 \x T^2 \x \C$, and descending it to $\mcal K \x \C$ we get the desired relation $4\mathbb{F}_2 \sim 0 \in \Cob(\mcal K)$.
    
    We have now proved  \Cref{eq:cobordismtozerosection} up to torsion.
    The complete result follows from noting that 
    $$
    (F_p,\lambda_x\otimes \lambda_y)   - (F_{pr(p)},\lambda_x) \sim 4((F_q,\lambda_x\otimes \lambda_y)   - (F_{pr(p)},\lambda_x))
    $$
    for any $q \in \pi_x\inv(\theta)$ such that $4(q-pr(q)) = p -pr(p) \in \Alb(\pi_x\inv(\theta))$ (again, a consequence of \Cref{prop:fluxisaniso}).
\end{proof}

A similar argument to the proof of \Cref{lem:extensionalbanesemap} shows that $\alb^{loc}_{fib}$  extends to the whole tropical cobordism group, giving a map 
\begin{equation}\label{eq:albloc}
    \alb^{loc}: \Cob^{trop}(\mcal K)_0 \to \Alb(K) \oplus G.
\end{equation}

\subsection{Computation and dimensionality}\label{sec:finitedimensionality}
The following is one of the two main results of this paper:

\begin{thm}\label{th:cobK}
The short exact sequence
\begin{equation}\label{eq:cobKses}
0\to \Cob^{trop}(\mcal K)_{0}\to \Cob^{trop}(\mcal K)\xra{\widetilde{cyc}} H_2(\mcal K;\xi) \oplus G_{(2)}\to 0
\end{equation}
is split, and $\Cob^{trop}(\mcal K)_{0} \cong (S^1 \oplus G)^2$.
\end{thm}

\begin{proof}

    Let $\Psi : \mcal K \to \mcal K$ be the  involution swapping the $y$-coordinate on $K$ with its cotangent coordinate (recall $\mcal K = T^\ast K / T^*_\Z K$ and $K$ is as in \Cref{lem:classificationklein}). 
    That is, $\Psi$ is given in local coordinates by
    $$
    (x_1,y_1,x_2,y_2) \mapsto (x_1,y_1,y_2,x_2),
    $$
    where we are identifying $\mcal K$ with the quotient of $(\R^4,dx_1\wdg dy_1 + dx_2 \wdg dy_2)$ by the symplectic transformations
    \begin{align*}
        (x_1,y_1,x_2,y_2) &\mapsto (x_1,y_1,x_2,y_2) + e_i,  \quad i=1,\dots,4\\
        (x_1,y_1,x_2,y_2) &\mapsto  (x_1 + p_1  ,y_1 + p_2,-x_2,-y_2),  \quad 2(p_1,p_2) = 0.
    \end{align*}
    The composition $\pi' = \pi\circ \Psi:\mcal K \to K$ is still a Lagrangian torus fibration; let $K'$ be the Klein bottle whose tropical affine structure is that induced by $\pi'$.
    
    Consider the map $\alb^{loc}: \Cob^{trop}(\mcal K)_{0} \to \Alb(K) \x G$ of \Cref{eq:albloc}. 
    There is an analogous map $\alb'^{,loc}:\Cob^{trop}(\mcal K)_{0} \to \Alb(K') \x G$ arising from the Lagrangian torus fibration $\pi'$.
    We claim the map
    \begin{equation}\label{eq:cobhomisoalb}
        \alb^{loc}\oplus\alb'^{,loc}: \Cob^{trop}(\mcal K)_{0} \to (\Alb(K) \oplus G) \oplus (\Alb(K') \oplus G)
    \end{equation}
    is an isomorphism.
    Indeed, note that $\Psi$ interchanges fibers with the tropical Lagrangians $L_{f_{m\theta}}$, so that $\alb'$ is the map 
    \begin{align*}
        \begin{split}
        \alb':\Cob^{trop}(\mcal K)_{0} &\to \Alb(K')\\
        \sum_i(F_{b_i^+} - F_{b_i^-}) &\mapsto 0\\
        \sum_i(L_{f_{m_i^+\theta_i^+}} - L_{f_{m_i^-\theta_i^-}}) &\mapsto \alb'(p_i^+ - p_i^-)
        \end{split}
    \end{align*}
    where $p_i^\pm = \pi'(L_{f_{m_i^\pm\theta_i^\pm}}) \in K'$.
    The fact that (\ref{eq:cobhomisoalb}) is an isomorphism follows now immediately from \Cref{prop:fiberedalbaneseiso} applied to both $\alb^{loc}$ and $\alb'^{,loc}$.

    To show that \Cref{eq:cobKses} splits, we consider the map $H_2(\mcal K;\xi) \oplus G_{(2)} \cong \Z^4 \oplus \Z_2 \oplus G_{(2)} \to \Cob^{trop}(\mcal K)_0$ that sends an element $(N_f,N_0,m,n,[l],g)$ to the cobordism class
    $$
    N_f F_{(0,1/2)} + N_0 \Gamma_0 + m L_{f_{10}} + n L_{f^{10}} + l L_{f^{01}} +  ((L_{f^{10}},g) - L_{f^{10}}).
    $$
    Here, all the Lagrangians are equipped with the preferred Pin structure and trivial local system, except $(L_{f^{10}},g)$ which carries a local system whose monodromy around the $dy$-coordinate is $g$.
    The argument at the end of \Cref{cor:generatorscobhom} shows that
    $$
    \sum_i((L_{f^{10}},g_i) - L_{f^{10}}) \sim (L_{f^{10}},\prod_i g_i) - L_{f^{10}},
    $$
    hence this map is a group homomorphism, and it is clearly a section by \Cref{ap:homology}.
    It follows that the short exact sequence (\ref{eq:cobKses}) is split.
\end{proof}

Note that, in particular, 
$$
\Cob^{trop}(\mcal K) \cong H_2(\mcal K;\xi) \oplus G_{(2)} \oplus (S^1 \oplus G)^2
$$
is \emph{$2$-dimensional} in the sense of \Cref{def:infinitedimensionality}.
As discussed in the introduction, this should be compared with the work of Sheridan-Smith \cite{sheridan2020rational}, who show that a symplectic $4$-manifold with trivial canonical bundle containing a Lagrangian of genus at least 1 has infinite-dimensional cobordism group.
The crucial difference in \Cref{th:cobK} is that the canonical bundle of a bielliptic surface is torsion but non-trivial, thus a graded Lagrangian is not necessarily oriented (see \Cref{sec:results} for more details).

\begin{rmk}
    One would expect---based on the analogous result for the mirror bielliptic surface---that a computation of the cobordism group of a symplectic bielliptic surface living over $K = K_2$ (see \Cref{rmk:Kleinfamilies}) would follow the same argument.
    However, in the symplectic setting some difficulties arise.
    While we do have a short exact sequence as in \Cref{eq:cobKses}, the definition of the Abel-Jacobi type map $\Cob^{trop}(\mcal K)_0 \to \Alb(K)$ is not so clear in this case. 
\end{rmk}

    \subsection{The Lagrangian cobordism group and the Grothendieck group}\label{sec:isoK}
We now turn to the comparison between the cobordism group $\Cob^{trop}(\mcal K)$ and the Grothendieck group of the Fukaya category of $\mcal K$.
We do this via intermediate comparison with the Chow groups of the mirror $Y$ (see \Cref{sec:hmsbielliptic}).
Thus, from now on we fix $G = U_\Lambda$ to be the unitary group of the Novikov field.

Recall from \Cref{ap:homology} that there is an isomorphism $H_2(\mcal K;\xi)\cong \Z^4\oplus\Z_2$.
Putting this together with Theorem \ref{th:cobK}, one obtains an explicit presentation 
$$
\Cob^{trop}(\mcal K) \cong \Z^4\oplus \Z_2^2\oplus E^2
$$
for $E$  the closed points of an elliptic curve over the Novikov field.
We now sketch how to compute the Chow groups of the mirror $Y$ and exhibit a specific isomorphism between $\Cob^{trop}(\mcal K)$ and $\CH^*(Y)$:
\begin{itemize}
    \item $\CH_2(Y)\cong \Z$ is generated by the bielliptic surface $Y$ itself.
    This factor is mirror to a $\Z$-factor in $\Cob^{trop}(\mcal K)$ corresponding to the zero-section.

    \item For complex bielliptic surfaces $\CH_0(Y)\cong \Z\oplus \Alb(Y)$: as follows from the results in \cite{bloch1976zero}.
    The proof is based on analyzing the kernel of the Albanese map $\alb:\CH_0(Y)_{\hom} \to \Alb(Y)$. 
    The Albanese map exists for fields other than the complex numbers, and one can run the arguments in \cite{bloch1976zero} to show that $\ker \alb = 0$ still holds.
    Thus $\CH_0(Y) \cong \Z \oplus \Alb(Y)$,  where the $\Z$ factor is generated by a fixed reference point and the $\Alb(Y)$ factor is generated by differences $p^+ - p^- \in \CH_0(Y)_{hom}$ for points $p^\pm \in Y$.
    This is mirror to $\Cob_{\text{fib}}(\mcal K)\cong \Z\oplus (\Alb(K) \oplus U_\Lambda)$ (an explicit isomorphism is given by the map $\deg \oplus \alb_{\text{fib}}^{loc}:\Cob_{\text{fib}}(\mcal K) \to \Z \oplus (\Alb(K) \oplus U_\Lambda)$, where $\alb_{\text{fib}}^{loc}$ is constructed from the map in \Cref{prop:fiberedalbaneseiso} after choosing a reference fiber).
    
    \item $\CH_1(Y)\cong \Pic(Y)$ fits into a short exact sequence
    $$
    0 \to \Pic^0(Y) \to \Pic(Y) \to NS(Y) \to 0
    $$
    where $\Pic^0(Y)$ is the Picard variety of $Y$ (the connected component of the identity in $\Pic(Y)$) and $NS(Y) = \Pic(Y)/\Pic^0(Y)$ is the so-called Neron-Severi group.
    For a complex bielliptic surface, the group $NS(Y) \cong H^2(Y;\Z)$ has been fully computed in \cite{serrano1990divisors}, and this computation applies to bielliptic surfaces over algebraically closed fields of characteristic zero by the Lefschetz principle. 
    Since $\Pic^0(Y)\cong (\Alb(Y))^\vee \cong E$,\footnote{Recall that $\Alb(Y)$ is an abelian variety and that its dimension is equal to $h^1(Y,\mcal O_Y) = 1$. It follows that $\Alb(Y)$ is an elliptic curve, hence isomorphic to its dual.} we get
    $$
    \CH_1(Y)\cong NS(Y) \oplus E
    $$
    where $E$ is an elliptic curve over the Novikov field.
    The group $NS(Y) \cong \Z^2 \oplus \Z_2^2$ is generated by two general fibers of the projections $\pi_1 : Y \to E_1/\Z_2$ and $\pi_2 : Y \to E_2/Z_2 \cong \P^1$ (cf. \Cref{sec:biellipticsurfaces}) and differences $D_1 - D_2, D_1 - D_3$ of the four reduced special fibers $D_i$ of the projection $\pi_2$ \cite{bergstrom2019brauer}; this is mirror to the subgroup $(\Z^2\oplus\Z_2)\oplus\Z_2\subset H_2(\mcal K;\xi) \oplus \Z_2$ generated by the tropical Lagrangians $L_{C_1}, L_{C_2}, L_{C_2} - L_{C_3}$ and $L_{C_3} - (L_{C_3},\eta)$.
    Here we denote by $L_{C_i}$ the tropical Lagrangians living over the tropical subvarieties $C_i$ of \Cref{eq:tropicalsubmanifolds}; they all carry the natural Pin structure and the trivial local system, except $(L_{C_3},\eta)$ that carries a $2$-torsion local system that is non-trivial in the $dy$-direction.
    The  $E$  factor is generated by differences of fibers of the projection $\pi_1$; it is mirror to $\Alb(K')\oplus U_\Lambda$ and is generated by differences of tropical Lagrangians $L_{C_1} - L_{\{x = a\}}$ equipped with  arbitrary local systems.
\end{itemize}

We now turn to the comparison between the cobordism group $\Cob^{trop}(\mcal K)$ and the Grothendieck group of the Fukaya category.
A homomorphism relating the Lagrangian cobordism group (without the extra data of a brane structure) and the Grothendieck group of the Fukaya category  as in \Cref{eq:BCmap} was first considered by Biran-Cornea  in \cite{biran2014lagrangian}.
Their work proves the existence of such homomorphism for a version of the cobordism group whose generators and relations are {\it monotone} Lagrangians and Lagrangian cobordisms, i.e. Lagrangians $L \subset M$ for which the homomorphisms
\begin{align*}
    \omega: \pi_2(M, L) &\to \R\\
    \mu:\pi_2(M, L) &\to \Z
\end{align*}
given by measuring the symplectic area and the Maslov index of topological discs with boundary on $L$ satisfy $\omega = \lambda_L \mu$ for some $\lambda_L >0$.
This monotonicity condition is imposed to exclude bubbling of pseudo-holomorphic disks, so that Floer homology is well-defined.\footnote{In that paper Biran-Cornea impose an additional condition on the fundamental groups of all Lagrangians and cobordism. This condition gives a bound on the area of holomorphic disks and is not needed when working over the Novikov field, so we will not treat it here.} 
In later work \cite{biran2021lagrangian}, Biran-Cornea-Shelukhin  relax the assumptions on cobordisms to include tautologically unobstructed  cobordisms (recall this means there exists a compatible almost complex structure such that $V$ bounds no non-constant holomorphic disks, a property to which they refer as `quasi-exactness').
This comes at the cost of strengthening the monotonicity condition imposed on the ends, where they must now be weakly exact Lagrangians (recall this means  $\omega (\pi_2(M,L))= 0$).
See their results in \cite[Chapter 4]{biran2021lagrangian}.

\begin{rmk}
    There is a stronger expectation---which is believed by experts but has not yet appeared in the literature---that the iterated cone decomposition results of \cite{biran2014lagrangian,biran2021lagrangian} hold for the broadest class of Lagrangians and cobordisms for which Floer theory can be defined, namely  {\it (Floer-theoretically) unobstructed} Lagrangians in the sense of \cite{fukaya2009lagrangian}.
    These are the most natural Lagrangians to consider in the context of mirror symmetry (for instance, to compare Lagrangian cobordism groups with Chow groups).
    We remark that the group $\Cob^{trop}(\mcal K)$ that we compute in this paper (a group generated by Lagrangian fibers and sections) does {\it not} depend on whether we consider it as a subgroup of $\Cob(\mcal K)$, as defined in \Cref{def:cobgroupKk}, or as a subgroup of a cobordism group built from Floer-theoretically unobstructed Lagrangians and cobordisms.
    However, since the existence of a map from the cobordism group to the Grothendieck group has only been proved to exist for a cobordism group of weakly exact Lagrangians and tautologically unobstructed cobordisms, we think of $\Cob^{trop}(\mcal K)$ as a subgroup of $\Cob(\mcal K)$ as defined in \Cref{def:cobgroupKk}.
\end{rmk}

Recall from  \Cref{def:cobgroupKk} that the generators of $\Cob(\mcal K)$ are weakly exact.
On the other hand, we imposed the weaker condition that cobordisms are tautologically unobstructed only when pulling them back under the covering map $T^4 \x \C \to \mcal K \x \C$.
We argue in \Cref{ap:multivaluedFuk} that cobordisms with the above property still induce iterated cone decompositions between their ends, see \Cref{prop:conesfromTOpullbacks}.
Hence we have a map
$$
\Cob(\mcal K) \to K_0(D\mcal Fuk(\mcal K)).
$$

The above matching of the tropical cobordism group and the Chow groups of the mirror together with the HMS statement of \Cref{sec:hmsbielliptic} suggest our second main result:
\begin{thm}\label{th:isoK}
    The natural  homomorphism
    \begin{equation}\label{eq:isoK}
        \Cob^{trop}(\mcal K) \to K_0(D^b \mcal Fuk(\mcal K))
    \end{equation}
    is an isomorphism.
\end{thm}
\begin{proof}
    We first prove an isomorphism $\Cob^{trop}(\mcal K) \cong K_0(D^\pi \mcal Fuk(\mcal K))$ between the Lagrangian cobordism group and the Grothendieck group of the split-closed derived Fukaya category.
    Recall from \Cref{sec:hmsbielliptic} that we have an equivalence of triangulated categories $D^\pi\mcal Fuk(\mcal K) \simeq D^bCoh(\check {\mcal K})$, the latter category being isomorphic to the derived category of coherent sheaves $D^bCoh(Y)$ of an algebraic bielliptic surface over the Novikov field.
    Taking the Grothendieck group on both sides yields an isomorphism of abelian groups 
    $$
    K_0(D^\pi\mcal Fuk(\mcal K)) \cong K_0(D^bCoh(Y))
    $$
    which reduces the computation of $K_0(D^\pi\mcal Fuk(\mcal K))$ to algebraic geometry.
    Now recall that there is a Chern character map
    $$
    ch : K_0(D^bCoh(Y)) \to \CH^*(Y)_\Q
    $$
    which furthermore is a rational isomorphism.
    In the present case, one can show that $ch$ is well-defined as a map to $\CH^*(Y)$ (i.e. we can get rid of denominators), and that such map is an isomorphism---see \Cref{ap:integralChern}.
    Putting this together with the Chow group computations at the beginning of this section, we have
    $$
    K_0(D^\pi\mcal Fuk(\mcal K)) \cong K_0(D^bCoh(Y)) \cong \CH^*(Y) \cong \Z^4 \oplus \Z_2^2 \oplus E^2,
    $$
    showing that $K_0(D^\pi\mcal Fuk(\mcal K))$ and $\Cob^{trop}(\mcal K)$ are, at least abstractly, isomorphic groups.
    
    To see that the natural homomorphism $\Cob^{trop}(\mcal K) \to K_0(D^\pi \mcal Fuk(\mcal K))$ provides an isomorphism we will show that the composition 
    $$
    \Cob^{trop}(\mcal K) \to K_0(D^\pi\mcal Fuk(\mcal K)) \xrightarrow{\sim} K_0(D^bCoh(Y))
    $$
    is an isomorphism.
    For this we use that the mirror functor of \cite{abouzaid2014family,abouzaid2021homological} sends a tropical Lagrangian $L_V$ living over a tropical submanifold $V\subset K$ to the structure sheaf $\mcal O_{Z_V}$ of the corresponding algebraic subvariety $Z_V \subset Y$ living over $V$.
    Choosing a set of generators of $\Cob^{trop}(\mcal K)$ given by (linear combinations of) tropical Lagrangians as at the beginning of this section, the map
    $$
    \Cob^{trop}(\mcal K) \cong \Z^4 \oplus \Z_2^2 \oplus E^2 \to K_0(D^bCoh(Y))
    $$ 
    coincides with the map of \Cref{cor:generatorsK0Coh}, hence the result follows.

    To conclude the proof, note that the natural embedding $D^b\mcal Fuk(\mcal K) \into D^\pi \mcal Fuk(\mcal K)$ induces an isomorphism of Grothendieck groups: it is always an injection and the generators of $K_0(D^\pi\mcal Fuk(\mcal K))$ are honest Lagrangians (as opposed to direct summands), thus it is also surjective.
\end{proof}

     \subsection{Admissibility of the surgery cobordism}\label{ap:compatibility}
In this section we upgrade \Cref{lem:cobsectionstogenerators} to include all the extra structure that is present in the cobordism group ($G$-brane structures and unobstructedness properties).
It is mostly technical and can safely be skipped by readers on a first read.

There are two separate issues to treat.
The first is the transformation of gradings, Pin structures and local systems under fiberwise addition.
The second is  the admissibility of the Lagrangian cobordisms obtained as surgery of a section $\Gamma^{n[l]}$ (or $\Gamma_{m\theta}$) and the zero section $\Gamma_0$---that is, whether these cobordisms admit $G$-brane structures and whether their pullback to $T^4 \x \C$ is tautologically unobstructed.

Let us start with the transformation of gradings, Pin structures and local systems under fiberwise addition.
This is part of the more general story of how these properties transform under geometric composition of Lagrangian correspondences:
transformation of gradings is treated in \cite[Section 3]{wehrheim2010quilted}, transformation of Pin structures is discussed in \cite{wehrheim2015orientations} and transformation of local systems is detailed in \cite{subotic}.
To apply these results, recall that the Lagrangian correspondence giving rise to fiberwise addition is given by 
$$
\Sigma= \{ (q_1, p_1, q_2,p_2, q_3,p_3) \in (X(B) \x X(B))^- \x X(B) \st q_1 = q_2 = q_3,\, p_3 = p_1 + p_2 \}
$$
for $(q_i,p_i) \in X(B)$ canonical coordinates \cite{subotic}; one then defines 
$$
L_1 \otimes L_2 = \Phi_\Sigma (L_1 \x L_2) = \pi_{13}((L_1 \x L_2 \x X(B)) \cap \Sigma).
$$
\begin{lem}\label{lem:branestructureoncorrespondence}
    The Lagrangian correspondence $\Sigma$ admits a natural grading.
    Furthermore, any Pin structure on $B$ induces a natural Pin structure on $\Sigma$.
\end{lem}
\begin{proof}
    With respect to a quadratic volume form defined by complexifying a section of $(\wedge^{3n} T^*B^3)^{\otimes2}$, a computation shows the Lagrangian correspondence $\Sigma$ has phase map $\alpha_\Sigma: \Sigma\to S^1$ constant and equal to $(-1)^{\dim B} \in S^1$.
    We equip $\Sigma$ with the grading $\tilde \alpha_\Sigma : \Sigma\to \R$ given by the unique constant function $\tilde \alpha_\Sigma \in \{0,1/2\}$ (in particular, $\tilde\alpha_\Sigma\equiv 0$ when $\dim B$ is even).
    To equip $\Sigma$ with a Pin structure, note that we have an isomorphism $T\Sigma\cong TB \oplus T^*B \oplus T^*B$. 
    Hence a Pin structure on $B$ induces one on  $\Sigma$.
\end{proof}
Note that in our case we have chosen a Pin structure on $B$ (this is an input to run Family Floer theory, see \cite[Section 3.1]{abouzaid2014family}).
Equip the Lagrangian correspondence $\Sigma$ with the natural choice of grading and Pin structure given by \Cref{lem:branestructureoncorrespondence}.

\begin{lem}\label{lem:gradingfiberwisesum}
    There are gradings on $\Gamma_{m\theta}$ and $\Gamma^{n[l]}$ inducing any given grading on $\Gamma^{n[l]}_{m\theta}$.
\end{lem}
\begin{proof}
    This follows directly by tracing the construction of a grading on the geometric composition explained in \cite[Section 3]{wehrheim2010quilted}.
\end{proof}

\begin{lem}\label{lem:pinfiberwisesum}
   The Pin structure induced on the fiberwise addition of any two tropical Lagrangians equipped with the preferred Pin structure is again the preferred Pin structure.
\end{lem}
\begin{proof}
    Let us first recall how the geometric composition of Lagrangian correspondences inherits a Pin structure.
Given a Lagrangian correspondence $\Sigma\subset X^- \x Y$ and a Lagrangian $L \subset X$, the geometric composition is defined by $\Phi_\Sigma(L) = \pi_Y(L\x_{\Delta_X} \Sigma) \subset Y$, where $\Delta_X \subset X \x X$ is the diagonal (we assume that the intersection is transverse and that the map $\pi_Y$ restricts to an embedding).
Let $Z = L\x_{\Delta_X} \Sigma$, which is diffeomorphic to $\Phi_\Gamma(L)$ via $\pi_Y$.
Consider the isomorphism $TZ \oplus T\Delta_X \cong TL \oplus T \Sigma$.
It follows that Pin structures on $\Sigma$, $L$ and $\Delta_X$ determine one on $Z$ (and hence on $\Phi_\Sigma(L) \cong Z$).

In the present case, $X = \mcal K$ comes equipped with a polarisation $\mcal P$ given by the tangent space to the fibers, giving an isomorphism $T\mcal K \cong T\mcal P \oplus T\mcal P$.
Since fibers have a preferred Pin structure induced from a choice of Pin structure on the base,  $\mcal K$ (and thus $\Delta_{\mcal K}$) have a natural Pin structure.
One can then see that, with the choice of Pin structure for $\Sigma$ explained above, the fiberwise addition of two tropical Lagrangians equipped with the preferred Pin structure yields another Lagrangian with the preferred Pin structure (the reason being that the inclusion $TZ \cong TK \into TL \oplus T \Sigma \cong (TK)^{\oplus 3} \oplus (T^*K)^{\oplus 2}$ can be homotoped to the map that identifies $TZ$ with the first $TK$ factor on the right, hence the Pin structure induced on $Z$ is that of $K$).
\end{proof}

We now turn to the second issue, namely the admissibility of the surgery cobordisms.
The main observation is that all the Lagrangians we are considering are product-type when pulled back to the cover $T^4 = T^2 \x T^2 \to \mcal K$.
Furthermore,  the surgery cobordisms $V\subset \mcal K \x \C$ involved in \Cref{eq:jeffsurgerysections} pull back under $p: T^4 \x \C \to \mcal K \x \C$ to either
\begin{equation}\label{eq:cobordismisproducttype}
    p^{-1}(V) = S^1 \x W \subset T^2 \x T^2\x \C \quad \text{ or } \quad  p^{-1}(V) = W \x S^1 \subset T^2 \x \C \x T^2,
\end{equation}
where $W \subset T^2 \x \C$ is a surgery cobordism  in $T^2$ and $S^1 \subset T^2$ is the zero section.

\begin{lem}\label{lem:matchinggradings}
    The cobordism $V$ admits a grading.
\end{lem}
\begin{proof}
    Note that the phase map of $p\inv(V)$ is the product of the phase maps of $S^1$ and $W$.
    The former is constant and the latter is null-homotopic by \cite[Lemma 5.2]{haug2015lagrangian}; thus $p\inv(V)$ admits a grading.
    Since furthermore $\alpha_{p\inv(V)} =\alpha_V \circ p\restr{p\inv(V)}$ and $p\restr{p\inv(V)}:p\inv(V) \to V$ induces an injection on $H^1(-;\Z)$, the map $\alpha_V$ is also null-homotopic.    
    Thus $V$ admits a grading and, given a grading on $\Gamma_{m\theta}$ (resp. $\Gamma^{n[l]}$),  \Cref{eq:jeffsurgerysections} holds for $L_{f_{m\theta}}$ (resp. $L_{f^{nl}}$) and $\Gamma_0$ equipped with the induced grading from $V$.
    Note that Lagrangian cobordism classes are invariant under a grading change of $2k$, whereas changing the grading by $2k + 1$ modifies their class in the cobordism group by a factor of $-1$ (this can be seen using a product-type cobordism $\gamma \x L$, where $\gamma \subset \C$ is a C-shaped curve).
\end{proof}

\begin{lem}\label{lem:matchingPin}
    The cobordism $V$ admits a unique Pin structure restricting to the preferred one on its two section ends.
\end{lem}
\begin{proof}

Let us analyze the first case in \Cref{eq:cobordismisproducttype} (the second is similar but simpler).
The topology of $W$ is that of a pair of pants. 
In particular, $TW$ is trivial and it admits a canonical Pin structure, which in turn induces one on $[0,1] \x W$. 
Let $t$ be a coordinate on $[0,1]$ and $(p,z)$ be coordinates on $W \subset T^2 \x \C$.
Then giving a Pin structure on 
$$
V \simeq \frac{[0,1] \x W }{(t,p,z)\sim (t+1,-p,z)}
$$
amounts to choosing a lift to $Pin_3$ of the linearisation of the gluing map. 
One can see that in a suitable trivialisation of $TW$, the linearization of the gluing map takes the form $(\partial_u,\partial_v,\partial_t) \mapsto (\partial_u,-\partial_v,\partial_t)$, which is a reflection along the hyperplane $\partial_v = 0$. 
Recalling that lifts of reflections $R \in O_n$ to $Pin_n$ are identified with co-orientations of the fixed hyperplane,  the two choices of Pin structure above can be identified with the vectors $\pm \partial_v$.
Given such a choice, we equip the ends of the cobordism with the restricted Pin structure. 
Note that the induced Pin structure is the same on all three ends.

Lastly, recall that Pin structures on  $V$ are a torsor over $H^1(V;\Z_2) \cong \Z_2^2 \oplus \Z_2$. 
The restriction map $H^1(V;\Z_2) \to H^1(\partial V;\Z_2) \cong \Z_2^2\oplus \Z_2^2\oplus \Z_2^2$ is given by 
\begin{equation}\label{eq:restrictionoftorsor}
    (a_1,a_2,b) \mapsto (a_1,b) \oplus (a_2,b) \oplus (a_1 + a_2,b)
\end{equation}
where we have decomposed $\partial V = \sqcup_{i=1}^3 Z_i$ as a disjoint union of its three boundary Klein bottles.
Since the canonical Pin structure on $W \x [0,1]$ together with a choice $\pm \partial_t$ induces the same Pin structure on the three ends, it follows from this and \Cref{eq:restrictionoftorsor} that there exists a unique Pin structure on $V$ restricting to the preferred one on $\Gamma^{n[l]}$ and $\Gamma_0$.
\end{proof}

\begin{lem}\label{lem:matchinglocsyst}
    Given a local system  $\eta = \eta_\Z \otimes\eta_{\Z_2}: H_1(\Gamma^{n[l]}_{m\theta}) \to G$ on $\Gamma^{n[l]}_{m\theta}$, one can put local systems on $L_{f_{m\theta}}$ and $\Gamma_0$ and the trivial local system on $L_{f^{nl}}$  such that \Cref{lem:cobsectionstogenerators} holds.
\end{lem}
\begin{proof}
Let us first recall how the construction in \cite{subotic} interacts with local systems.
Let $\Sigma \subset X^- \x Y$ be a  Lagrangian correspondence and consider a Lagrangian $L \subset X$ equipped with a local system $\eta$.
Then, assuming $\pi_{Y}: L_1 \x_{X} \Sigma \to Y$ is an embedding, one equips the Lagrangian $\Phi_\Sigma(L) = \pi_{Y}(L_1 \x_{X} \Sigma)$ with the local system given by the composition 
$$
\pi_1(\Phi_\Sigma (L)) \xrightarrow{(\pi_{Y}^{-1})_\#} \pi_1(L_1 \x_{X} \Sigma) \xrightarrow{(p_1\circ i)_\#} \pi_1(L_1) \xrightarrow{\eta} G
$$
where $i: L_1 \x_{X} \Sigma \into L_1 \x \Sigma$ is the inclusion and $p_1 : L_1 \x \Sigma \to L_1$ is the projection.

The above construction shows that, given any local system $\eta$ on $\Gamma^{n[l]}_{m\theta}$, one has
$$
(\Gamma^{n[l]}_{m\theta},\eta) = (\Gamma_{m\theta},\eta) \otimes (\Gamma^{n[l]},\mathbf 1),
$$
Here, in the right-hand side we denote still by $\eta$ the  local systems on $\Gamma_{m\theta}$ induced by the canonical diffeomorphism between $\Gamma^{n[l]}_{m\theta}$ and $\Gamma_{m\theta}$, and we denote by  $\mathbf{1}$  the trivial local system. 

We first study the potential restrictions on local systems for the existence of a cobordism $V$ between $\Gamma_{m\theta}$, $\Gamma_0$ and $L_{f_{m\theta}}$ as in \Cref{eq:jeffsurgerysections}.
One can use a Mayer-Vietoris decomposition of the cobordism $V$ between $\Gamma_{m\theta}$, $\Gamma_0$ and $L_{f_{m\theta}}$ to show that 
$$
H_1(Z) \cong 
\frac{H_1(\Gamma_{m\theta}) \oplus H_1(\Gamma_0)}{v \sim w} 
\oplus
\Z^{m-1}
$$
where $v,w \in \Z_2$ are the generators of the torsion factor of $H_1(\Gamma_{m\theta})$ and $H_1(\Gamma_0)$ respectively.
The last factor $\Z^{m-1}$ can be identified with cycles living in the conormal direction of $m-1$ of the Lagrangians in the surgery $L_{f_{m\theta}} \cong \sqcup_{i=1}^m N^*S^1 / N^*_\Z S^1$.
Hence, given the local system  $\eta = \eta_\Z \otimes \eta_{\Z_2} : H_1(\Gamma_{m\theta}) \to G$, \Cref{eq:jeffsurgerysections} can be upgraded to
$$
(\Gamma_{m\theta}, \eta_\Z \otimes \eta_{\Z_2}) \sim 
(\Gamma_0,h\otimes \eta_{\Z_2})
+
(L_{f_{m\theta}}, \oplus_{i=1}^m (\nu_i  \otimes \eta_{\Z_2})),
$$
where $h: H_1(\Gamma_0) \to G$ is a local system that factors through $H_1(\Gamma_0)/\Z_2$ and $\nu_i$ are local systems on the $i$-th copy of $L_{f_{m\theta}}$ such that $\eta_\Z(v) = h(v) \prod_i \nu_i(v_i)$.
Here, $v_i$ are the generators of the fiber direction of $L_{f_{m\theta}}$ and $v$ is the generator of $\Z \subset \Z \oplus \Z_2 \cong H_1(\Gamma_{m\theta}) \cong H_1(\Gamma_0)$.

On the other hand, there is clearly a cobordism relation
$$
(\Gamma^{n[l]},\mathbf{1})
\sim
(\Gamma_0, \mathbf{1})
+
(L_{f^{n[l]}}, \mathbf{1}).
$$
Choosing $h = \mathbf 1$ and $\nu_i \in G$ satisfying the relation $\eta_\Z = \prod_i \nu_i$, we get  cobordism relations compatible with local systems
\begin{align*}
\begin{split}
    (\Gamma_{m\theta}^{n[l]},\eta) &=(\Gamma_{m\theta},\eta_\Z \otimes \eta_{\Z_2}) \otimes (\Gamma^{n[l]},\mathbf 1)\\
    &\sim ( 
    (L_{f_{m\theta}}, \oplus_{i=1}^m (\nu_i  \otimes \eta_{\Z_2}))
    +
    (\Gamma_0,\eta_{\Z_2})
    )
    \otimes 
    ( 
    (L_{f^{n[l]}}, \mathbf{1})
    +
    (\Gamma_0, \mathbf{1})
    )\\
    &=(L_{ f_{m\theta}}\otimes L_{f^{n[l]}},\oplus_{ij}\nu_i) 
    + (L_{f_{m\theta}},\oplus_i(\nu_i\otimes\eta_{\Z_2})) 
    + (L_{f^{n[l]}},\mathbf 1 ) 
    + (\Sigma_0,\eta_{\Z_2}).
\end{split}
\end{align*}
\end{proof}

\begin{lem}\label{lem:surgerycobordismisunobstructed}
    The pull back to $T^4 \x \C$ of the cobordism $V$ is tautologically unobstructed. 
\end{lem}

\begin{proof}
    Recall from \Cref{eq:cobordismisproducttype} that, up to reordering of the factors, $p\inv(V)$ is of the form
    $$
    p\inv(V) = S^1 \x W \subset T^2 \x (T^2 \x \C)
    $$
    for $S^1 \subset T^2$  a non-contractible curve and $W$  a surgery cobordism between Lagrangians circles in $T^2$ (in fact, it is always the case for us that one of the Lagrangians is the $0$-section and the other is the graph of a tropical function).
    It was proved in \cite[Lemma 5.2]{haug2015lagrangian} that $W$ has Maslov index 0, thus it bounds no non-constant $J$-holomorphic disks for a generic choice of almost complex structure $J$ on $T^2 \x \C$.
    On the other hand, the projection to the first factor of any disk must be null-homotopic for topological reasons; furthermore, choosing the almost complex structure on $T^2 \x (T^2 \x \C)$ to be of product type $j_{T^2}\oplus j_{T^2\x \C}$ the projection of the disk is a null-homotopic pseudo-holomorphic disk, hence constant.
    Thus $p\inv(V)$ does not bound any non-constant pseudo-holomorphic disks for a complex structure of the form $j_{T^2}\oplus J$.
    
\end{proof}
\begin{rmk}
    The almost complex structure for which $W$ does not bound any holomorphic disk need not be invariant under the covering group.
    Hence, we cannot say that $V$ is tautologically unobstructed (the almost complex structure does not descend to $\mcal K$).
    This is the reason why we need to work with cobordisms with the weaker property that their pullback is tautologically unobstructed.
    Nonetheless, we argue in \Cref{ap:multivaluedFuk} that such cobordisms still induce iterated cone decompositions between their ends (see \Cref{prop:conesfromTOpullbacks}).
\end{rmk}

Summarizing, we have proved the following Proposition:
\begin{prop}\label{prop:admissibilitysurgery}
    Consider a Lagrangian $G$-brane supported on $\Gamma_{m\theta}$ equipped with the preferred Pin structure and an arbitrary grading and local system.
    There exists a Lagrangian cobordism $V \subset \mcal K \x \C$ between $\Gamma_{m\theta}, \Gamma_0$ and $L_{f_{m\theta}}$ such that:
    \begin{enumerate}
        \item $V$ admits a $G$-brane structure restricting to the given $G$-brane structure on $\Gamma_{m\theta}$;
        \item the restriction of the Pin structure of $V$ to $\Gamma_0$ and $L_{f_{m\theta}}$ is the preferred one;
        \item the pullback of $V$ to $T^4 \x \C$ is tautologically unobstructed.
    \end{enumerate}
    A similar statement holds for the Lagrangians $\Gamma^{n[l]}, \Gamma_0$ and $L_{f^{nl}}$.
\end{prop} 
\appendix

    \section{Homology with local coefficients}\label{ap:homology}
Let $\mcal K$ be a symplectic bielliptic surface and $\pi:\mcal K \to K$ its Lagrangian torus fibration over a tropical Klein bottle.
Denote by $\det_\Z K = \wedge^2 T^*_\Z K$ the (integral) orientation bundle of $K$ and let $\xi := \pi^* \det_\Z K$ be the local system on $\mcal K$ obtained by pullback.
In this Appendix we compute the second singular homology $H_2(\mcal K;\xi)$ of $\mcal K$ with coefficients in the local system $\xi$.

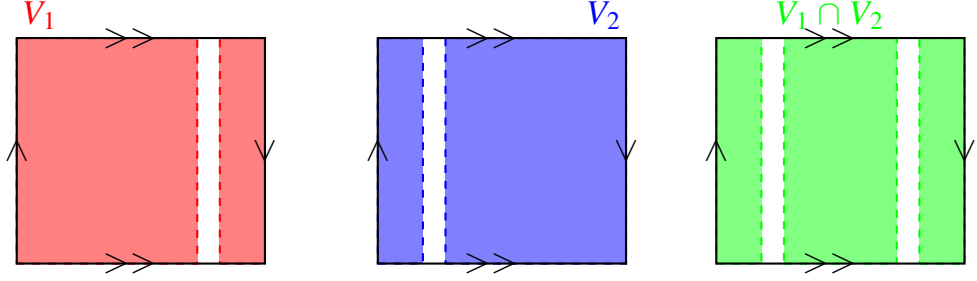
\begin{figure}
    \centering
    \begin{tikzpicture}[scale = 0.6, transform shape, thick]

        \begin{scope}
            \filldraw[color = red, fill = red!50, dashed]  (-3,3) node (v1) {} rectangle (1,-2);
            \filldraw[color = red, fill = red!50, dashed]  (1.5,3) rectangle  (2.5,-2) node (v2) {};
            
            \draw  (v1) rectangle(v2);
            
            \node[scale = 2, rotate = 90] at (-3,0.5) {$>$};
            \node[scale = 2, rotate = 90] at (2.5,0.5) {$<$};
            
            \node[scale = 2] at (-0.5,3) {$>>$};
            \node[scale = 2] at (-0.5,-2) {$>>$};
            
            \node[color = red, scale = 2] at (-2.5,3.5) {$V_1$}; 
        \end{scope}
        
        \begin{scope}[shift={(8,0)}]
            \filldraw[color=blue, fill=blue!50, dashed]   (-3,3) node (v3) {} rectangle (-2,-2);
            \filldraw[color=blue, fill=blue!50, dashed]  (-1.5,3) rectangle (2.5,-2) node (v4) {};
            
            \draw  (v3) rectangle(v4);
            
            \node[scale = 2, rotate = 90] at (-3,0.5) {$>$};
            \node[scale = 2, rotate = 90] at (2.5,0.5) {$<$};
            
            \node[scale = 2] at (-0.5,3) {$>>$};
            \node[scale = 2] at (-0.5,-2) {$>>$};
            
            \node[color = blue, scale = 2] at (2,3.5) {$V_2$};
        \end{scope}
        
        \begin{scope}[shift={(15.5,0)}]
            \filldraw[color=green, fill=green!50, dashed]  (-3,3) node (v1) {} rectangle (-2,-2);
            \filldraw[color=green, fill=green!50, dashed]  (1.5,3) rectangle  (2.5,-2) node (v2) {};
            \filldraw[color=green, fill=green!50, dashed]  (-1.5,3) rectangle (1,-2);
            
            \draw  (v1) rectangle(v2);
            
            \node[scale = 2, rotate = 90] at (-3,0.5) {$>$};
            \node[scale = 2, rotate = 90] at (2.5,0.5) {$<$};
            
            \node[scale = 2] at (-0.5,3) {$>>$};
            \node[scale = 2] at (-0.5,-2) {$>>$};
            
            \node[color = green, scale = 2] at (-0.5,3.5) {$V_1 \cap V_2$};     
        \end{scope}
        
    \end{tikzpicture}
    
    \caption{Decomposition of the Klein bottle to run Mayer-Vietoris. Note that both $V_1$ and $V_2$ retract to a circle, whereas $V_1 \cap V_2$ retracts to the disjoint union of two circles}
    \label{fig:MVdecomposition}
\end{figure}

\begin{prop}
    The second homology of $\mcal K$ with coefficients in the local system $\xi$ is given by
    \begin{equation}
    H_2(\mcal K; \xi) \cong \Z^4 \oplus \Z_2.
    \end{equation}
    Furthermore, every class has a Lagrangian representative.
\end{prop}

\begin{proof}   
    We will use a Mayer-Vietoris decomposition $\mcal K = U_1 \cup U_2$ coming from a decomposition $K = V_1 \cup V_2$ of the base (see \Cref{fig:MVdecomposition}), i.e. $U_i = \pi\inv(V_i)$.
    Note the restricted fibrations $U_i \to V_i$ and $U_1 \cap U_2 \to V_1 \cap V_2$ are trivial.
    Furthermore, since the local system $\xi$ is trivial on the fiber direction (it's pulled back from the base) and the homology of the fiber is free, the K\"unneth isomorphism for homology with local coefficients says\footnote{Although Equations (\ref{eq:kunnethUi})-(\ref{eq:kunnethintersection}) seem to forget the local system, we will see shortly that the maps appearing in the Mayer-Vietoris long exact sequence are different from those computing integral homology.}
    \begin{align}\label{eq:kunnethUi}
        \begin{split}
            H_j(U_i;\xi\restr{U_i}) &\cong \bigoplus_{k=0}^j H_k(V_i; \det_\Z K \restr{V_i}) \otimes H_{j-k} (T^2;\Z)\\
            &\cong \bigoplus_{k=0}^j H_k(S^1; \Z) \otimes H_{j-k} (T^2;\Z)
        \end{split}
    \end{align}
    \begin{equation}\label{eq:kunnethintersection}
           H_j(U_1 \cap U_2;\xi\restr{U_1 \cap U_2}) \cong \left(\bigoplus_{k=0}^j H_k(S^1;\Z) \otimes H_{j-k}(T^2;\Z)\right)^{\oplus 2}.
    \end{equation}
    
    {\it To simplify notation, from now on we will omit the restriction notation and simply write $H_k(U;\xi) \equiv H_k(U;\xi\restr{U})$ for any $U \subset \mcal K$.
    We will also write $H_k(U) \equiv H_k(U;\Z)$ whenever we are taking integral homology.}
    
    The Mayer-Vietoris long exact sequence splits to give
    \begin{equation}\label{eq:mayervietorisses}
        0 \to \frac{H_2(U_1;\xi) \oplus H_2(U_2;\xi)}{\im (i^2_1\oplus i^2_2)_*} \to H_2(\mcal K;\xi) \to \ker (i^1_1\oplus i^1_2)_* \to 0
    \end{equation}
    where the maps $(i^j_1\oplus i^j_2)_*:H_j(U_1 \cap U_2;\xi) \to H_j(U_1;\xi) \oplus H_j(U_2;\xi)$ are induced by the inclusions $i_k: U_1 \cap U_2 \into U_k$.
    One can check that $\ker (i^1_1\oplus i^1_2)_*  \cong \Z^2$, with one factor  generated by the difference of the two primitive elements of $H_1(V_1\cap V_2;\Z) \otimes H_0(T^2)$
    and the other by the difference of two primitive elements of $H_0(V_1 \cap V_2; \Z) \otimes \vspan{dy} \subset H_0(V_1 \cap V_2; \Z) \otimes  H_1(T^2)$ (recall there is a canonical identification $H_1(F_p;\Z) \simeq T^*_{p,\Z} K$). 
    Analyzing the Mayer-Vietoris boundary map, one sees the short exact sequence (\ref{eq:mayervietorisses}) is split by the map $\ker (i^1_1\oplus i^1_2)_* \to  H_2(\mcal K;\xi)$ sending these two generators to the $0$-section (which is a cycle with $\xi$-coefficients) and the conormal lift of the tropical curve $\{y = 1/2\}$.
    We now compute the quotient in the left of \Cref{eq:mayervietorisses}:
    
    \begin{itemize}
        \item Using (\ref{eq:kunnethUi}) we get $H_2(U_1;\xi)\cong \Z \oplus \Z^2 \cong H_2(U_2;\xi)$, the first factor generated by a fiber and the last two by the conormal and cotangent lift of the tropical circle $\{x = 1/2\} \subset K$.

        \item To compute $\im (i^2_1\oplus i^2_2)_*$, note that $H_2(U_1 \cap U_2;\xi) \cong (\Z \oplus \Z^2)^{\oplus 2}$ by \Cref{eq:kunnethintersection}.
        Denoting by $a_1,b_1,c_1,a_2,b_2,c_2$ the generators of $H_2(U_1 \cap U_2;\xi) \cong (\Z \oplus \Z^2)^{\oplus 2}$, the map $(i^2_1,i^2_2)_*:H_2(U_1 \cap U_2;\xi) \to H_2(U_1;\xi) \oplus H_2(U_2;\xi)$ is given by
        $$
        (a_1,b_1,c_1,a_2,b_2,c_2) \mapsto (a_1 + a_2, b_1 - b_2, c_1 + c_2)\oplus (a_1 + a_2, b_1 + b_2, c_1 + c_2).
        $$
        Hence 
        $$
        \im (i^2_1\oplus i^2_2)_* = \vspan{(u,v-w,z)\oplus (u,v+w,z), \, u,v,w,z \in \Z} \subset H_2(U_1;\xi) \oplus H_2(U_2;\xi).
        $$
        \item Putting everything together, we get
        $$
        \frac{H_2(U_1;\xi) \oplus H_2(U_2;\xi)}{\im (i^2_1\oplus i^2_2)_*} \cong \Z \oplus \Z_2 \oplus \Z
        $$
    generated by a fiber and the cotangent and conormal lift to the curve $\{ x = 1/2 \}$ respectively.
    \end{itemize}
    Reordering the terms, we have computed
    $$
    H_2(\mcal K; \xi) \cong \Z^4 \oplus \Z_2,
    $$
    with generators a fiber, the $0$-section, the conormal lifts to the curves $\{x = 1/2\}$ and $\{ y= 1/2\}$ and the cotangent lift of the curve $\{x = 1/2\}$.
    
    For the last part of the Theorem, note all but the last homology class have been given Lagrangian representatives. 
    For the last one, the $3$-chain obtained by taking the $y$-cotangent direction to the $2$-chain given by the bottom-half of $K$ has boundary $T^*\{x=0\} - N^*\{y=1/2\} + N^*\{y=0\}$, showing that the symplectic torus $T^*\{x=0\}$ representing the $\Z_2$-factor has the difference of Lagrangians $N^*\{y=1/2\} - N^*\{y=0\}$ as a representative.
\end{proof}

\begin{rmk}
    Note there are important differences between $H_2(\mcal K;\xi)$ and $H_2(\mcal K; \Z)$.
    Most notably for our purposes, the zero-section is not a cycle with integer coefficients, whereas it is with $\xi$-coefficients and furthermore it defines a non-zero homology class.
    A similar phenomenon happens for the conormal lift of the tropical circles $\{y=0\}$ and $\{y=1/2\}$: they are topologically Klein bottles and do not define cycles with integer coefficients, but they do with $\xi$-coefficients.
    Also, the Lagrangian (conormal) lift of the $y$-axis is $2$-torsion in $H_2(\mcal K;\Z)$, but it has infinite order in $H_2(\mcal K;\xi)$.
\end{rmk}

     \section{Integral Chern character}\label{ap:integralChern}
In this appendix we show that for a bielliptic surface $Y$ the Chern character $ch: K_0(D^bCoh(Y)) \to \CH^*(Y)_\Q$ can be lifted to an integral isomorphism.
That is, there exists a well-defined map $\widetilde{ch}: K_0(D^bCoh(Y)) \to \CH^*(Y)$ and it is an isomorphism.
\begin{rmk}
    Similar results appear in \cite[Corollary 1.5]{huybrechts2016lectures}, where Huybrechts shows that for a K3 surface the Chern character $ch: K_0(K3) \to \CH^*(K3)_\Q$ factors through the natural map $\CH^*(K3) \to \CH^*(K3)_\Q$ in a canonical way, and that the corresponding map $K_0(K3) \to \CH^*(K3)$ is an isomorphism.
    A crucial property for his proof to work is that $\CH_0(K3)$ is torsion-free. 
    That clearly does not hold in our case, where $\CH_0(Y) = \Z \oplus \Alb(Y)$, the latter factor being an elliptic curve.
\end{rmk}

We first state a preliminary result, whose proof is a simple computation:
\begin{lem}\label{lem:quasilinear}
    Let $V = \oplus_i V_i$ be a graded $\Z$-module and let $W$ be another $\Z$-module. 
    Assume we have a commutative map $V\otimes V \to W$ and maps $H_i:V_i \to W$ such that
    \begin{equation}\label{eq:quasilinear}
        H_i(\alpha + \beta) = H_i(\alpha) + H_i(\beta) + \alpha\beta.
    \end{equation}
    Then the map $H: V \to W$ defined by
    $$
    H(\oplus_iv_i) : = \sum_i H_i(v_i) + \sum_{i<j} v_iv_j
    $$
    satisfies \Cref{eq:quasilinear}.
\end{lem}

\begin{prop}\label{prop:integralchern}
    Let $Y$ be an algebraic bielliptic surface.
    There is an isomorphism of abelian groups 
        $$
            ch:K_0(D^bCoh(Y)) \to CH^*(Y).
        $$
\end{prop}

\begin{proof}
    Recall from \Cref{sec:isoK} that $\CH^1(Y) = \Z^2 \oplus \Z_2^2 \oplus \Pic^0(Y)$.
    Write $Y = (E_1 \x E_2)/\Z_2$ as in \Cref{def:biellipticalgebraic} and let $\pi_i: Y \to E_i/\Z_2$, where $E_1/\Z_2$ is again an elliptic curve and $E_2/G \cong \P^1$.
    Then, following \cite{serrano1990divisors,bergstrom2019brauer}, we can give the following explicit generators for $\CH^1(Y)$.
    The first two $\Z$-factors are generated by a fiber of $\pi_1$ and a reduced multiple fiber of $\pi_2$; call these $D_1$ and $D_2$.
    The second two $\Z_2$-factors are generated by $F_1 - F_2$ and $F_1 - F_4$, where $F_i$ are the four reduced multiple fibers of $\pi_2$ and we take $F_1 = D_2$; call these two generators $D_3$ and $D_4$.
    Lastly,  $\Pic^0(Y)$ is an elliptic curve, whose points correspond to the divisor classes  $D_p:=D_1 - \pi_1\inv(p)$ for $p\in E_1/\Z_2 \cong \Pic^0(Y)$;  we will write $D_5^p$ for these generators.
    Since the normal bundles to both $D_1$ and $D^p_5$ are trivial, it follows that $D_1^2 = (D^p_5)^2 = 0 \in \CH_0(Y)$.
    Furthermore, we have
    $$
    D_3^2 = (F_1 - F_2)^2 = F_1^2 + F_2^2 = 0
    $$
    since $F_i^2 = F_j^2$ for all $i,j$ and these classes are necessarily $2$-torsion.
    The same argument shows that $D_4^2 = 0$.

Now write $\CH^1(Y) = \Z^2 \oplus \Z_2^2 \oplus \Pic^0(Y) = \oplus_i V_i$, where $V_i$ denotes the $i$-th summand of $\CH^1(Y)$ as in \Cref{lem:quasilinear}.
    Then the previous discussion shows that the maps $H_i \equiv 0$ for $i \neq 2$ satisfy \Cref{eq:quasilinear}.
    Furthermore, define $H_2(kD_2) = k^2 P$, where $P \in \CH^2(Y)$ is chosen such that $2P = D_2^2$ (that such $P$ exists follows from $D_2^2$ being a $2$-torsion element in $\CH_0(Y)$, hence it lives in the divisible subgroup $\Alb(Y) \subset \CH_0(Y)$).
    Then
    $$
    H_2(mD_2 + nD_2) = (m+n)^2 P = H_2(mD_2) + H_2(nD_2) + 2mnP 
    $$
    with $2mnP = mnD_2^2$, so $H_2$ also satisfies \Cref{eq:quasilinear}.
    Let  $H:\CH^1(Y) \to \CH^2(Y)$ be the map defined in \Cref{lem:quasilinear}. 
    We define a modified Chern character $\tilde{ch}:K_0(D^bCoh(Y)) \to \CH^*(Y)$ via
    $$
    \tilde{ch}(\mcal F) = rk(\mcal F) + c_1(\mcal F) + (H\circ c_1 - c_2)(\mcal F).
    $$
    Note that $\tilde{ch}$ is a group homomorphism: for $\mcal F_1,\mcal F_2\in K_0(D^bCoh(Y))$ we have
    \begin{align*}
        \tilde{ch}_2(\mcal F_1 \oplus \mcal F_2) &= H(c_1(\mcal F_1\oplus\mcal F_2)) -c_2(\mcal F_1\oplus \mcal F_2)\\
        &=H(c_1(\mcal F_1) + c_1(\mcal F_2)) - (c_2(\mcal F_1) + c_2(\mcal F_2) + c_1(\mcal F_1)c_1(\mcal F_2))\\
        &=H(c_1(\mcal F_1)) + H(c_1(\mcal F_2)) + c_1(\mcal F_1)c_1(\mcal F_2) - (c_2(\mcal F_1) + c_2(\mcal F_2) + c_1(\mcal F_1)c_1(\mcal F_2))\\
        &=\tilde{ch}_2(\mcal F_1) + \tilde{ch}_2(\mcal F_2).
    \end{align*}
    
    After establishing the existence of a well-defined integral Chern character $\tilde{ch}:K_0(D^bCoh(Y)) \to \CH^*(Y)$, the fact that it is an isomorphism is a minor adaptation of the argument in \cite[Corollary 1.5]{huybrechts2016lectures}.
    Namely, we have $\tilde{ch}(\Oo) = 1$ so $\CH_2(Y) \subset \im(\tilde{ch})$.  
    Since $\tilde{ch}(\mcal O_p) = c_2 (\mcal O_p) = p$, it follows that also $\CH_0(Y) \subset \im(\tilde{ch})$.
    Lastly, $\CH_1(Y) \subset \im(\tilde{ch})$ because $\tilde{ch}(\Oo(D)) = 1 + D + H(D)$ and both $1$ and $H(D)$ are in $\im(\tilde{ch})$. 
    Hence $\tilde{ch}$ is surjective.
    For injectivity, note that for any smooth surface $S$ we have isomorphisms \cite[Example 15.3.6]{fulton2013intersection}
    \begin{align*}
        rk:F_0 K_0(S) / F_1 K_0(S) &\to \Z\\
        c_1: F_1K_0(S)/F_2K_0(S) &\to \Pic(S)\\
        c_2: F_2K_0(S) &\to \CH_0(S),
    \end{align*}
    where $F_iK_0(S) \subset K_0(S)$ is the filtration given by sheaves whose support has dimension at most $2-i$.
    Suppose now that $\tilde{ch}(\mcal F) = 0$.
    Since $\tilde{ch}(\mcal F) = 0$ implies $0 = rk(\mcal F) = c_1(\mcal F)$, we get $\mcal F \in F_2K_0(D^bCoh(Y))$. 
    Furthermore $0 = H(c_1(\mcal F)) - c_2(\mcal F) = -c_2(\mcal F)$, so using that $c_2: F_2K_0(D^bCoh(Y)) \to \CH_0$ is an isomorphism we get $\mcal F = 0$.
\end{proof}

\begin{rmk}
    Unlike for K3 surfaces, our proof does not produce a canonical lift of the Chern character: we have made a choice of the `half class' $P \in \CH_0(Y)_{hom}$ such that $2P = D_2^2$. 
\end{rmk}

In \Cref{sec:isoK} we argue that $\CH^*(Y) = \Z^4 \oplus \Z_2^2 \oplus E^2$ and give explicit generators: we denote these by $Z_i, i = 1,\dots,6$ and $Z_{p},Z'_{p'}, (p,p')\in E^2$.
More explicitly, $Z_1 = Y, Z_2$ is some reference point $p_0 \in Y$, the divisors $Z_3,\dots,Z_6,Z_p$ are the divisors denoted by $D_1,\dots,D_4,D_5^p$ in the proof of \Cref{prop:integralchern}, and $Z'_{p'}$ is represented by $p_0 - p' \in \CH_0(Y)_{hom}$.
We consider classes $[\mcal O_{Z_i}], [\mcal O_{Z_{p}}], [\mcal O_{Z'_{p'}}] \in K_0(D^bCoh(Y))$ represented by the structure sheaves of such generators of $\CH^*(Y)$.
The following result will be useful to prove \Cref{th:isoK}:
\begin{cor}\label{cor:generatorsK0Coh}
    There is a well-defined group homomorphism
    \begin{align*}
        h:\Z^4 \oplus \Z_2^2 \oplus E^2 &\to K_0(D^bCoh(Y))\\
        (n_1,\dots,n_6,p,p') &\mapsto \sum_i n_i[\mcal O_{Z_i}] + [\mcal O_{Z_{p}}] +[\mcal O_{Z'_{p'}}]
    \end{align*}
    and it is an isomorphism.
\end{cor}
\begin{proof}
    To see it is well-defined we must check that the relations
    \begin{align*}
        2[\mcal O_{Z_i}] &= 0,  \, i = 5,6\\
        [\mcal O_{Z_{p_1+p_2}}] &= [\mcal O_{Z_{p_1}}] + [\mcal O_{Z_{p_2}}],\, p_j \in E\\
        [\mcal O_{Z'_{p'_1+p'_2}}] &= [\mcal O_{Z'_{p'_1}}] + [\mcal O_{Z'_{p'_2}}],\, p'_j \in E\\
    \end{align*}
    hold in $K_0(D^bCoh(Y))$. 
    It is equivalent to show that their images under $\tilde{ch}$ satisfy such relations.
    For the first two, note that $\widetilde{ch}([\mcal O_{Z_i}]) = Z_i$ for $i = 5,6$, and that the classes $Z_i$ are $2$-torsion in $\CH^*(Y)$. 
    For the last two, we have $\widetilde{ch}([\mcal O_{Z_{p}}]) = Z_{p}$ for any $p \in E$ (and similarly for $Z_{p'}$), so the relations also hold.
    
    To see it is an isomorphism, note that $\widetilde{ch}([\mcal O_{Z_i}]) = Z_i$ for all $i \neq 4$ and $\widetilde{ch}([\mcal O_{Z_4}]) = Z_4 + P$, hence  the composition $\widetilde{ch}\circ h$ is given by
    \begin{align*}
        \Z^4 \oplus \Z_2^2 \oplus E^2 &\to \CH^*(Y)\\
        (n_1,\dots,n_6,p,p') &\mapsto \sum_i n_iZ_i +n_4 P + Z_{p} +Z'_{p'}.
    \end{align*}
    This map is clearly an isomorphism, hence so is $h$.
\end{proof}
     \section{Fukaya categories from multivalued perturbations}\label{ap:multivaluedFuk}

In this appendix we sketch a definition of a Fukaya category $\mcal Fuk(X;Y)$ of a symplectic manifold together with a finite covering $\pi: Y \to X$ using multivalued perturbations.
We then apply this construction to $Y = T^4 \x \C$ and $X = \mcal K \x \C$ and argue that quasi exact cobordisms in $\mcal Fuk(\mcal K\x \C; T^4 \x \C)$ between weakly exact Lagrangians still give iterated cone decomposition results as in \cite{biran2021lagrangian}, as claimed in \Cref{sec:isoK}.

Let $X$ be a symplectic manifold and $\pi:Y \to X$ a finite covering with covering group $H$.
Given a collection of Lagrangians $\mcal L(X)$, we define two Fukaya categories:
\begin{itemize}
    \item  $\mcal Fuk(X) = \mcal Fuk(X;\mcal L(X))$ is the Fukaya category whose objects are pairs $(L,J_L)$ for $L \in \mcal L(X)$ and  $J_L$  an almost complex structure such that $L$ bounds no non-constant $J_L$-holomorphic disk.
    Morphism spaces and the $A_\infty$-structure are defined as usual \cite{fukaya2009lagrangian,seidel2008fukaya}.
    
    \item $\mcal Fuk(X;Y) = \mcal Fuk(X; \mcal L(X),Y)$ is a  Fukaya category whose objects are pairs $(L,\tilde J_L)$ for $L\in\mcal L(X)$ and such that $\pi\inv(L)$ bounds no $\tilde J_L$ holomorphic disks (note that now $\tilde J_L$ is an almost complex structure on $Y$, not on $X$).
    Morphism spaces in  $\mcal Fuk(X;Y)$ are defined to be the same as in $\mcal Fuk(X)$.
    To define the $A_\infty$-operations, let $(L_0,\tilde J_0),\dots,(L_k,\tilde J_k)$ be objects of $\mcal Fuk(X;Y)$ and let $D_k$ be a disk with $k+1$ boundary punctures $p_0,\dots,p_k\in \partial D$.
    Given an element $\mathtt h=(h_0,\dots,h_k)\in H^{k+1}$, we define an {\it $\mathtt h$-perturbation datum} to be a map $J_g: D_k \to \mcal J= \{\omega-\text{compatible a.c.s on } Y\}$  whose restriction to the $i$-th boundary component of $D_k$ agrees with $(h_i)^*\tilde J_i$ (here we denote by $h_i: Y \to Y$ the covering automorphism induced by $h_i$).
    A {\it universal choice of perturbation datum} is a choice of $\mathtt h$-perturbation datum for each $\mathtt h\in H^{k+1}$ and each fiber of the universal curve over the moduli-space of disks with $k+1$ punctures; these should be regular and compatible in the usual sense (where we think of the elements $h_i$ as labeling the $i$-th boundary component of $D_k$).
    Given  $y$ the data of $y_0,\dots,y_{k-1}$  Hamiltonian chords from $L_i$ to $L_{i+1}$ and $y_k$ a Hamiltonian chord from $L_0$ to $L_k$ and $\mathtt h\in H^{k+1}$, we define $\mcal M(\mathtt h;y)$ to be the moduli-space of disks $u:D^2 \setminus\{p_1,\dots,p_k\} \to Y$ that are $J_g$-holomorphic for the given choice of $\mathtt h$-perturbation datum and asymptotic to some lift of the $y_i$ at each boundary puncture $p_i$.
    Then we have  $A_\infty$-operations
    \begin{align}\label{eq:ainfinityoperations}
        \begin{split}
        \mu_k:
        \hom((L_{k-1},\tilde J_{k-1}),(L_{k},\tilde J_{k}))
        \otimes\dots\otimes
        \hom&((L_{0},\tilde J_{0}),(L_{1},\tilde J_{1}))
        \to 
        \hom((L_{0},\tilde J_{0}),(L_{k},\tilde J_{k}))\\
        \vspan{\mu_k(y_0\otimes\dots\otimes y_{k-1}),y_k} & = \sum_{\mathtt h \in H^{k+1}} \frac{\# \mcal M(\mathtt h;y)}{|H|^{k+2}}
        \end{split}
    \end{align}
    The proof that these operations satisfy the $A_\infty$-relations is the usual one.

    \begin{rmk}
        From the above expression one sees  that $|H|$ should be invertible in the coefficient ring used to define the Fukaya category.
    \end{rmk}
\end{itemize}
Note that if $\tilde J_{i}$ is invariant under $H$ then there is an action of $H$ on $\mcal M(\mathtt h;y)$ via $h\cdot u = h\circ u$ (which is still holomorphic for the original choice of $\mathtt h$-perturbation datum due to the invariance of $\tilde J_{i}$).
In the particular case where $\tilde J_i = \pi^* J_i$ is an almost complex structure pulled back by $\pi$, we get an isomorphism $\mcal M(\mathtt h;y)/H \cong \mcal M_X(y)$ for $\mcal M_X(y)$  the moduli-space of disks used in the definition of $\mcal Fuk(X)$.
Hence $\# \mcal M_X(y) = \frac{\mcal M(\mathtt h;y)}{|H|}$, so it follows from \Cref{eq:ainfinityoperations}  that we have an embedding 
\begin{align}\label{eq:embeddingFuktoFukmulti}
    \begin{split}
            \iota: \mcal Fuk(X) & \into \mcal Fuk(X;Y)\\
            (L,J_L) &\mapsto (L,\pi^* J_L).
    \end{split}
\end{align}
More generally, we have an action of $H$ on $\cup_{\mathtt h\in H^{k+1}} \mcal M(\mathtt h;y)$ that, for any $h \in H$, sends a disk $u \in \mcal M(\mathtt h;y)$ to the disc $h\circ u \in \mcal M(h'\cdot \mathtt h;y)$, where $h\cdot \mathtt h := (h\cdot h_0,\dots, h\cdot h_k)$.
We think of the quotient moduli-space
$$
\frac{\bigcup_{\mathtt h\in H^{k+1}} \mcal M(\mathtt h;y)}{H }\cong \bigcup_{\mathtt h\in H^{k+1}/H} \mcal M(\mathtt h;y) 
$$
as a moduli space of multi-perturbed holomorphic disks in $X$, in the sense of \cite[Section 5]{salamon1999lectures}.
An $H$-orbit of a disk $u$ in $Y$ corresponds to the projected disk $\pi \circ u$ inside $X$.

Now we apply the above construction to $\pi:Y \x \C \to  X \x \C$.
We restrict the objects $\mcal L(X\x \C)$ to be  Lagrangian cobordisms in $X \x \C$ equipped with all the extra data that we would like the corresponding cobordism group to have (for us, they come equipped with $G$-brane structures).
This gives a Fukaya category $\mcal Fuk_{cob}(X \x \C;Y \x \C)$ whose objects are by definition Lagrangian cobordisms in $X\x \C$ (carrying a $G$-brane structure) that are tautologically unobstructed when pulled back under $\pi$. 
The machinery of  \cite{biran2014lagrangian,biran2021lagrangian} shows that these objects induce iterated cone decompositions in $\mcal Fuk(X;Y)$.
Using that the functor $\iota$ in \Cref{eq:embeddingFuktoFukmulti} is an embedding, it follows that cobordisms in $\mcal Fuk(X \x \C; Y \x \C)$ with ends $\iota(L_0),\dots,\iota(L_k)$ induce cone decompositions between the $L_i$ in $\mcal Fuk(X)$.
In particular, we have the following:
\begin{prop}\label{prop:conesfromTOpullbacks}
    Let $\pi: Y \to X$ be a finite covering, with both $Y$ and $X$ symplectic manifolds and $\pi^*\omega_X = \omega_Y$.
    Let $\Cob(X)$ be a cobordism group whose generators are weakly exact Lagrangians in $X$, and whose relations come from cobordisms $V \subset X \x \C$ such that $(\pi\x\id_\C)\inv(V) \subset Y \x \C$ is tautologically unobstructed.
    Then there is a well-defined map
    $$
    \Cob(X) \to K_0(\mcal Fuk(X)).
    $$
\end{prop}

Specializing to the case $X = \mcal K$ and $Y = T^4$ gives the result claimed in \Cref{sec:isoK}.
     \section{Proof of \Cref{prop:symplecticRoitman}}\label{ap:Roitmantheorem}
In this appendix we prove \Cref{prop:symplecticRoitman} by using the arguments in \cite[Section 7]{sheridan2020rational} and applying some elementary linear algebra.
It is worth pointing out that, while \cite[Lemma 7.9]{sheridan2020rational} should be thought of as the symplectic analog of Mumford's theorem \cite{mumford1969rational}, our extension is the symplectic analog of a more general theorem due to Roitman \cite[Theorem 5]{roitman1971gamma}.
In fact, the key result from linear algebra that we use was already encountered and proved by Roitman in his original paper.

Let $L \subset X$ be a Lagrangian submanifold.
Weinstein's tubular neighborhood theorem gives an identification between deformations of $L$ and closed $1$-forms on $L$ that are close to $0$. 
Since exact $1$-forms define Hamiltonian isotopic Lagrangians, we obtain an identification between nearby Lagrangians to $L$ modulo Hamiltonian isotopy and a neighborhood of $0$ in the first deRham cohomology group $H^1_{dR}(L) \cong H^1(L;\R)$.
More precisely, recall that given a Lagrangian isotopy $\mcal L = \{L_t\}_{t \in [0,1]}$, there is an associated {\it flux map} $Flux\, \mcal L \in H^1(L;\R) \cong \Hom(\pi_1(L),\R)$ defined via
\begin{align*}
    \begin{split}
        Flux\, \mcal L: \pi_1(L) &\to \R \\
        [\gamma] &\mapsto \int_{\Gamma} \omega
    \end{split}
\end{align*}
where $\Gamma := \cup_t \gamma_t$ is the cylinder obtained by parallel transporting $\gamma = \gamma_0$ along the isotopy $\mcal L$ (see \cite[Section 6]{solomon2013calabi} for more details).
Since the flux is invariant under homotopies of $\mcal L$ with fixed end-points and it vanishes if (and only if for nearby Lagrangians)  $L_0$ and $L_1$ are Hamiltonian isotopic,\footnote{Similar statements for the flux of a path of symplectomorphisms can be found in \cite[Section 10.2]{mcduff2017introduction}, and were adapted to this setting in \cite{solomon2013calabi}.} the mapping
$$
L' \mapsto Flux\, \mcal L_{L'}
$$
sending a Lagrangian $L'$ close to $L$  to the flux of an isotopy $\mcal L_{L'} = \{L_t\}_{t\in[0,1]}$ with $L_0 = L$ and $L_1 = L'$ provides an explicit indentification between Lagrangians close to $L$ modulo Hamiltonian isotopy and a neighborhood of the origin $\mcal U_L \subset H^1(L;\R)$ (note that for sufficiently close Lagrangians, the isotopy $\mcal L_{L'}$ is well-defined up to homotopy with fixed end-points). 
\begin{defn}
    Given a class $\alpha \in \mcal U_L$, we denote by $L(\alpha)$ the corresponding deformation of $L$.
\end{defn}
Note the Lagrangian $L(\alpha)$ can be explicitly represented (up to Hamiltonian isotopy) as the graph of any $1$-form representing the cohomology class $\alpha$.

Let $\mathbb L := (L_1,\dots L_k)$ be a tuple of Lagrangians and let $\mathbb U_{\mathbb L} := \prod_i \mcal U_{L_i}$.
It follows from the invariance of the Lagrangian cobordism class under Hamiltonian isotopy that there is a well-defined map 
\begin{align*}
    \begin{split}
        f_{\mathbb L}:\mathbb U_{\mathbb L} & \to \Cob(X)\\
        \alpha & \mapsto \sum_i L(\alpha_i).
    \end{split}
\end{align*}

One can easily incorporate local systems into the picture so that if  $\mathbb L = (L_1,\dots L_k)$ is a tuple of Lagrangians as above and $\mathbb U_{\mathbb L} := \prod_i (\mcal U_{L_i} \x H^1(L_i;G))$, there is a well-defined map 
\begin{align*}
    \begin{split}
        f_{\mathbb L}:\mcal U_{\mathbb L} & \to \Cob(X)\\
        (\alpha_i,\eta_i) & \mapsto \sum_i (L(\alpha_i),\eta_i).
    \end{split}
\end{align*}
Here, the notation $(L(\alpha_i),\eta_i)$ means the Lagrangian $L(\alpha_i)$ equipped with the local system induced from $\eta_i:\pi_1(L_i) \to G$ via the isotopy.
Writing $\dim \mathbb U_{\mathbb L} := \dim (\prod_i \mcal U_{L_i})$, we present the following notion of dimensionality of Lagrangian cobordism groups:
\begin{defn}\label{def:infinitedimensionality}
    We say the Lagrangian cobordism group $\Cob(X)$ is {\it $d$-dimensional} if $d$ is the minimum $N \in \N$ such that there exists a sequence $(\mathbb L_i,\mathbb U_{\mathbb L_i})_{i\in \N}$ of Lagrangian tuples $\mathbb L_i = (L^{(i)}_1,\dots, L^{(i)}_{n_i})$ and (not-necessarily open) subsets $\mathbb U_{\mathbb L_i} \subset \prod_j (H^1(L^{(i)}_j;\R) \x H^1(L^{(i)}_j;G))$ containing $0$ and of dimension at most $N$ such that $\cup_i \im(f_{\mathbb{L}_i}) = \Cob(X)$.
\end{defn}
\begin{rmk}
    The above notion of dimensionality of cobordism groups is similar, but not equal, to that in \cite{sheridan2020rational}---even without the presence of local systems.
    The key difference is that we allow arbitrary sets $\mathcal U_{L_i}$ containing $0$, whereas Sheridan-Smith require them to be {\it open}.
    For instance, for $X = T^2$ it is known by Haug \cite{haug2015lagrangian} (see also our \Cref{sec:lagcobT2}) that the Lagrangian cobordism group of a $2$-torus is $\Cob(T^2) \cong \Z^2 \oplus S^1$.
    This group is $1$-dimensional under \Cref{def:infinitedimensionality}, but it is not finite-dimensional if one requires the subsets $\mathbb U_{\mathbb L_i}$ to be open neighborhoods of $0$ and their dimensions to be bounded.
\end{rmk}

Sheridan-Smith prove an infinite-dimensionality result for cobordism groups as follows. 
First, they prove:

\begin{lem}(\cite[Lemma 7.4]{sheridan2020rational})
    Given $\mathbb L$ and $\mathbb L'$ finite tuples of Lagrangians, the subset
    $$
    \mcal Z_{\mathbb L,\mathbb L'} =
    \{(\alpha,\alpha') \in \mathbb U_{\mathbb L} \x \mathbb U_{\mathbb L'} | f_{\mathbb L}(\alpha) = f_{\mathbb L'} (\alpha')\}
    $$
    is a countable union of (open subsets of)  affine subspaces of $\mathbb U_{\mathbb L} \x \mathbb U_{\mathbb L'}$.
\end{lem}

Consider now the $n$-form $\Omega\in  \wedge^n T^* \mathbb U_{\mathbb L}$ defined by
\begin{equation}\label{eq:Omega}
    \Omega(\alpha^1,\dots,\alpha^n) = \sum_i \int_{L_i} \alpha^1_i\cup\dots\cup\alpha^n_i
\end{equation}
for $\alpha^j=(\alpha^j_1,\dots,\alpha^j_k) \in \mcal U_{L_1}\x\dots\x\mcal U_{L_k} \simeq T\mathbb U_{\mathbb L}$.
It can be naturally extended to an $n$-form on $\mathbb U_{\mathbb L} \x \mathbb U_{\mathbb L'}$ via the identification $\mathbb U_{\mathbb L} \x \mathbb U_{\mathbb L'} \simeq \mathbb U_{\mathbb L \cup \mathbb L'}$. 

\begin{lem}\label{lem:Omegavanishes}(\cite[Lemma 7.5]{sheridan2020rational})
    The $n$-form $\Omega$ vanishes on $\mcal Z_{\mathbb L,\mathbb L'}$.
\end{lem}

The key property is that  $\Omega$ is non-degenerate for $n = 2$.
Hence we have a symplectic form on $\mathbb U_{\mathbb L} \x \mathbb U_{\mathbb L'}$ and an isotropic subspace $\mcal Z_{\mathbb L, \mathbb L'} \subset \mathbb U_{\mathbb L} \x \mathbb U_{\mathbb L'}$, which gives a bound $\dim \mcal Z_{\mathbb L, \mathbb L'} \leq \frac{1}{2}\dim (\mathbb U_{\mathbb L} \x \mathbb U_{\mathbb L'})$.
Using this, Sheridan-Smith state the infinite-dimensionality of $\Cob(X)$:

\begin{lem}(\cite[Lemma 7.9]{sheridan2020rational})\label{lem:sslemma7.9}
    Suppose $X$ is a symplectic four manifold containing a Lagrangian of genus at least one.
    Then, for any countable family $\{\mathbb L_i\}$ of finite tuples $\mathbb L_i = (L^i_1,\dots,L^i_{n_i})$ such that $\sup_i \dim \mathbb U_{\mathbb L_i} < \infty$, the cobordism group $\Cob (X)$ is not covered by the images of the maps $f_{\mathbb L_i}$. 
\end{lem}

The dimension bound $\dim \mcal Z_{\mathbb L \mathbb L'} \leq \frac{1}{2}\dim (\mathbb U_{\mathbb L} \x \mathbb U_{\mathbb L'})$ is the key to make their proof work.
This bound comes from having a symplectic form, which only works in the case $n = 2$. 
For arbitrary $n$, recall from \Cref{eq:Omega} that we have  an $n$-form $\Omega$ of the form $\Omega = \sum_i pr_i^*\omega_i$, where $pr_i:\mathbb U_{\mathbb L} \to \mcal U_{L_i}$ is the projection and $\omega_i:H^1(L_i;\R)^{\otimes n} \to \R$ is given by cup-product and integration.
It turns out that in a situation like this one still has a useful bound:

\begin{lem}\label{lem:roitmanbound}(\cite[Lemma 9]{roitman1971gamma})
    Let $V = \oplus_{j=1}^m V_j$ be a graded vector space and $0 \neq \omega_j \in \wedge^q V_j^*$ non-zero $q$-forms on $V_j$, $q \geq 2$.
    Denote by $pr_j:V \to V_j$  the natural projection.
    If $W \subset V$ is an isotropic subspace for the $q$-form $\Omega = \sum_j pr^*_j \omega_j$, then $\dim W \leq \dim V - m$.
\end{lem}

With this result, the proof of  \cite[Lemma 7.9]{sheridan2020rational} translates almost directly:

\begin{proof}[Proof of \Cref{prop:symplecticRoitman}]
    Let $N$ be such that $\dim \mathbb U_{\mathbb L_i} < N$ for all $i$, and consider the family $\mathbb L^{(N)} = (T,\dots,T)$ consisting of $N$ copies of the Lagrangian torus $T$.
    We choose $\mcal U_T$ to be an open neighborhood of $0$ and as before write $\mathbb U_{\mathbb L^{(N)}} = \prod_{k=1}^N \mcal U_T$.
    We will prove that the images of the maps $f_{\mathbb L_i}$ can't cover $\im(f_{\mathbb L^{(N)}}) \subset \Cob(X)$, so that in particular they can't cover $\Cob(X)$.
    Note that the images of the maps $f_{\mathbb L_i}$ cover $\im(f_{\mathbb L^{(N)}})$ if and only if the projections  $$
     \mcal Z_{\mathbb L_i, \mathbb L^{(N)}} 
    \subset 
    \mathbb U_{\mathbb L_i} \x \mathbb U_{\mathbb L^{(N)}} 
    \to 
    \mathbb U_{\mathbb L^{(N)}}
    $$
    cover all of $\mathbb U_{\mathbb L^{(N)}}$.
    We will show that $\dim \mcal Z_{\mathbb L_i, \mathbb L^{(N)}} < \dim \mathbb U_{\mathbb L^{(N)}}$ for all $i$, and since countably many sets of strictly lower dimension cannot surject onto $\mathbb U_{\mathbb L^{(N)}}$ the result will follow.

    Consider the decomposition $\mathbb U_{\mathbb L_i} \oplus \mathbb U_{\mathbb L^{(N)}} = \oplus_j V_j$ given by  $V_1 = \mathbb U_{\mathbb L_i} \oplus \mathcal U_T$ and $V_j = \mathcal U_T$ for $j = 2,\dots,N$.
    Given a Lagrangian $L$ we denote by $\omega_L \in \wedge^n H^1(L;\R)^*$ the $n$-form on $H^1(L;\R)$ given by cup-product and integration; they naturally restrict to $n$-forms on $\mcal U_L \subset H^1(L;\R)$.
    We then define $n$-forms $\omega_j \in \wedge^n V_j^*$ by $\omega_1 = \sum_k \omega_{L_k^{(i)}} + \omega_T$ and $\omega_j =  \omega_T$ for $j = 2,\dots,N$.
    Note that $\omega_j \neq 0$ for all $j$ since $T$ is a torus and thus $\cup: H^1(T;\R)^{\otimes n} \to H^n(T;\R)$ is a surjection.
    The form
    $$
    \Omega = \sum_{j=1}^N pr_j^* \omega_j,\quad pr_j : V  = \mathbb U_{\mathbb L_i} \oplus \mathbb U_{\mathbb L^{(N)}} = \oplus_j V_j \to V_j
    $$
    vanishes on $\mcal Z_{\mathbb L_i, \mathbb L^{(N)}}$ by \Cref{lem:Omegavanishes}, hence we have:
    \begin{align*}\label{eq:bound2}
        \dim \mcal Z_{\mathbb L_i, \mathbb L^{(N)}} & \leq \dim(\mathbb U_{\mathbb L_i} \x \mathbb U_{\mathbb L^{(N)}}) - N \quad & \text{(\Cref{lem:roitmanbound})}\\
&< N + \dim \mathbb U_{\mathbb L^{(N)}} - N \quad & (\dim \mathbb U_{\mathbb L_i} < N)\\
        &=\dim \mathbb U_{\mathbb L^{(N)}}.
    \end{align*}
    In particular the projection $\cup_i \mcal Z_{\mathbb L_i, \mathbb L^{(N)}} \subset \cup_i \mathbb U_{\mathbb L_i} \x \mathbb U_{\mathbb L^{(N)}} 
    \to 
    \mathbb U_{\mathbb L^{(N)}}$ cannot be surjective.
\end{proof}

\renewbibmacro{in:}{}
\def\bibrangedash{ -- }
\printbibliography

\end{document}

%% file: preamble.tex
\usepackage{todonotes}
\usepackage{caption}
\usepackage{subcaption}

\usepackage{tikz-cd}
\usepackage{enumerate}
\usepackage{mathtools}
\usepackage{verbatim}
\usepackage{amssymb}
\usepackage[citestyle=alphabetic,bibstyle=alphabetic,backend=bibtex, maxbibnames=10,minbibnames=5]{biblatex}

\def\P{\mathbb{P}}
\def\Oo{\mathcal{O}}

\DeclareMathOperator*{\GL}{\mathrm{GL}}
\DeclareMathOperator{\Aff}{Aff}

\newtheorem{thmdef}[thm]{Theorem/Definition}

\newtheorem{maintheorem}{Theorem}
\Crefname{maintheorem}{Theorem}{Theorems}

\Crefname{maincorollary}{Corollary}{Corollaries}

\Crefname{mainexample}{Example}{Examples}

\def\x{\times} 
\newcommand{\xra}[1]{\xrightarrow{#1}}
\def\inv{^{-1}}
\DeclareMathOperator{\Alb}{\mathrm{Alb}}
\DeclareMathOperator{\Cob}{\mathrm{Cob}}
\DeclareMathOperator{\CH}{\mathrm{CH}}
\DeclareMathOperator{\alb}{\mathrm{alb}}
\newcommand{\into}{\hookrightarrow}
\newcommand{\set}[1]{\left\lbrace{#1}\right\rbrace}
\newcommand{\vspan}[1]{\langle #1 \rangle}
\newcommand{\mcal}{\mathcal}

\newcommand{\restr}[1]{|_{#1}}

\newcommand{\st}{\,\vert\,}
\newcommand{\wdg}{\wedge}
\usepackage{dsfont}
